\documentclass[12pt,reqno]{amsart}
\usepackage{amsmath, amsthm, amssymb, stmaryrd}

\advance \topmargin by -\headheight
\advance \topmargin by -\headsep
     
\evensidemargin \oddsidemargin
\newtheorem{theorem}{Theorem}[section]
\newtheorem{lemma}[theorem]{Lemma}
\newtheorem{corollary}[theorem]{Corollary}

\theoremstyle{definition}

\theoremstyle{remark}

\numberwithin{equation}{section}

\newcommand{\mmod}[1]{\,\,(\text{\rm mod}\,\,#1)}
\newcommand{\nmod}[1]{\,\,\text{\rm mod}\,\,#1}

\def\bfa{{\mathbf a}}

\def\bfd{{\mathbf d}}  

\def\bfh{{\mathbf h}}

\def\bfr{{\mathbf r}}
\def\bft{{\mathbf t}}
\def\bfu{{\mathbf u}}
\def\bfv{{\mathbf v}}
\def\bfw{{\mathbf w}}
\def\bfx{{\mathbf x}}
\def\bfy{{\mathbf y}}
\def\bfz{{\mathbf z}}

\def\calB{{\mathcal B}}

 \def\Ktil{{\widetilde K}}

\def\calN{{\mathcal N}}

\def\calP{{\mathcal P}}

\def\calZ{{\mathcal Z}}

\def\Gtil{\widetilde G}\def\Ktil{\widetilde K}

\def\dbC{{\mathbb C}}\def\dbD{{\mathbb D}}
\def\dbF{{\mathbb F}}\def\dbK{{\mathbb K}}
\def\dbN{{\mathbb N}} \def\dbO{{\mathbb O}} 
\def\dbR{{\mathbb R}}\def\dbT{{\mathbb T}}
\def\dbZ{{\mathbb Z}}\def\dbQ{{\mathbb Q}}

\def\gra{{\mathfrak a}}\def\bfgra{{\boldsymbol \gra}}
\def\grb{{\mathfrak b}}\def\grB{{\mathfrak B}}\def\bfgrb{{\boldsymbol \grb}}
\def\grc{{\mathfrak c}}\def\bfgrc{{\boldsymbol \grc}}
\def\grd{{\mathfrak d}}
\def\grf{{\mathfrak f}}
\def\grg{{\mathfrak g}}
\def\grh{{\mathfrak h}}\def\grH{{\mathfrak H}}
\def\grJ{{\mathfrak J}}

\def\grm{{\mathfrak m}}\def\grM{{\mathfrak M}}
\def\grN{{\mathfrak N}}\def\grn{{\mathfrak n}}
\def\grO{{\mathfrak O}}
\def\grS{{\mathfrak S}}
\def\grB{{\mathfrak B}}
\def\grp{{\mathfrak p}}

\def\alp{{\alpha}} \def\bfalp{{\boldsymbol \alpha}}
\def\bet{{\beta}}  \def\bfbet{{\boldsymbol \beta}}
\def\gam{{\gamma}} \def\Gam{{\Gamma}}
\def\bfgam{{\boldsymbol \gam}}
\def\del{{\delta}} \def\Del{{\Delta}} 
\def\zet{{\zeta}}  
 
\def\tet{{\theta}} \def\bftet{{\boldsymbol \theta}} 
\def\Tet{{\Theta}} \def\bfTet{{\boldsymbol \Tet}}
\def\kap{{\kappa}}
\def\lam{{\lambda}} \def\Lam{{\Lambda}}

\def\Rho{{\mathrm P}}
\def\sig{{\sigma}}  

\def\Ups{{\Upsilon}} \def\bfUps{{\boldsymbol \Ups}}
\def\bfPhi{{\boldsymbol \Phi}} \def\bfvarphi{{\boldsymbol \varphi}}
\def\bfchi{{\boldsymbol \chi}}
\def\bfpsi{{\boldsymbol \psi}} \def\bfPsi{{\boldsymbol \Psi}}
\def\ome{{\omega}} \def\Ome{{\Omega}}

\def\eps{\varepsilon}
\def\implies{\Rightarrow}

\def\le{\leqslant} \def\ge{\geqslant}

\def\d{{\,{\rm d}}}

\newcommand{\rrbracketsub}[1]{\rrbracket_{\textstyle{_#1}}}

\begin{document}
\title[Nested efficient congruencing]{Nested efficient congruencing and relatives of 
Vinogradov's mean value theorem}
\author[Trevor D. Wooley]{Trevor D. Wooley}
\address{School of Mathematics, University of Bristol, University Walk, Clifton, Bristol BS8 
1TW, United Kingdom}
\email{matdw@bristol.ac.uk}
\date{}
\begin{abstract} We apply a nested variant of multigrade efficient congruencing to 
estimate mean values related to that of Vinogradov. We show that when 
$\varphi_j\in \dbZ[t]$ $(1\le j\le k)$ is a system of polynomials with non-vanishing 
Wronskian, and $s\le k(k+1)/2$, then for all complex sequences $(\gra_n)$, and for each 
$\eps>0$, one has
$$\int_{[0,1)^k}\Bigl| \sum_{|n|\le X}\gra_n e(\alp_1\varphi_1(n)+\ldots 
+\alp_k\varphi_k(n))\Bigr|^{2s}\d\bfalp \ll X^\eps \biggl( \sum_{|n|\le X}
|\gra_n|^2\biggr)^s.$$
As a special case of this result, we confirm the main conjecture in Vinogradov's mean 
value theorem for all exponents $k$, recovering the recent conclusions of the author (for 
$k=3$) and Bourgain, Demeter and Guth (for $k\ge 4$). In contrast with the 
$l^2$-decoupling method of the latter authors, we make no use of multilinear Kakeya
estimates, and thus our methods are of sufficient flexibility to be applicable in algebraic 
number fields, and in function fields. We outline such extensions.\end{abstract}
\maketitle

\section{Introduction} This memoir is devoted to a general class of exponential sums and 
their mean values. We demonstrate how the efficient congruencing method, developed by 
the author and others starting in late 2010 (see \cite{Woo2012}) in the context of 
Vinogradov systems, may be adapted to handle relatives of Vinogradov's mean value 
theorem. Indeed, this {\it nested efficient congruencing method} succeeds in establishing 
the main conjecture in Vinogradov's mean value theorem for all degrees, a conclusion 
obtained first by the author in the cubic case (see \cite{Woo2016} and arXiv:1401.3150) 
and subsequently for degrees exceeding three by Bourgain et al.~(see \cite{BDG2016} 
and arXiv:1512.01565v3). In contrast with the $l^2$-decoupling method of the latter 
authors, nested efficient congruencing makes no use of multilinear Kakeya estimates or other 
tools apparently intertwined with harmonic analysis in the real setting, and thus our methods 
are of sufficient flexibility to be applicable in algebraic number fields, and in function fields. 
We outline such extensions. Further discussion requires that we introduce some notation in 
order that we pass from descriptive statements to concrete technicalities.\par

Given $k\in \dbN$, we consider polynomials $\varphi_j\in \dbZ[t]$ $(1\le j\le k)$ and the 
associated Wronskian
\begin{equation}\label{1.1}
W(t;\bfvarphi)=\det \left( \varphi_j^{(i)}(t)\right)_{1\le i,j\le k}.
\end{equation}
Here, following the usual convention, we write $\varphi_j^{(r)}(t)$ for the $r$-th derivative 
${\rm d}^r\varphi_j(t)/{\rm d}t^r$. A measure of the independence of this system of 
polynomials $\bfvarphi$ is given by whether or not $W(t;\bfvarphi)=0$. Our first conclusion 
supplies an estimate of Strichartz type. As is usual, we write $e(z)$ for $e^{2\pi iz}$.

\begin{theorem}\label{theorem1.1} Suppose that $\varphi_j\in \dbZ[t]$ $(1\le j\le k)$ is a 
system of polynomials with $W(t;\bfvarphi)\ne 0$. Let $s$ be a positive real number with 
$s\le k(k+1)/2$. Also, suppose that $(\gra_n)_{n\in \dbZ}$ is a sequence of complex 
numbers. Then for each $\eps>0$, one has
\begin{equation}\label{1.2}
\int_{[0,1)^k}\biggl| \sum_{|n|\le X}\gra_ne(\alp_1\varphi_1(n)+\ldots 
+\alp_k\varphi_k(n))\biggl|^{2s}\d\bfalp \ll X^\eps \biggl( \sum_{|n|\le X}
|\gra_n|^2\biggr)^s.
\end{equation}
In particular, under these conditions, one has
\begin{equation}\label{1.3}
\int_{[0,1)^k}\biggl| \sum_{1\le n\le X}e(\alp_1\varphi_1(n)+\ldots +\alp_k\varphi_k(n))
\biggl|^{2s}\d\bfalp \ll X^{s+\eps}.
\end{equation}
\end{theorem}

We emphasise that here and elsewhere, unless indicated otherwise, the implicit constants in 
Vinogradov's notation $\ll$ and $\gg$ may depend on $\eps$, $s$, $k$ and the coefficients 
of $\bfvarphi$. It follows via orthogonality that when $s$ is a positive integer, then the mean 
value on the left hand side of (\ref{1.3}) counts the number of integral solutions of the 
system of equations
\begin{equation}\label{1.4}
\sum_{i=1}^s\left( \varphi_j(x_i)-\varphi_j(y_i)\right)=0\quad (1\le j\le k),
\end{equation}
with $1\le x_i,y_i\le X$ $(1\le i\le s)$. Variants of our methods would yield analogues of 
Theorem \ref{theorem1.1} in which the polynomials $\varphi_j\in \dbZ[x]$ are replaced 
by rational functions lying in $\dbQ(t)$, or suitably smooth real or $p$-adic valued 
functions, with equations replaced by inequalities as appropriate. Likewise, the summands 
$x$ in (\ref{1.3}) could be replaced by discretely spaced sets of real or $p$-adic numbers. 
We have chosen to provide the most accessible exposition here rather than explore the most 
general and least transparent framework available to our methods.\par

By putting $\varphi_j(t)=t^{d_j}$ $(1\le j\le k)$, we obtain a conclusion on Vinogradov 
systems in which missing slices are permitted.

\begin{corollary}\label{corollary1.2} Let $d_1,\ldots ,d_k$ be distinct positive integers, and 
let $s$ be a real number with $0<s\le k(k+1)/2$. Then for each $\eps>0$, one has
\begin{equation}\label{1.5}
\int_{[0,1)^k}\biggl| \sum_{1\le x\le X}e(\alp_1x^{d_1}+\ldots +\alp_kx^{d_k})
\biggr|^{2s}\d\bfalp \ll X^{s+\eps}.
\end{equation}
\end{corollary}

Again, when $s$ is a positive integer, it follows via orthogonality that the mean value on 
the left hand side of (\ref{1.5}) counts the number of integral solutions of the system
$$\sum_{i=1}^s(x_i^{d_j}-y_i^{d_j})=0\quad (1\le j\le k),$$
with $1\le x_i,y_i\le X$ $(1\le i\le s)$. Aside from recent progress on the Vinogradov 
system with $d_j=j$ $(1\le j\le k)$, previous published progress on such systems has 
fallen far short of achieving the range for $s$ delivered by Corollary \ref{corollary1.2}. The 
estimate (\ref{1.5}) was established for $s\le k+1$ in \cite[Theorem 1]{Woo1993} via 
polynomial identities and divisor sum estimates, and indeed such ideas were extended to 
the setting of Theorem \ref{theorem1.1} for the same range of $s$ in 
\cite[Theorem 1]{PW2002}. We note, however, that extensions to this range have been 
announced previously. Thus, in 2014 the author announced the proof\footnote{See the 
talk https://www.youtube.com/watch?v=Q5gcLVYqEks from the ELEFANT Workshop, Bonn, 
July 2014.} (via multigrade efficient congruencing) of the upper bound (\ref{1.5}) in the 
range $s\le k(k+1)/2-k/3+o(k)$. In addition, Bourgain \cite[equation (6.6)]{Bou2017} has 
implicitly announced a result tantamount to the conclusion of Corollary \ref{corollary1.2} in 
the special case in which $(d_1,\ldots ,d_k)=(1,2,\ldots ,k-1,d)$, with $d\ge k$. Although 
one of our purposes in this memoir is to provide a complete published proof for these 
earlier assertions, we go considerably beyond this earlier work. We remark further that 
when the degrees $d_j$ are suitably large, the conclusion of Corollary \ref{corollary1.2} 
follows in a potentially wider range via enhancements of the determinant method of 
Heath-Brown. Thus, as a consequence of work of the author joint with Salberger (see 
\cite[Theorems 1.3 and 5.2]{SW2010}), one has
$$\int_{[0,1)^k}\biggl| \sum_{1\le x\le X}e(\alp_1x^{d_1}+\ldots 
+\alp_kx^{d_k})\biggr|^{2s}\d\bfalp =s!X^s+O(X^{s-1/2}),$$
provided only that the exponents $d_j$ are distinct and satisfy the condition
$$d_1\cdots d_k\ge (2s)^{4s}.$$

\par The special case of Corollary \ref{corollary1.2} in which $(d_1,\ldots ,d_k)=
(1,2,\ldots ,k)$ corresponds to the Vinogradov system
\begin{equation}\label{1.6}
\sum_{i=1}^s(x_i^j-y_i^j)=0\quad (1\le j\le k).
\end{equation}
Adopting standard notation, we write $J_{s,k}(X)$ for the number of integral solutions of 
the system (\ref{1.6}) with $1\le x_i,y_i\le X$ $(1\le i\le s)$. More generally, when $s$ is 
not necessarily an integer, we put
\begin{equation}\label{1.7}
J_{s,k}(X)=\int_{[0,1)^k}\biggl| \sum_{1\le x\le X}e(\alp_1x+\ldots +\alp_kx^k)
\biggr|^{2s}\d\bfalp .
\end{equation}
The {\it main conjecture} in Vinogradov's mean value theorem asserts that for each 
$\eps>0$, one has
\begin{equation}\label{1.8}
J_{s,k}(X)\ll_{\eps,s,k} X^{s+\eps}+X^{2s-k(k+1)/2}.
\end{equation}
Here, we have deviated very slightly from the formulation of the main conjecture asserted 
in earlier work (see for example \cite[equation (1.4)]{Woo2012}) by omitting the term 
$\eps$ from the exponent in the second summand on the right hand side of (\ref{1.8}). 
This merely recognises the observation, well-known for more than half a century, that the 
validity of the estimate (\ref{1.8}) for $s=k(k+1)/2$ implies its validity for all positive real 
numbers $s$. Such is evident from an application of H\"older's inequality when 
$s<k(k+1)/2$, and is immediate from an application of the circle method for $s>k(k+1)/2$. 
A transparent consequence of Corollary \ref{corollary1.2} yields the main conjecture in full.

\begin{corollary}\label{corollary1.3}
The main conjecture holds in Vinogradov's mean value theorem. Thus, for each $\eps>0$, 
one has $J_{s,k}(X)\ll X^{s+\eps}+X^{2s-k(k+1)/2}$. Indeed, when $k\ge 3$ and 
$s>k(k+1)/2$, one has the asymptotic formula
$$J_{s,k}(X)\sim C_{s,k}X^{2s-k(k+1)/2},$$
where $C_{s,k}$ is a positive number depending at most on $s$ and $k$.
\end{corollary}

Theorem \ref{theorem1.1} also delivers the expected Strichartz inequality established for 
$s\ge k(k+1)$ in \cite{Woo2017a}, and subsequently in full in \cite{BDG2016}.

\begin{corollary}\label{corollary1.4}
Suppose that $k\in \dbN$, that $s$ is a positive number, and $(\gra_n)_{n\in \dbZ}$ is a 
complex sequence. Then, for each positive number $\eps$, one has
$$\int_{[0,1)^k}\biggl| \sum_{|n|\le X}\gra_ne(n\alp_1+\ldots +n^k\alp_k)\biggr|^{2s}
\d\bfalp \ll X^\eps (1+X^{s-k(k+1)/2})\biggl( \sum_{|n|\le X}|\gra_n|^2\biggr)^s.$$
\end{corollary} 

As we have already noted, the main conjecture (\ref{1.8}) follows for all $s$ from the 
special case in which $s=k(k+1)/2$. The validity of (\ref{1.8}) in the case $k=1$ is of 
course trivial, and when $k=2$ the asymptotic formula
$$J_{3,2}(X)\sim \frac{18}{\pi^2}X^3\log X$$
follows via classical methods (see Rogovskaya \cite{Rog1986}, and Blomer and 
Br\"udern \cite{BB2010} for sharp versions of this formula). When $k\ge 3$ it is expected 
that the upper bound (\ref{1.8}) should hold with $\eps=0$, though at present such is 
known only when $s\le k+1$ (see \cite[Theorem 1]{VW1997}) and $s>k(k+1)/2$ (see 
Corollary \ref{corollary1.3}). The main conjecture (\ref{1.8}) was established in full for 
$k=3$ in the author's previous work 
\cite[Theorem 1.1, Theorem 8.1 and its proof]{Woo2016} (see also \cite{HB2015} for an 
account with certain simplifications). When $k>3$, the multigrade efficient congruencing 
method established (\ref{1.8}) in the range 
$s\le k(k+1)/2-k/3+O(k^{2/3})$, missing the critical exponent $s=k(k+1)/2$ by roughly 
$k/3$ variables (see \cite[Theorem 1.3]{Woo2017}). This defect was later remedied in the 
work of Bourgain et al.~\cite{BDG2016} by means of their $l^2$-decoupling method (the 
reader might refer to \cite{Pie2017} for the status of developments at the end of 2016). The 
nested variant of the multigrade efficient congruencing method that we outline in \S2 now 
also remedies this defect. Work prior to 2010 preceding the efficient congruencing methods 
was, meanwhile, far weaker (see for example \cite{Hua1965, Vin1947, Woo1992}).\par

We have already expended considerable space on recording our main conclusions and 
describing previous results, without pausing to explain the importance of Vinogradov's mean 
value theorem and its relatives. The recent burst of activity surrounding efficient 
congruencing, $l^2$-decoupling, and Vinogradov's mean value theorem offers some 
justification for this concentration on mean value estimates rather than applications. This 
is an opportune moment, however, to highlight the central position of Vinogradov's mean 
value theorem across a large swath of analytic number theory. Current approaches to the 
asymptotic formula in Waring's problem, the sharpest available estimates for the zero-free 
region of the Riemann zeta function, and the investigation of equidistribution modulo $1$ of 
polynomial sequences, all depend for their success on estimates associated with 
Vinogradov's mean value theorem (see \cite{Bou2017, Woo2012b}, \cite{For2002}, 
\cite{Bak1986, Woo2016a}, respectively). We direct the reader to \cite{Woo2014a} for an 
overview of several other applications and an account of recent developments. In \S\S13 and 
14 we record some applications of Theorem \ref{theorem1.1} and its corollaries to Waring's 
problem and cognate applications.\par

Several commentators have described the work of Bourgain et al.~\cite{BDG2016} 
concerning Vinogradov's mean value theorem as inherently analytic in nature, contrasting it 
to earlier number-theoretic methods. A comparison of the efficient congruencing methods 
(see especially \cite{Woo2012, Woo2015, Woo2016, Woo2017} and the present paper) with 
the $l^2$-decoupling approach \cite{BDG2016}, however, shows the core of both methods 
to be strikingly similar. The former applies $p$-adic short intervals (which is to say, 
congruence class restrictions) to achieve a $p$-adic concentration argument via a 
multiscale iteration, whereas the latter applies real short intervals to achieve the same 
effect. Number theorists will have little difficulty in translating arguments over $\dbQ_p$ to 
analogous arguments over $\dbR$ (which is to say, over $\dbQ_\infty$), and vice versa. 
One of the important messages of the present memoir is that the efficient congruencing 
ideas are sufficiently flexible that they may be applied to estimate mean values associated 
with discrete sets of points in a wide variety of fields and their localisations. Examples of 
such situations include, but are not limited to:
\begin{enumerate}
\item[(i)] discrete sets of points in $\dbQ$ and its localisations $\dbQ_p$ and 
$\dbR=\dbQ_\infty$;
\item[(ii)] discrete sets of points in a number field $K$ and its localisations;
\item[(iii)] discrete sets of points in $\dbF_q(t)$ and its localisations;
\item[(iv)] discrete sets of points in a function field defined by a curve over a finite field, 
and its localisations.
\end{enumerate}

\par In order to illustrate that generalisations of our principal conclusions are easily 
accessible to our methods, in \S15 we establish an analogue of Theorem \ref{theorem1.1} 
and its corollaries for rings of integers in an arbitrary number field. We also consider function 
field analogues of our principal conclusions in their most basic form in \S17. In both 
instances, we exploit the relative simplicity of the nested efficient congruencing method as 
compared to the $l^2$-decoupling method of Bourgain et al.~\cite{BDG2016}. We are able, 
for example, to avoid any discussion of multilinear Kakeya estimates, the nature of which 
would be necessarily more mysterious (and presently unknown) in the setting of number 
fields and function fields. Our approach is also bilinear rather than multilinear, leading to a 
considerable streamlining of detail. Finally, one inductive aspect of our methods (concerning 
the number of equations or congruences in play) obviates the need for detailed knowledge 
of systems of congruences in many variables, as previously supplied by \cite{Woo1996}. 
Thus, the much simpler theory associated with a single congruence in one variable underpins 
the wider theory without further elaboration. We refer the reader to \S\S15 and 17 for 
details, and defer applications to a future paper. We finish by noting that a comprehensive 
analogue of Theorem \ref{theorem1.1} in the function field setting is the subject of 
forthcoming work of the author joint with Yu-Ru Liu.

\par This memoir is organised as follows. We provide a crude outline of the nested efficient 
congruencing method in \S2. Here, readers will find an outline of the ideas new to the 
nested variant of the multigrade efficient congruencing method, as well as an overview of 
the efficient congruencing strategy in the large. Sections 3 to 10 inclusive provide a 
detailed account of the proof of Theorem \ref{theorem3.1}, the basic nested inductive 
step, via the nested efficient congruencing method. The necessary infrastructure is 
introduced in \S3, with a discussion of the implications for translation-dilation invariant 
families in \S4. The nested structure is built inductively, and we introduce the foundation 
for this induction with the case of a single equation in \S5. The general situation is 
addressed in \S\S6-9. In \S6 we prepare our bilinear structure with some initial conditioning. 
Then we employ the approximate translation-dilation invariance in \S7 so as to generate 
strong congruence constraints efficiently. In the language of harmonic analysis, this 
constitutes a {\it multiscale} aspect to the method. These congruence constraints must be 
employed iteratively, as described in \S8, the analysis of which in \S9 prepares the ground 
for the proof of Theorem \ref{theorem3.1} in \S10. Then, in \S11, we derive some 
consequences of Theorem \ref{theorem3.1} for problems involving the number of solutions 
of congruences in short intervals. In \S12 we show how Theorem \ref{theorem3.1} may be 
employed to deliver Theorem \ref{theorem1.1} and its corollaries. A brief excursion in \S13 
explores the consequences of Corollary \ref{corollary1.3} for Tarry's problem. An account of 
the implications of Theorem \ref{theorem1.1} for relatives of Vinogradov's mean value 
theorem analogous to Hua's lemma is provided in \S14, and here we numerically refine some 
recent work of Bourgain \cite{Bou2017} concerning the asymptotic formula in Waring's 
problem, adding extra details to that exposition. In \S15 we indicate how to establish the 
main conjecture in Vinogradov's mean value theorem for number fields, and we provide 
some consequences of these conclusions for multivariable analogues of Vinogradov's mean 
value theorem in \S16. Finally, in \S17, we confirm the main conjecture in Vinogradov's mean 
value theorem for function fields in its most basic form.\par

A couple of notational conventions already deserve mention. Throughout, we make liberal 
use of vector notation in settings not always conventional in nature. Thus, for example, we 
write $1\le \bfx\le X$ to denote that every coordinate $x_i$ of $\bfx$ satisfies 
$1\le x_i\le X$, and $\bfx\equiv \xi\mmod{p^h}$ to denote that $x_i\equiv \xi\mmod{p^h}$ 
for all indices $i$. Also, we adopt the convention that when $F:[0,1)^n\rightarrow \dbC$ is 
integrable, then
\begin{equation}\label{1.9}
\oint F(\bfalp)\d\bfalp=\int_{[0,1)^n}F(\bfalp)\d\bfalp .
\end{equation}
Finally, given a situation in which the parameter $\eps$ has not already been fixed, in any 
statement involving the letter $\eps$ it is implicitly asserted that the statement holds for 
each $\eps>0$. In such circumstances, implicit constants in Vinogradov's notation $\ll$ and 
$\gg$ may depend on $\eps$.
\vskip.1cm

\noindent{\bf Acknowledgements:} The bulk of this work was completed in February 2016. 
A period of heavy administration at the University of Bristol slowed the final production of 
this memoir, but also permitted an evolution of ideas that has greatly simplified several 
aspects. The author is grateful to the Fields Institute in Toronto for excellent working 
conditions and support that made the completion of this work possible during the Thematic 
Program on Unlikely Intersections, Heights, and Efficient Congruencing in the first half of 
2017. Further work was supported by the National Science Foundation under Grant 
No.~DMS-1440140 while the author was in residence at the Mathematical Sciences 
Research Institute in Berkeley, California, during the Spring 2017 semester. The author's 
work was supported by a European Research Council Advanced Grant under the European 
Union's Horizon 2020 research and innovation programme via grant agreement No.~695223. 
Without the release time from the University of Bristol funded by the latter grant, it is 
difficult to envision that completion of this memoir would have been feasible, and the 
author wishes to express his gratitude to the ERC for this support.

\section{A crude outline of nested efficient congruencing} Granted latitude to be economical 
with technical details and mathematical rigour, we begin in this section by outlining in broad 
terms the key features of the method central to this paper. The starting point is the 
translation-dilation invariant (TDI) system of equations (\ref{1.6}). This TDI property is 
made evident by the observation that, whenever $\xi\in \dbZ$ and $q\in \dbN$, the pair 
$\bfx,\bfy$ satisfies (\ref{1.6}) if and only if this $2s$-tuple also satisfies the system of 
equations
$$\sum_{i=1}^s\left( (qx_i+\xi)^j-(qy_i+\xi)^j\right)=0\quad (1\le j\le k).$$
An application of the binomial theorem rapidly confirms such to be the case.\par

In the basic version of efficient congruencing (see \cite{Woo2012}), one relates the mean 
value $J_{s,k}(X)$ defined in (\ref{1.7}) to associated mean values equipped with 
additional congruence conditions. Write
$$\grg_c(\bfalp;\xi)=\sum_{\substack{1\le x\le X\\ x\equiv \xi\mmod{p^c}}}
e(\alp_1x+\ldots +\alp_kx^k),$$
in which $p$ is a preselected auxiliary prime number of size having order given by a small 
power of $X$. Further, for $0\le r\le k$, define the auxiliary mean value
\begin{equation}\label{2.1}
K_{a,b}^r(X)=\max_{\xi\not \equiv \eta\mmod{p}}\oint
|\grg_a(\bfalp;\xi)^{2r}\grg_b(\bfalp;\eta)^{2s-2r}|\d\bfalp .
\end{equation}
Then it follows via an application of H\"older's inequality that a prime $p$ may be chosen 
with 
$$J_{s,k}(X)\ll p^{2s-2r}K_{0,1}^r(X).$$

\par It is convenient to possess notation which makes transparent the extent to which the 
auxiliary mean value $K_{a,b}^r(X)$ exceeds its anticipated magnitude. With this goal in 
mind, when $s>k(k+1)/2$, we put
$$\llbracket K_{a,b}^r(X)\rrbracket =\frac{K_{a,b}^r(X)}{(X/p^a)^{2r-k(k+1)/2}
(X/p^b)^{2s-2r}},$$
and when $r\le s\le k(k+1)/2$ we instead define
$$\llbracket K_{a,b}^r(X)\rrbracket =\frac{K_{a,b}^r(X)}{(X/p^a)^r(X/p^b)^{s-r}}.$$
If, for a given value of $s$, the mean value $J_{s,k}(X)$ grows approximately like
$$X^\Lam (X^s+X^{2s-k(k+1)/2}),$$
with $\Lam>0$, then we can infer that for a small number $\eps>0$ one has 
$\llbracket K_{0,1}^r(X)\rrbracket \gg X^{\Lam-\eps}$. Here, we have made use of the 
implicit assumption that $p$ has size of order $X^\tet$, where $\tet>0$ is sufficiently small 
in terms of $\Lam$. Our strategy is now to concentrate this over-abundance of solutions 
underlying the mean value, relative to the expectation suggested by the main conjecture, 
through a sequence of auxiliary mean values $K_{a_n,b_n}^{r_n}(X)$ $(n\ge 1)$. The 
situation is simplest to describe when $s=k(k+1)$. Here, roughly speaking, one shows that 
for each $\eps>0$ one has
\begin{equation}\label{2.2}
\llbracket K_{a_n,b_n}^k(X)\rrbracket \gg X^{\Lam-\eps}(p^{\psi_n})^\Lam ,
\end{equation}
with $a_n\approx k^{n-1}$, $b_n\approx k^n$ and $\psi_n\approx nk^{n-1}(k-1)$. 
Provided that $\Lam>0$, then by permitting $n$ to become arbitrarily large sufficiently 
slowly, one finds that this lower bound for $\llbracket K_{a_n,b_n}^k(X)\rrbracket$ vastly 
exceeds even a trivial estimate for its upper bound, yielding a contradiction. Thus, we are 
forced to conclude that $\Lam=0$, and the main conjecture follows for $s\ge k(k+1)$.\par

The lower bound (\ref{2.2}) is obtained iteratively. By orthogonality, the mean value on the 
right hand side of (\ref{2.1}) counts the number of integral solutions of the system
\begin{equation}\label{2.3}
\sum_{i=1}^r(x_i^j-y_i^j)=\sum_{l=1}^{s-r}
\left( (p^bu_l+\eta)^j-(p^bv_l+\eta)^j\right) \quad (1\le j\le k),
\end{equation}
with $1\le \bfx,\bfy\le X$ and $(1-\eta)/p^b\le \bfu,\bfv\le (X-\eta)/p^b$, subject to the 
condition $\bfx\equiv \bfy\equiv \xi\mmod{p^a}$. The TDI property of the system 
(\ref{1.6}) ensures that (\ref{2.3}) is equivalent to the system of equations
$$\sum_{i=1}^r\left( (x_i-\eta)^j-(y_i-\eta)^j\right)=p^{jb}\sum_{l=1}^{s-r}
\left( u_l^j-v_l^j\right) \quad (1\le j\le k),$$
and thus one obtains the strong congruence condition
\begin{equation}\label{2.4}
\sum_{i=1}^r(x_i-\eta)^j\equiv \sum_{i=1}^r(y_i-\eta)^j\mmod{p^{jb}}\quad 
(1\le j\le k).
\end{equation}
One of the technical details suppressed here is the need to condition variables so that 
$x_1,\ldots ,x_r$ lie in distinct congruence classes modulo $p^{a+1}$. This guarantees 
a level of non-singularity in the solution set that may be exploited via Hensel's lemma.\par

In the simplest set-up with $s=k(k+1)$ and $r=k$, one shows that for a fixed choice of the 
$k$-tuple $\bfy$ modulo $p^{kb}$, there are at most $k!(p^{a+b})^{k(k-1)/2}$ possible 
choices for the $k$-tuple $\bfx$ modulo $p^{kb}$. Since $x_i\equiv \xi\mmod{p^a}$, each 
variable $x_i$ has at most $p^{kb-a}$ possible choices for its residue class modulo 
$p^{kb}$. Reinterpreting the system (\ref{2.3}) via orthogonality, and applying H\"older's 
inequality with care, one deduces that
\begin{align}
K_{a,b}^k(X)&\ll (p^{a+b})^{k(k-1)/2}(p^{kb-a})^k
\max_{\zeta\not\equiv \eta\mmod{p}}\oint |\grg_{kb}(\bfalp;\zeta)^{2k}
\grg_b(\bfalp;\eta)^{2s-2k}|\d\bfalp \notag\\
&\ll p^{\frac{1}{2}k(k-1)(a+b)+k(kb-a)}K_{b,kb}^k(X)^{\tfrac{k}{s-k}}
J_{s,k}(X/p^b)^{\tfrac{s-2k}{s-k}}.\label{2.5}
\end{align}
Here, we have applied the TDI property of the system (\ref{1.6}) to infer that
$$\oint |\grg_b(\bfalp;\eta)|^{2s}\d\bfalp \ll J_{s,k}(X/p^b).$$
The relation (\ref{2.5}) may be unwound with the estimate
$$J_{s,k}(X/p^b)\ll (X/p^b)^{2s-k(k+1)/2+\Lam+\eps}.$$
Thus one sees that
\begin{equation}\label{2.6}
\llbracket K_{a,b}^k(X)\rrbracket \ll X^\eps \llbracket K_{b,kb}^k(X)
\rrbracket ^{\tfrac{k}{s-k}}(X/p^b)^{\tfrac{s-2k}{s-k}\Lam},
\end{equation}
and a relation of the shape (\ref{2.2}) follows.\par

In the multigrade variant of efficient congruencing, the parameter $r$ in the congruences 
(\ref{2.4}) is varied, and we interpret this system in the form
\begin{equation}\label{2.7}
\sum_{i=1}^r(x_i-\eta)^j\equiv \sum_{i=1}^r(y_i-\eta)^j\mmod{p^{(k-r+1)b}}\quad 
(k-r+1\le j\le k).
\end{equation}
It follows via Hensel's lemma that, when $x_1,\ldots ,x_r$ lie in distinct congruence classes 
modulo $p^{a+1}$, then they are essentially congruent to a permutation of the residue 
classes $y_1,\ldots ,y_r$ modulo $p^{(k-r+1)b-(r-1)a}$. A variant of this observation plays 
a role in the work of the author \cite{FW2014} joint with Ford. Fixing a parameter $R$ with 
$1\le R\le k$, this new congruence information may be exploited in a similar manner to that 
delivering the relation (\ref{2.6}). One now finds that when $1\le r\le R$ and 
$R<s\le k(k+1)/2$, then one has
\begin{equation}\label{2.8}
\llbracket K_{a,b}^r(X)\rrbracket \ll X^\eps \llbracket K_{a,b}^{r-1}(X)
\rrbracket^{\tfrac{s-R-r}{s-R-r+1}}\llbracket K_{b,b'}^R(X)
\rrbracket^{\tfrac{1}{s-R-r+1}},
\end{equation}
where $b'=(k-r+1)b-(r-1)a$. Here, we note that when $r=1$, we have 
$K_{a,b}^{r-1}(X)=K_{a,b}^0(X)\ll J_{s,k}(X/p^b)$, using the TDI property of the system 
(\ref{1.6}).\par

The analysis of this new multigrade efficient congruencing method is necessarily more 
complicated than with (\ref{2.6}), since we have numerous quantities $K_{a,b}^r(X)$, with 
$1\le r\le R$, in play. The treatment of the iteration involves a complicated tree of possible 
outcomes. The analysis of \cite{Woo2015, Woo2016, Woo2017} iterates the relation 
(\ref{2.8}) to obtain a relation of the shape
$$\llbracket K_{a,b}^R(X)\rrbracket \ll X^\eps (X/p^b)^{\Lam \phi_0}\prod_{r=1}^R
\llbracket K_{b,b_r}^R(X)\rrbracket^{\phi_r},$$
in which $b_r=(k-r+1)b-(r-1)a$ and the exponents $\phi_r$ are appropriate positive 
numbers with $\phi_0+\ldots +\phi_R=1$. From here one can analyse the tree structure of 
the iterative process by weighting the outcomes, simplifying to a situation in which one has 
an averaged relation of the shape
$$\llbracket K_{a,b}^R(X)\rrbracket \ll X^\eps 
\llbracket K_{b,b_r}^R(X)\rrbracket^{\psi b/b_r}(X/p^b)^{\Lam \phi_0},$$
for some integer $r$ with $1\le r\le R$, with $\psi$ a positive number determined by the 
averaging process, and depending on $s$ and $R$. Of critical importance is whether or not 
the exponent $\psi$ exceeds $1$. If $\psi>1$, then the $p$-adic concentration argument is 
successful in delivering a contradiction to the assertion that $\Lam>0$, much as before, and 
the main conjecture follows for the value of $s$ in question. On the other hand, if $\psi<1$, 
then the iteration fails to deliver a contradiction. It transpires that a choice for $R$ may be 
made which permits $s$ to be as large as $k(k+1)/2-k/3+O(k^{2/3})$, and which allows 
the main conjecture to be proved for $J_{s,k}(X)$ in this way.\par

It is at this point that nested efficient congruencing enters the scene. The key observation 
is that the above iterative processes may be applied without alteration when the system of 
{\it equations} (\ref{1.6}) is replaced by a corresponding system of {\it congruences}
\begin{equation}\label{2.9}
\sum_{i=1}^s(x_i^j-y_i^j)\equiv 0\mmod{p^B}\quad (1\le j\le k),
\end{equation}
in which $B$ should be interpreted as a large integral parameter. At least, such is the case 
so long as two subsets of the variables $x_i$ and $y_i$ are restricted to congruence 
classes modulo $p^a$ and $p^b$, respectively, in which $a$ and $b$ are sufficiently small 
that the limitation to a mod $p^B$ environment plays no role in the above 
arguments. Such is assured in the system (\ref{2.9}) when $kb\le B$. Thus, when 
$s\le k(k+1)/2-k/3+O(k^{2/3})$, then in the multigrade argument just described, one may 
essentially conclude from (\ref{2.9}) that in an average sense the variables $\bfx$ and 
$\bfy$ are automatically constrained by the additional condition 
$x_i\equiv y_i\mmod{p^H}$, with $H=\lfloor B/k\rfloor$. More is true. If one has a system 
of polynomials $\varphi_1,\ldots ,\varphi_k\in \dbZ[t]$ with 
$\varphi_i(t)\equiv t^j\mmod{p^c}$, for some reasonably large parameter $c$, then the 
relation (\ref{2.9}) remains approximately true, since it holds modulo 
$p^{\text{min}\{c,B\}}$. This may be exploited to show that, in an average sense, one 
has $x_i\equiv y_i\mmod{p^h}$ with $h=\lfloor c/k\rfloor$. This additional information 
refines the approximation to the relation (\ref{2.9}) in a manner similar to the 
conventional proof of Hensel's lemma. By iterating this idea, one finds even in this more 
general situation that in an average sense one has $x_i\equiv y_i\mmod{p^H}$.\par

The observation just sketched may be employed as a substitute for Hensel's lemma in 
congruence systems of the shape
$$\sum_{i=1}^u(x_i-\eta)^j\equiv \sum_{i=1}^u(y_i-\eta)^j\mmod{p^{(k-r+1)b}}\quad 
(k-r+1\le j\le k),$$
analogous to (\ref{2.7}), though with $u$ as large as $r(r+1)/2$. The resulting congruence 
information on the variables $x_i$ and $y_i$ modulo $p^{b'}$, with 
$$b'=\lfloor (k-r+1)b/r\rfloor,$$
though weaker than in our earlier treatment, is spread out over many more variables. This, 
it transpires, offers a sufficient advantage that the earlier defect of $k/3$ variables may be 
remedied, thereby upgrading the applicability of the multigrade method so as to establish 
the main conjecture for $J_{s,k}(X)$ when $s\le k(k+1)/2$. The basic approach to the 
$p$-adic concentration argument and analysis of the tree of possible outcomes makes use 
of the same circle of ideas as in the basic multigrade approach. The detailed prosecution of 
this method will occupy our attention throughout \S\S3--12.

\section{The infrastructure for nested efficient congruencing} We begin by introducing the 
apparatus required for our proof of Theorem \ref{theorem1.1} via nested efficient 
congruencing. Although analogous to that of our previous work (see especially 
\cite{Woo2012, Woo2015, Woo2016, Woo2017}) concerning Vinogradov's mean value 
theorem, we deviate significantly from our previous path. In particular, we incorporate ideas 
from our work on discrete restriction theory \cite{Woo2017a} into the method.\par

Let $k$ be an integer with $k\ge 1$, and consider polynomials $\varphi_1,\ldots ,\varphi_k
\in \dbZ[t]$. Throughout our discussion, we have in mind a fixed prime number $p$ with 
$p>k$, and a large positive integer $B$. We require the polynomials $\varphi_j$ to be 
sufficiently independent for $1\le j\le k$. This is achieved by imposing the condition that the 
system of polynomials $\bfvarphi$ be $p^c$-{\it spaced} for an appropriate positive integer 
$c$, meaning that
\begin{equation}\label{3.1}
\varphi_j(t)\equiv t^j\mmod{p^c}\quad (1\le j\le k).
\end{equation}
For such a $p^c$-spaced system of polynomials $\bfvarphi$, one finds that the Wronskian 
$W(t;\bfvarphi)$ defined in (\ref{1.1}) is necessarily non-zero. Indeed, one has
$$W(t;\bfvarphi)\equiv \det 
\left( j(j-1)\cdots (j-i+1) t^{j-i}\right)_{1\le i,j\le k}=\prod_{i=1}^kj!
\not\equiv 0\mmod{p},$$
whence $W(t;\bfvarphi)\ne 0$. It transpires that the restriction to $p^c$-spaced systems 
of polynomials is easily accommodated when establishing such conclusions as Theorem 
\ref{theorem1.1}.\par

We next define the exponential sums and mean values central to our arguments. Consider 
a complex sequence $(\gra_n)_{n\in \dbZ}$ with $\sum_{n\in \dbZ}|\gra_n|<\infty$. We 
impose the latter condition for convenience, since this ensures that the Fourier series 
employed in our arguments are absolutely convergent, and hence that the moments of such 
series are finite. Formally speaking, all of our arguments apply under the assumption only 
that $\sum_{n\in \dbZ}|\gra_n|^r<\infty$ for some $r<2$, though this requires some 
interpretation. We have in mind the device of normalising our exponential sums. To this end, 
when $h$ is a non-negative integer and $\xi\in \dbZ$, we define 
$\rho_h(\xi)=\rho_h(\xi;\bfgra)$ by putting
\begin{equation}\label{3.2}
\rho_h(\xi;\bfgra)=\biggl( \sum_{n\equiv \xi\mmod{p^h}}|\gra_n|^2\biggr)^{1/2}.
\end{equation}
Here, we suppress the implicit assumption that the sum is taken over $n\in \dbZ$. For 
concision, we write $\rho_0=\rho_0(\bfgra)$ for 
$\rho_0(1;\bfgra)=\sum_{n\in \dbZ}|\gra_n|^2$. Note that, for all $h$ and $\xi$, one has 
$\rho_h(\xi)\le \rho_0<\infty$.\par

Next, write
\begin{equation}\label{3.3}
\psi(n;\bfalp)=\alp_1\varphi_1(n)+\ldots +\alp_k\varphi_k(n).
\end{equation}
When $h$ is a non-negative integer and $\xi\in \dbZ$, we define the exponential sum 
$\grf_h(\bfalp;\xi)=\grf_h(\bfalp;\xi;\bfgra;\bfvarphi)$ as follows. When $\rho_h(\xi)>0$, we 
put
\begin{equation}\label{3.4}
\grf_h(\bfalp;\xi)=\rho_h(\xi)^{-1}\sum_{n\equiv \xi\mmod{p^h}}\gra_ne(\psi(n;\bfalp)),
\end{equation}
and otherwise, when $\rho_h(\xi)=0$, we instead put $\grf_h(\bfalp;\xi)=0$. Of course, in 
the second of these alternatives, one has $\gra_n=0$ for each $n\in \dbZ$ with 
$n\equiv \xi\mmod{p^h}$, and the summation in (\ref{3.4}) is necessarily $0$. It is 
occasionally useful to abbreviate $\grf_0(\bfalp;\xi)$ to $f_\bfgra (\bfalp)$. Note that in 
the situation in which
$$\gra_n=\begin{cases} 1,&\text{when $1\le n\le N$},\\
0,&\text{otherwise},\end{cases}$$
one has
$$N^{1/2}f_\bfgra(\bfalp)=\sum_{1\le n\le N}e(\alp_1\varphi_1(n)+\ldots 
+\alp_k\varphi_k(n)).$$

\par In order to define the mean values of interest to us, we introduce some concise 
notation to ease our exposition. We extend the notation (\ref{1.9}) to accommodate implicit 
congruences as follows. Thus, when $B$ is a positive integer, we write
\begin{equation}\label{3.5}
\oint_{p^B}F(\bfalp)\d\bfalp =p^{-kB}\sum_{u_1\nmod{p^B}}\ldots 
\sum_{u_k\nmod{p^B}}F(\bfu/p^B),
\end{equation}
where the summations are taken over complete sets of residues modulo $p^B$. We then 
define the mean value $U_{s,k}^B(\bfgra)=U_{s,k}^{B,\bfvarphi}(\bfgra)$ by putting
\begin{equation}\label{3.6}
U_{s,k}^B(\bfgra)=\oint_{p^B}|f_\bfgra(\bfalp)|^{2s}d\bfalp .
\end{equation}
Note that, by orthogonality, the mean value $U_{s,k}^B(\bfgra)$ counts the integral 
solutions of the simultaneous congruences
\begin{equation}\label{3.7}
\sum_{i=1}^s\left( \varphi_j(x_i)-\varphi_j(y_i)\right) \equiv 0\mmod{p^B}\quad 
(1\le j\le k),
\end{equation}
with $\bfx,\bfy\in \dbZ$, and with each solution $\bfx,\bfy$ being counted with weight
$$\rho_0^{-2s}\prod_{i=1}^s\gra_{x_i}{\overline \gra_{y_i}}.$$

\par We consider also a mean value related to $U_{s,k}^B(\bfgra)$, though with 
underlying variables restricted to a common congruence class. Thus, we define 
$U_{s,k}^{B,h}(\bfgra)=U_{s,k}^{B,h,\bfvarphi}(\bfgra)$ by putting
\begin{equation}\label{3.8}
U_{s,k}^{B,h}(\bfgra)=\rho_0^{-2}\sum_{\xi\nmod{p^h}}\rho_h(\xi)^2\oint_{p^B}
|\grf_h(\bfalp;\xi)|^{2s}\d\bfalp .
\end{equation}
In this instance, it follows via orthogonality that the integral on the right hand side of 
(\ref{3.8}) counts the integral solutions $\bfx,\bfy$ of the system (\ref{3.7}) satisfying 
$\bfx\equiv \bfy\equiv \xi \mmod{p^h}$ with weight
$$\rho_h(\xi)^{-2s}\prod_{i=1}^s\gra_{x_i}{\overline \gra_{y_i}}.$$

\par Observe that when $H\ge 1$, one may decompose the exponential sum 
$f_\bfgra(\bfalp)$ according to residue classes modulo $p^H$. Thus, it follows from 
(\ref{3.4}) that
\begin{align*}
f_\bfgra(\bfalp)&=\rho_0^{-1}\sum_{\xi\nmod{p^H}}\,\sum_{n\equiv \xi\mmod{p^H}}
\gra_ne(\psi(n;\bfgra))\\
&=\rho_0^{-1}\sum_{\xi\nmod{p^H}}\rho_H(\xi)\grf_H(\bfalp;\xi).
\end{align*}
An application of H\"older's inequality therefore reveals that
$$|f_\bfgra(\bfalp)|^{2s}\le \rho_0^{-2s}\biggl( \sum_{\xi\nmod{p^H}}1\biggr)^s
\biggl( \sum_{\xi\nmod{p^H}}\rho_H(\xi)^2\biggr)^{s-1}\sum_{\xi\nmod{p^H}}
\rho_H(\xi)^2|\grf_H(\bfalp;\xi)|^{2s}.$$
Since it follows from (\ref{3.2}) that
$$\sum_{\xi\nmod{p^H}}\rho_H(\xi)^2=\sum_{n\in \dbZ}|\gra_n|^2=\rho_0^2,$$
we deduce that
\begin{equation}\label{3.9}
|f_\bfgra(\bfalp)|^{2s}\le \rho_0^{-2}p^{sH}\sum_{\xi\nmod{p^H}}\rho_H(\xi)^2
|\grf_H(\bfalp;\xi)|^{2s}.
\end{equation}
We consequently deduce from (\ref{3.6}) and (\ref{3.8}) that
\begin{equation}\label{3.10}
U_{s,k}^B(\bfgra)\le p^{sH}U_{s,k}^{B,H}(\bfgra).
\end{equation}
This relation motivates us to seek an exponent $\lam$ having the property that, for each 
$\eps>0$, one has
\begin{equation}\label{3.11}
U_{s,k}^B(\bfgra)\ll (p^H)^{\lam+\eps}U_{s,k}^{B,H}(\bfgra),
\end{equation}
with as much uniformity in the various parameters as is feasible. It is already apparent from 
(\ref{3.10}) that (\ref{3.11}) holds for some positive number $\lam$ satisfying the 
condition $\lam\le s$.

\par In order to make sense of the goal just enunciated, we make some simplifying 
observations. Observe first that the definition (\ref{3.4}) of $\grf_h(\bfalp;\xi)$ is scale 
invariant with respect to the sequence $(\gra_n)$. Thus, if $\gam>0$ and the sequence 
$(\gra_n)$ is replaced by $(\gam \gra_n)$, then the exponential sum $\grf_h(\bfalp;\xi)$ 
remains unchanged, and likewise therefore the mean value $U_{s,k}^B(\bfgra)$ defined 
in (\ref{3.6}) remains unchanged. Denote by $\dbD$ the set of sequences 
$(\gra_n)_{n\in \dbZ}$ with $|\gra_n|\le 1$ $(n\in \dbZ)$ and 
$$0<\sum_{n\in \dbZ}|\gra_n|<\infty.$$
Also, write $\dbD_0=\dbD\cup \{{\mathbf 0}\}$. Then we see that there is no loss of 
generality in restricting the sequences $(\gra_n)$ under consideration to lie in $\dbD$.\par

Next, when $\tau>0$, denote by $\Phi_\tau(B)$ the set of $p^c$-spaced $k$-tuples of 
polynomials $\bfvarphi$, with $c\ge \tau B$. The relation (\ref{3.10}) ensures that for 
each fixed $B\in \dbN$ and $\bfvarphi\in \Phi_\tau(B)$, one has
$$\sup_{(\gra_n)\in \dbD}\frac{\log (U_{s,k}^B(\bfgra)/U_{s,k}^{B,H}(\bfgra))}
{\log (p^H)}\le s$$
for every non-negative integer $H$. On the other hand, by considering a sequence 
$(\grb_n)\in \dbD$ with $\grb_n=0$ whenever $n\not\equiv 0\mmod{p^H}$, one discerns 
from (\ref{3.8}) that $U_{s,k}^B(\bfgrb)=U_{s,k}^{B,H}(\bfgrb)$, whence
$$\sup_{(\gra_n)\in \dbD}\frac{\log (U_{s,k}^B(\bfgra)/U_{s,k}^{B,H}(\bfgra))}
{\log (p^H)}\ge 0.$$
Given $s>0$, $\tet\ge 1$ and $\tau>0$, write $H=\lceil B/\tet\rceil$, and define
\begin{equation}\label{3.12}
\lam^*(s,\tet;\tau)=\limsup_{B\rightarrow \infty}\sup_{\bfvarphi\in \Phi_\tau(B)}
\sup_{(\gra_n)\in \dbD}\frac{\log (U_{s,k}^B(\bfgra)/U_{s,k}^{B,H}(\bfgra))}
{\log (p^H)}.
\end{equation}
Then we may be assured that $0\le \lam^*(s,\tet;\tau)\le s$. Finally, we define the limiting 
exponent
\begin{equation}\label{3.13}
\lam(s,\tet)=\limsup_{\tau\rightarrow 0}\lam^*(s,\tet;\tau),
\end{equation}
noting that, once again, one has
\begin{equation}\label{3.14}
0\le \lam(s,\tet)\le s.
\end{equation}

\par We are now equipped to announce a pivotal estimate that underpins the main inductive 
step in our argument.

\begin{theorem}\label{theorem3.1} Suppose that $k\in \dbN$ and that $p$ is a prime 
number with $p>k$. Then one has $\lam(k(k+1)/2,k)=0$.
\end{theorem}

Our methods would show, in fact, that for each natural number $k$ and positive number 
$s$, one has
\begin{equation}\label{3.15}
\lam(s,k)=\max\{ 0,s-k(k+1)/2\}.
\end{equation}
Our main discussion restricts attention to the special case $s=k(k+1)/2$ recorded in 
Theorem \ref{theorem3.1}. This case may be described in a more accessible manner, and is 
all that is required for the proof of the conclusions recorded in the introduction.

Theorem \ref{theorem3.1} may appear difficult to interpret, and for this reason we present 
the following corollary.

\begin{corollary}\label{corollary3.2} Suppose that $k\in \dbN$ and that $p$ is a prime 
number with $p>k$. In addition, let $\tau>0$ and $\eps>0$. Finally, let $B$ be sufficiently 
large in terms of $k$, $\tau$ and $\eps$. Put $s=k(k+1)/2$ and $H=\lceil B/k\rceil$. 
Then for every $\bfvarphi\in \Phi_\tau(B)$, and every sequence $(\gra_n)\in \dbD_0$, one 
has
\begin{equation}\label{3.16}
U_{s,k}^B(\bfgra)\ll p^{B\eps}U_{s,k}^{B,H}(\bfgra).
\end{equation}
\end{corollary}

We emphasise that the implicit constant in (\ref{3.16}) may depend on $k$, $\tau$ and 
$\eps$, but is independent of $\bfvarphi$. We note also that the validity of the relation 
(\ref{3.15}) implies that (\ref{3.16}) holds for all positive numbers $s\le k(k+1)/2$. The 
conclusion of Corollary \ref{corollary3.2} provides an essentially cost-free concentration 
towards the diagonal when $s\le k(k+1)/2$. For then the weighted count of solutions of the 
system (\ref{3.7}) is dominated by a diagonal contribution in which one may suppose, 
essentially speaking, that the variables are subject to the additional condition that 
$\bfx\equiv \bfy\mmod{p^{\lceil B/k\rceil}}$. In order to be more concrete concerning this 
phenomenon, consider the mean value
\begin{equation}\label{3.17}
\oint_{p^B}|\grf_H(\bfalp;\xi)|^{2s}\d\bfalp 
\end{equation}
occurring in the definition (\ref{3.8}) of $U_{s,k}^{B,H}(\bfgra)$. Should this mean value 
exhibit square-root cancellation, then in view of the normalisation visible in (\ref{3.4}), it 
follows from (\ref{3.8}) that $U_{s,k}^{B,H}(\bfgra)\ll 1$. The conclusion of Corollary 
\ref{corollary3.2} then shows that $U_{s,k}^B(\bfgra)\ll p^{B\eps}$, which is tantamount 
to square-root cancellation in the mean value
\begin{equation}\label{3.18}
\oint_{p^B}|f_\bfgra(\bfalp)|^{2s}\d\bfalp .
\end{equation}
Yet, a priori, the mean value (\ref{3.17}) is more likely to exhibit square-root cancellation 
than is (\ref{3.18}), since the former constrains its underlying variables to an arithmetic 
progression modulo $p^H$.\par

We remark that it would be possible to eliminate the condition $p>k$. Here, two 
approaches are possible. On the one hand, one could incorporate coefficients divisible by 
$p$ arising from the extraction of derivatives directly, taking account of the extent to which 
this inflates subsequent estimates. The impact is modest. Alternatively, one could replace 
the powers $t^j$ occurring in our definition of $p^c$-spaced systems by the binomial 
polynomials $\binom{t+j-1}{j}$.

\par We next introduce certain auxiliary mean values that play a key role in our 
arguments. Throughout, we fix $s=k(k+1)/2$. Let $a$, $b$, $c$ and $\nu$ be non-
negative integers. We consider a $p^c$-spaced $k$-tuple of polynomials 
$\bfvarphi\in \dbZ[t]^k$, and we fix a complex sequence $(\gra_n)\in \dbD_0$. When 
$0\le r\le k$, we define the mean value 
$K_{a,b}^r=K_{a,b,c}^{r,\bfvarphi,\nu}(\bfgra)$ by putting
\begin{equation}\label{3.19}
K_{a,b,c}^{r,\bfvarphi,\nu}(\bfgra)=\rho_0^{-4}\sum_{\xi\nmod{p^a}}
\sum_{\substack{\eta\nmod{p^b}\\ \xi\not\equiv \eta \mmod{p^\nu}}}\rho_a(\xi)^2
\rho_b(\eta)^2K_{a,b,c}^{r,\bfvarphi,\nu}(\bfgra;\xi,\eta),
\end{equation}
in which
\begin{equation}\label{3.20}
K_{a,b,c}^{r,\bfvarphi,\nu}(\bfgra;\xi,\eta)=\oint_{p^B}|\grf_a(\bfalp;\xi)^{2R}
\grf_b(\bfalp;\eta)^{2s-2R}|\d\bfalp 
\end{equation}
and $R=r(r+1)/2$. By orthogonality, the mean value 
$K_{a,b,c}^{r,\bfvarphi,\nu}(\bfgra;\xi,\eta)$ counts the 
integral solutions of the simultaneous congruences
\begin{equation}\label{3.21}
\sum_{i=1}^R\left(\varphi_j(x_i)-\varphi_j(y_i)\right)\equiv \sum_{l=1}^{s-R}
\left( \varphi_j(v_l)-\varphi_j(w_l)\right) \mmod{p^B}\quad (1\le j\le k),
\end{equation}
satisfying
\begin{equation}\label{3.22}
\bfx\equiv \bfy\equiv \xi\mmod{p^a}\quad \text{and}\quad \bfv\equiv 
\bfw\equiv \eta\mmod{p^b},
\end{equation}
with each solution being counted with weight
\begin{equation}\label{3.23}
\rho_a(\xi)^{-2R}\rho_b(\eta)^{2R-2s}\biggl( \prod_{i=1}^R\gra_{x_i}
{\overline \gra_{y_i}}\biggr) \biggl( \prod_{l=1}^{s-R}\gra_{v_l}
{\overline \gra_{w_l}}\biggr) .
\end{equation}

\par As in our previous work on efficient congruencing, our arguments are considerably 
simplified by making transparent the relationship between various mean values, on the 
one hand, and their anticipated magnitudes, on the other. We therefore consider 
normalised versions of these mean values $K_{a,b}^r$ as follows. When $1\le r\le k-1$ and 
$\Del$ is a positive number, we define
\begin{equation}\label{3.24}
\llbracket K_{a,b,c}^{r,\bfvarphi,\nu}(\bfgra)\rrbracketsub{\Del}=
\left( \frac{K_{a,b,c}^{r,\bfvarphi,\nu}(\bfgra)}{p^{\Del H}U_{s,k}^{B,H}(\bfgra)}
\right)^{\textstyle{\frac{k-1}{r(k-r)}}}.
\end{equation}
For much of our discussion, the choices of $\bfgra$, $\bfvarphi$, $\nu$ and $c$ will be 
considered fixed, and in such circumstances we suppress mention of them from our 
notation. We note that the presence of the exponent
$$\frac{k-1}{r(k-r)}$$
is designed to equalise the weights with which the mixed mean values $K_{a,b}^r$ 
occur within our arguments. The utility of this device will become apparent in due course.

\section{Translation-dilation invariant families} Consider a $p^c$-spaced system of 
polynomials $\bfvarphi$. In general, of course, the system of congruences
\begin{equation}\label{4.1}
\sum_{i=1}^s\varphi_j(x_i)\equiv \sum_{j=1}^s\varphi_j(y_i)\mmod{p^B}\quad 
(1\le j\le k)
\end{equation}
will not be translation-dilation invariant. However, in such special cases as the system 
given by $\varphi_j(t)=t^j$ $(1\le j\le k)$, it follows via an application of the binomial 
theorem that whenever $a\in \dbZ$ and $q\in \dbN$, and 
$(\bfx,\bfy)=(q\bfu+a,q\bfv+a)$ is a solution of (\ref{4.1}), then so too is $(\bfx,\bfy)=
(\bfu,\bfv)$. This translation-dilation invariance property may be preserved in 
{\it families} of systems $\bfvarphi$ which are not individually translation-dilation 
invariant. Thus, given $a\in \dbZ$ and $b\in \dbN$, and a solution $(\bfx,\bfy)=
(p^b\bfu+a,p^b\bfv+a)$ of the system (\ref{4.1}), one finds that $(\bfu,\bfv)$ is a 
solution of the system
$$\sum_{i=1}^s\psi_j(u_i)\equiv \sum_{i=1}^s\psi_j(v_i)\mmod{p^{B'}}\quad 
(1\le j\le k),$$
for some other $p^c$-spaced system of polynomials $\bfpsi$, with $B'$ an integer 
depending on $B$ and $b$. It is our goal in this section to establish estimates making this 
property explicit.\par

Our first lemma shows that restriction to arithmetic progressions modulo $p^h$ leads to 
bounds on mean values of the type $U_{s,k}^B(\bfgra)$ that scale appropriately with 
respect to the height of the arithmetic progression.

\begin{lemma}\label{lemma4.1} Suppose that $k\in \dbN$ and that $p$ is a prime 
number with $p>k$. Let $\tau$, $\eps$ and $\del$ be positive numbers with 
$\eps<\tau<\del<1$, and let $B$ be sufficiently large in terms of $s$, $k$ and $\eps$. 
Write $H=\lceil B/k\rceil$. Then for every $\bfvarphi\in \Phi_\tau(B)$, every sequence 
$(\gra_n)\in \dbD_0$, and every non-negative integer $h$ with $h\le (1-\del )H$, one has
$$U_{s,k}^{B,h}(\bfgra)\ll (p^{H-h})^{\lam(s,k)+\eps}U_{s,k}^{B,H}(\bfgra).$$
\end{lemma}

\begin{proof} Consider a $p^c$-spaced $k$-tuple of polynomials $\bfvarphi$ with $c\ge 
\tau B$, and a complex sequence $(\gra_n)\in \dbD_0$, with associated parameters 
satisfying the hypotheses of the statement of the lemma. Also, let $h$ be an integer with 
$0\le h\le (1-\del)H$. From the definition (\ref{3.4}) of the exponential sum 
$\grf_h(\bfalp;\xi)$, one has
\begin{equation}\label{4.2}
\grf_h(\bfalp;\xi)=\rho_h(\xi)^{-1}\sum_{y\in \dbZ}\grb_y(\xi)e(\psi(p^hy+\xi;\bfalp)),
\end{equation}
in which $\psi(n;\bfalp)$ is given by (\ref{3.3}), and the coefficients $\grb_y=\grb_y(\xi)$ 
are defined by putting
\begin{equation}\label{4.3}
\grb_y(\xi)=\gra_{p^hy+\xi}\quad (y\in \dbZ).
\end{equation}
By orthogonality, the integral on the right hand side of (\ref{3.8}) counts the integral 
solutions $\bfy,\bfz$ of the simultaneous congruences
\begin{equation}\label{4.4}
\sum_{i=1}^s\varphi_j(p^hy_i+\xi)\equiv \sum_{i=1}^s\varphi_j(p^hz_i+\xi)
\mmod{p^B}\quad (1\le j\le k),
\end{equation}
with each solution being counted with weight
\begin{equation}\label{4.5}
\rho_h(\xi)^{-2s}\prod_{i=1}^s\grb_{y_i}{\overline \grb_{z_i}}.
\end{equation}

\par The polynomial system $\bfvarphi$ is $p^c$-spaced, and hence it follows from 
(\ref{3.1}) that for suitable polynomials $\psi_j\in \dbZ[t]$, one may write
$$\varphi_j(t)=\sum_{i=1}^k\ome_{ij}t^i+p^ct^{k+1}\psi_j(t)\quad (1\le j\le k),$$
for some integral coefficient matrix $A_1=(\ome_{ij})_{1\le i,j\le k}$ congruent modulo 
$p^c$ to the $k\times k$ identity matrix $I_k$. Since $A_1\equiv I_k\mmod{p^c}$, it 
follows that $A_1$ possesses a multiplicative inverse $A_1^{-1}$ modulo $p^B$ having 
integral coefficients. By replacing $\bfvarphi$ by $A_1^{-1}\bfvarphi$ and $\bfpsi$ by 
$A_1^{-1}\bfpsi$, which amounts to taking suitable integral linear combinations of the 
congruences comprising (\ref{4.4}), we discern that there is no loss of generality in 
supposing that for $1\le j\le k$, one has
\begin{equation}\label{4.6}
\varphi_j(t)=t^j+p^ct^{k+1}\psi_j(t)\quad (1\le j\le k).
\end{equation}

\par With the latter assumption in hand, we apply the binomial theorem to (\ref{4.6}). Thus, 
when $1\le j\le k$, one finds that for suitable polynomials $\Psi_j\in \dbZ[t]$, one 
has $\varphi_j(p^hy+\xi)-\varphi_j(\xi)=\Phi_j(p^hy)$, in which
$$\Phi_j(t)=\sum_{i=1}^k\Ome_{ij}\xi^{j-i}t^i+p^ct^{k+1}\Psi_j(t)\quad (1\le j\le k),$$
and the integral coefficients $\Ome_{ij}$ satisfy
$$\Ome_{ij}\equiv \binom{j}{i}\mmod{p^c}\quad (1\le i,j\le k).$$
Here, we adopt the convention that the binomial coefficient $\binom{j}{i}$ is zero for 
$i>j$. Our hypothesis that $p>k$ ensures that the matrix 
$A_2=(\Ome_{ij})_{1\le i,j\le k}$ possesses a multiplicative inverse $A_2^{-1}$ modulo 
$p^B$ having integral coefficients, since it is triangular modulo $p^c$ with diagonal 
entries all equal to $1$. We now replace $\bfPhi$ by $A_2^{-1}\bfPhi$ and $\bfPsi$ by 
$A_2^{-1}\bfPsi$. Again, this amounts to taking suitable integral linear combinations of the 
congruences comprising (\ref{4.4}), and we see that there is no loss of generality in 
supposing that the coefficient matrix $A_2$ is equal to $I_k$. Hence, there exist polynomials 
$\Ups_j\in \dbZ[t]$ having the property that whenever the system (\ref{4.4}) is satisfied, 
then
\begin{equation}\label{4.7}
\sum_{i=1}^s(p^h)^j\left( \Phi_j(y_i)-\Phi_j(z_i)\right) \equiv 0\mmod{p^B}\quad 
(1\le j\le k),
\end{equation}
in which
\begin{equation}\label{4.8}
\Phi_j(t)=t^j+p^{c+h}t^{k+1}\Ups_j(t).
\end{equation}
The integral solutions $\bfy,\bfz$ of the original system of congruences (\ref{4.4}), 
counted with the weight (\ref{4.5}) associated with the definition (\ref{3.8}) of 
$U_{s,k}^{B,h}(\bfgra)$, are therefore constrained by the additional system of 
congruences
\begin{equation}\label{4.9}
\sum_{i=1}^s\Phi_j(y_i)\equiv \sum_{i=1}^s\Phi_j(z_i)\mmod{p^{B-kh}}\quad 
(1\le j\le k).
\end{equation}

\par In order to incorporate the extra condition (\ref{4.9}) into the mean value 
$U_{s,k}^{B,h}(\bfgra)$, we introduce the exponential sum
\begin{equation}\label{4.10}
\grg_h(\bfalp,\bfbet;\xi)=\rho_h(\xi)^{-1}\sum_{y\in \dbZ}\grc_y(\xi;\bfalp)
e(\psi^*(y;\bfbet)),
\end{equation}
where
\begin{equation}\label{4.11}
\grc_y(\xi;\bfalp)=\grb_y(\xi)e(\psi(p^hy+\xi;\bfalp))
\end{equation}
and
$$\psi^*(y;\bfbet)=\bet_1\Phi_1(y)+\ldots +\bet_k\Phi_k(y).$$
Equipped with this notation, it follows via orthogonality that
\begin{equation}\label{4.12}
\oint_{p^B}|\grf_h(\bfalp;\xi)|^{2s}\d\bfalp =\oint_{p^B}\oint_{p^{B-kh}}
|\grg_h(\bfalp,\bfbet;\xi)|^{2s}\d\bfbet \d\bfalp.
\end{equation}

\par Note that the polynomial system $\bfPhi$ defined via (\ref{4.8}) is 
$p^{c+h}$-spaced. Since we may assume that $h\le (1-\del)H$, moreover, one has
$$B-kh\ge B-k\lceil B/k\rceil +k\del \lceil B/k\rceil \ge \del B-k,$$
so that $B-kh$ may be assumed to be sufficiently large in terms of $s$, $k$ and $\eps$. 
It therefore follows from the definitions (\ref{3.12}) and (\ref{3.13}) that
$$U_{s,k}^{B-kh}(\bfgrc)\ll (p^{H-h})^{\lam(s,k)+\eps}U_{s,k}^{B-kh,H-h}(\bfgrc).$$
In view of the definition (\ref{4.10}), we therefore deduce that
\begin{equation}\label{4.13}
\oint_{p^{B-kh}}|\grg_h(\bfalp,\bfbet;\xi)|^{2s}\d\bfbet \ll 
(p^{H-h})^{\lam(s,k)+\eps}I_1,
\end{equation}
where
$$I_1=\rho_h(\xi)^{-2}\sum_{\eta\nmod{p^{H-h}}}\rho_{H-\eta}(p^h\eta+\xi)^2
\oint_{p^{B-kh}}|\grg^*_{H-h}(\bfalp,\bfbet;\xi,\eta)|^{2s}\d\bfbet ,$$
and
$$\grg^*_{H-h}(\bfalp,\bfbet;\xi,\eta)=\rho_{H-h}(p^h\eta +\xi)^{-1}
\sum_{y\equiv \eta\mmod{p^{H-h}}}\grc_y(\xi;\bfalp)e(\psi^*(y;\bfbet)).$$
Observe that there is a correspondence between residues $\zet$ modulo $p^H$ and 
pairs $(\xi,\eta)$ of residues modulo $p^h$, and modulo $p^{H-h}$, respectively. 
Indeed, if $1\le \zet \le p^H$, then we may identify $\xi$ and $\eta$ via the relations 
$\xi\equiv \zet\mmod{p^h}$ and $\eta\equiv (\zet-\xi)p^{-h}\mmod{p^{H-h}}$. 
Likewise, given $(\xi,\eta)$, we put $\zet=p^h\eta+\xi$. With this correspondence in 
mind, we find from (\ref{4.3}) and (\ref{4.11}) that 
$\grg^*_{H-h}(\bfalp,\bfbet;\xi,\eta)$ is equal to
$$\rho_H(\zet)^{-1}\sum_{p^hy+\xi\equiv \zet\mmod{p^H}}\gra_{p^hy+\xi}
e(\psi(p^hy+\xi;\bfalp)+\psi^*(y;\bfbet)).$$
In particular, we deduce from (\ref{4.12}) and (\ref{4.13}) that
\begin{align}
\sum_{\xi\nmod{p^h}}&\rho_h(\xi)^2\oint_{p^B}|\grf_h(\bfalp;\xi)|^{2s}
\d\bfalp \notag \\
&\ll (p^{H-h})^{\lam(s,k)+\eps}\sum_{\substack{\zet\nmod{p^H}\\ \zet=p^h\eta+\xi}}
\rho_H(\zet)^2\oint_{p^B}\oint_{p^{B-kh}}|\grg^*_{H-h}(\bfalp,\bfbet;\xi,\eta)|^{2s}
\d\bfbet\d\bfalp .\label{4.14}
\end{align}

\par By orthogonality, the mean value on the right hand side of (\ref{4.14}) counts integral 
solutions of the system of congruences (\ref{4.4}), with each solution being counted with 
weight (\ref{4.5}), but subject to the additional condition that
$$p^h\bfy+\xi\equiv p^h\bfz+\xi\equiv \zet \mmod{p^H},$$
and further subject to the congruence conditions (\ref{4.9}). The latter congruence 
conditions are generated by the integral over $\bfbet$ in (\ref{4.14}). However, the 
conditions (\ref{4.9}) are implied by (\ref{4.4}), as we have shown in the discussion above. 
We note also that when $\zet=p^h\eta+\xi$, then
$$\grg_{H-h}^*(\bfalp,{\mathbf 0};\xi,\eta)=\rho_H(\zet)^{-1}
\sum_{n\equiv \zet \mmod{p^H}}\gra_ne(\psi(n;\bfalp))=\grf_H(\bfalp;\zet).$$
Then on recalling the definition (\ref{3.8}), and noting that the conditions (\ref{4.9}) may be 
omitted, we deduce that
\begin{align*}
U_{s,k}^{B,h}(\bfgra)&\ll (p^{H-h})^{\lam(s,k)+\eps}\rho_0^{-2}\sum_{\zet\nmod{p^H}}
\rho_H(\zet)^2\oint_{p^B}|\grf_H(\bfalp;\zet)|^{2s}\d\bfalp \\
&=(p^{H-h})^{\lam(s,k)+\eps}U_{s,k}^{B,H}(\bfgra).
\end{align*}
This completes the proof of the lemma.
\end{proof}

It is tempting to replace the proof of Lemma \ref{lemma4.1} with an informal argument 
appealing to Taylor expansions, rescaling and an appeal to translation invariance. However, 
as should be apparent from our proof above, there are subtle technical issues that arise in 
a detailed argument that threaten to sabotage the desired conclusion. We therefore offer 
no apology (beyond this remark) for expending space on the detailed account above.\par

The estimate supplied by Lemma \ref{lemma4.1} is easily transformed into a crude 
estimate for the mean value $K_{a,b,c}^{r,\bfvarphi,\nu}(\bfgra)$ which nonetheless 
has considerable utility.

\begin{lemma}\label{lemma4.2} Suppose that $k\in \dbN$ and that $p$ is a prime number 
with $p>k$. Let $\tau$, $\eps$ and $\del$ be positive numbers with $\eps<\tau<\del<1$, 
and let $B$ be sufficiently large in terms of $s$, $k$ and $\eps$. Write $H=\lceil B/k\rceil$ 
and suppose that $r$ and $\nu$ are non-negative integers with $1\le r\le k-1$. Then for 
every $\bfvarphi\in \Phi_\tau(B)$, every sequence $(\gra_n)\in \dbD_0$, and all 
non-negative integers $a$ and $b$ satisfying $\max\{a,b\}\le (1-\del)H$, one has
$$\llbracket K_{a,b,c}^{r,\bfvarphi,\nu}(\bfgra)\rrbracketsub{0}\ll (p^H)^{\lam(s,k)+\eps}.
$$
\end{lemma} 

\begin{proof} An application of H\"older's inequality delivers the estimate
\begin{equation}\label{4.15}
\rho_0^{-4}\sum_{\xi\nmod{p^a}}\sum_{\eta\nmod{p^b}}\rho_a(\xi)^2\rho_b(\eta)^2
\oint_{p^B}|\grf_a(\bfalp;\xi)^{2R}\grf_b(\bfalp;\eta)^{2s-2R}|\d\bfalp \le 
I_1^{R/s}I_2^{1-R/s},
\end{equation}
where
$$I_1=\rho_0^{-4}\sum_{\xi\nmod{p^a}}\sum_{\eta\nmod{p^b}}\rho_a(\xi)^2
\rho_b(\eta)^2\oint_{p^B}|\grf_a(\bfalp;\xi)|^{2s}\d\bfalp $$
and
$$I_2=\rho_0^{-4}\sum_{\xi\nmod{p^a}}\sum_{\eta\nmod{p^b}}\rho_a(\xi)^2
\rho_b(\eta)^2\oint_{p^B}|\grf_b(\bfalp;\eta)|^{2s}\d\bfalp .$$
By reference to (\ref{3.2}) and (\ref{3.8}), one finds that
$$I_1=\rho_0^{-2}\biggl( \sum_{\eta\nmod{p^b}}\rho_b(\eta)^2\biggr) 
U_{s,k}^{B,a}(\bfgra)=U_{s,k}^{B,a}(\bfgra),$$
and likewise one sees that $I_2=U_{s,k}^{B,b}(\gra)$. Consequently, on applying Lemma 
\ref{lemma4.1}, we deduce from (\ref{3.19}) and (\ref{4.15}) that
\begin{equation}\label{4.16}
K_{a,b,c}^{r,\bfvarphi,\nu}(\bfgra)\ll \left( (p^{H-a})^{\lam(s,k)+\eps}\right)^{R/s}
\left( (p^{H-b})^{\lam(s,k)+\eps}\right)^{1-R/s}U_{s,k}^{B,H}(\bfgra).
\end{equation}

\par Next we recall (\ref{3.24}). This conveys us from the estimate (\ref{4.16}) to the 
corresponding normalised bound
$$\llbracket K_{a,b,c}^{r,\bfvarphi,\nu}(\bfgra)
\rrbracketsub{0}^{\textstyle{\frac{r(k-r)}{k-1}}}\ll 
\left( (p^{H-a})^{R/s}(p^{H-b})^{1-R/s}\right)^{\lam(s,k)+\eps}.$$
But when $1\le r\le k-1$, one finds that $r(k-r)\ge k-1$, whence
$$\frac{r(k-r)}{k-1}\ge 1.$$
Thus we conclude that
$$\llbracket K_{a,b,c}^{r,\bfvarphi,\nu}(\bfgra)\rrbracketsub{0}\ll 
\left( p^{H(\lam(s,k)+\eps)}\right)^{\textstyle{\frac{k-1}{r(k-r)}}}\ll 
(p^H)^{\lam(s,k)+\eps}.$$
This completes the proof of the lemma.
\end{proof}

\section{The base of the induction: the trivial case $k=1$} Most of the basic infrastructure 
and the skeletal properties of the key mean values are now in place. Our overarching 
strategy is to establish Theorem \ref{theorem3.1} by induction. Thus, assuming the validity 
of Theorem \ref{theorem3.1} for exponents smaller than $k$, we seek to establish its 
conclusion for the exponent $k$. We begin in this section by establishing the base case 
$k=1$. Although essentially trivial, the underlying ideas are instructive. We remark that the 
passing similarity with the argument underlying the conventional proof of Hensel's lemma is 
not accidental.

\begin{lemma}\label{lemma5.1} For any prime number $p$, one has $\lam(1,1)=0$.
\end{lemma}

\begin{proof} Let $\tau $ be a small positive number, and let $B$ be a positive integer 
sufficiently large in terms of $\tau$. Consider any sequence $(\gra_n)\in \dbD$ and 
polynomial $\varphi\in \Phi_\tau(B)$. We may suppose that $\varphi$ is a $p^c$-spaced 
polynomial for some $c\ge \tau B$, so that $\varphi(t)=t+p^c\psi(t)$ for a polynomial 
$\psi\in \dbZ[t]$. The mean value $U_{1,1}^B(\bfgra)$ counts the integral solutions of the 
congruence
\begin{equation}\label{5.1}
\varphi(x)\equiv \varphi(y)\mmod{p^B},
\end{equation}
with each solution $x,y$ being counted with weight $\rho_0^{-2}\gra_x{\overline \gra_y}$. 
The congruence (\ref{5.1}) is in fact the relation
\begin{equation}\label{5.2}
x+p^c\psi(x)\equiv y+p^c\psi(y)\mmod{p^B},
\end{equation}
from which we infer that $x\equiv y\mmod{(p^c,p^B)}$. Here, we take the liberty of writing 
$p^{\min\{c,B\}}$ as the highest common factor $(p^c,p^B)$, this being cosmetically 
slightly less awkward.\par

Since $x-y$ divides $\psi(x)-\psi(y)$, it follows that the highest common factor $(p^c,p^B)$ 
divides $\psi(x)-\psi(y)$. Substituting this relation back into (\ref{5.2}), we deduce that 
$x\equiv y\mmod{(p^{2c},p^B)}$. By repeating this argument no more than 
$\lceil 1/\tau\rceil $ times, we conclude that $x\equiv y\mmod{p^B}$. We may therefore 
classify the solutions of the congruence (\ref{5.2}) according to the common congruence 
class $\xi$ modulo $p^B$ of $x$ and $y$. On recalling the definitions (\ref{3.4}), (\ref{3.6}) 
and (\ref{3.8}), we infer via orthogonality that
$$U_{1,1}^B(\bfgra)=\rho_0^{-2}\sum_{1\le \xi\le p^B}\rho_B(\xi)^2\oint 
|\grf_B(\bfalp;\xi)|^2\d\bfalp =U_{1,1}^{B,B}(\bfgra).$$
Thus we conclude that
$$\log \left( U_{1,1}^B(\bfgra)/U_{1,1}^{B,B}(\bfgra)\right) =0,$$
and it is immediate from (\ref{3.12}) that $\lam^*(1,1;\tau)=0$. We therefore conclude 
from (\ref{3.13}) that $\lam(1,1)=0$, completing the proof of the lemma.
\end{proof}

\section{The initial conditioning process} We now move on to our main inductive task of 
establishing Theorem \ref{theorem3.1} for the exponent $k\ge 2$, assuming its validity 
for exponents smaller than $k$. The initial step in the estimation of $U_{s,k}^B(\bfgra)$ 
is to bound it in terms of a mean value of the shape $K_{a,b,c}^{1,\bfvarphi,\nu}(\bfgra)$. 
In this section we describe the initial set-up to be applied in later sections, as well as the 
conditioning of the variables underlying the mean value 
$K_{a,b,c}^{1,\bfvarphi,\nu}(\bfgra)$.\par
 
Recall the definitions (\ref{3.12}) and (\ref{3.13}). We fix a prime number $p$ with 
$p>k$ and we also fix $s=k(k+1)/2$, and then we seek to show that $\lam(s,k)=0$. 
This we achieve by deriving a contradiction to the assumption that $\lam(s,k)=\Lam>0$. 
Note here that in view of (\ref{3.14}), there is no loss of generality in assuming that
\begin{equation}\label{6.1}
0<\Lam\le s.
\end{equation}
We next introduce a hierarchy of sufficiently small positive numbers $\eps$, $\tau$, $\del$ 
and $\mu$ with
\begin{equation}\label{6.2}
\eps<\tau<\del<\mu<1.
\end{equation}
We suppose that each element in the hierarchy is sufficiently small in terms of $k$, 
$\Lam$, and the larger elements of the hierarchy. It follows from (\ref{3.13}) that there is 
no loss of generality in supposing that $\lam^*(s,k;\tau)\ge \Lam-\eps/2$. Thus, in view 
of (\ref{3.12}), there exists a sequence $(B_m)_{m=1}^\infty $, with 
$B_m\rightarrow \infty$ as $m\rightarrow \infty$, having the property that whenever $m$ 
is sufficiently large, then there exists $\bfvarphi_m\in \Phi_\tau(B_m)$ and 
$(\gra_n^{(m)})\in \dbD$ for which
\begin{equation}\label{6.3}
U_{s,k}^{B_m}(\bfgra^{(m)})\ge (p^{H_m})^{\Lam-\eps}
U_{s,k}^{B_m,H_m}(\bfgra^{(m)}),
\end{equation}
where we have written $H_m=\lceil B_m/k\rceil$. We now fix such an integer $m$ with 
with the property that $B=B_m$ is sufficiently large in terms of $k$, $\Lam$, $\mu$, $\del$, 
$\tau$ and $\eps$. In our discussion to come, we may henceforth omit mention 
of the subscript or superscript $m$ in $\bfvarphi_m$, $(\gra_n^{(m)})$, $H_m$. Observe 
that since $B$ is sufficiently large in terms of the basic parameters, then we may suppose 
also that $\eps B$ is also sufficiently large in terms of $k$, $\Lam$, $\mu$, $\tau$, 
$\del$ and $\eps$.\par

Also, by virtue of Lemma \ref{lemma4.1}, we may assume that for every non-negative 
integer $h$ with $h\le (1-\del )H$, uniformly in $\bfvarphi'\in \Phi_\tau(B)$ and 
$(\gra_n')\in \dbD_0$, one has
\begin{equation}\label{6.4}
U_{s,k}^{B,h}(\bfgra')\le (p^{H-h})^{\Lam+\eps}U_{s,k}^{B,H}(\bfgra').
\end{equation}

\par We may now announce our first bound for $U_{s,k}^B(\bfgra)$ in terms of mean 
values of the shape $K_{a,b,c}^{1,\bfvarphi,\nu}(\bfgra)$. Here, we fix a choice for the 
parameter $\nu$ for future use by setting
\begin{equation}\label{6.5}
\nu=\lceil 4\eps H\Lam^{-1}\rceil .
\end{equation}
Since we may suppose that $\bfvarphi\in \Phi_\tau(B)$, it follows that there exists an 
integer $c$ with $c\ge \tau B$ with the property that $\bfvarphi$ is $p^c$-spaced.

\begin{lemma}\label{lemma6.1} One has $U_{s,k}^B(\bfgra)\ll p^{s\nu}
K_{\nu,\nu,c}^{1,\bfvarphi,\nu}(\bfgra)$.
\end{lemma}

\begin{proof} In view of the definition of $f_\bfgra(\bfalp)$ and the definition (\ref{3.4}), 
one has
$$\rho_0^2f_\bfgra(\bfalp)^2=\sum_{\xi\nmod{p^\nu}}\rho_\nu(\xi)\grf_\nu(\bfalp;\xi)
\left( \rho_0f_\bfgra(\bfalp)\right) .$$
Moreover, for each given residue $\xi$ modulo $p^\nu$, one has
$$\rho_0f_\bfgra(\bfalp)=\rho_\nu(\xi)f_\nu(\bfalp;\xi)+
\sum_{\substack{\eta\nmod{p^\nu}\\ \eta\not \equiv \xi\mmod{p^\nu}}}\rho_\nu(\eta)
f_\nu(\bfalp;\eta).$$
Thus
\begin{equation}\label{6.6}
|f_\bfgra(\bfalp)|^2\le T_1(\bfalp)+T_2(\bfalp),
\end{equation}
where
$$T_1(\bfalp)=\rho_0^{-2}\sum_{\xi\nmod{p^\nu}}\rho_\nu(\xi)^2|f_\nu(\bfalp;\xi)|^2$$
and
$$T_2(\bfalp)=\rho_0^{-2}\sum_{\xi\nmod{p^\nu}}\sum_{\substack{\eta\nmod{p^\nu}\\
\eta\not\equiv \xi\mmod{p^\nu}}}\rho_\nu(\xi)\rho_\nu(\eta)|f_\nu(\bfalp;\xi)
f_\nu(\bfalp;\eta)|.$$

\par In much the same manner as in the derivation of the relation (\ref{3.9}), an application 
of H\"older's inequality reveals that
\begin{align}
T_1(\bfalp)^s&\le \rho_0^{-2s}\biggl( \sum_{\xi\nmod{p^\nu}}\rho_\nu(\xi)^2
\biggr)^{s-1}\sum_{\xi\nmod{p^\nu}}\rho_\nu(\xi)^2|\grf_\nu(\bfalp;\xi)|^{2s}\notag \\
&=\rho_0^{-2}\sum_{\xi\nmod{p^\nu}}\rho_\nu(\xi)^2|\grf_\nu(\bfalp;\xi)|^{2s}.
\label{6.7}
\end{align}
Meanwhile, again applying H\"older's inequality, one finds that
$$T_2(\bfalp)^{2s}\le \rho_0^{-4s}T_3^sT_4^{s-2}T_5T_6,$$
where
\begin{align*}
T_3&=\sum_{\xi\nmod{p^\nu}}\sum_{\eta\nmod{p^\nu}}1,\\
T_4&=\sum_{\xi\nmod{p^\nu}}\sum_{\eta\nmod{p^\nu}}\rho_\nu(\xi)^2\rho_\nu(\eta)^2,
\\
T_5&=\sum_{\xi\nmod{p^\nu}}
\sum_{\substack{\eta\nmod{p^\nu}\\ \eta\not\equiv \xi\mmod{p^\nu}}}\rho_\nu(\xi)^2
\rho_\nu(\eta)^2|\grf_\nu(\bfalp;\xi)^2\grf_\nu(\bfalp;\eta)^{2s-2}|,\\
T_6&=\sum_{\xi\nmod{p^\nu}}
\sum_{\substack{\eta\nmod{p^\nu}\\ \eta\not \equiv \xi\mmod{p^\nu}}}\rho_\nu(\xi)^2
\rho_\nu(\eta)^2|\grf_\nu(\bfalp;\eta)^2\grf_\nu(\bfalp;\xi)^{2s-2}|.
\end{align*}
We have $T_3=p^{2\nu}$ and $T_4=\rho_0^4$. Also, by symmetry, one has 
$T_5=T_6$. We therefore deduce that
\begin{equation}\label{6.8}
T_2(\bfalp)^s\le \rho_0^{-4}p^{\nu s}\sum_{\xi \nmod{p^\nu}}
\sum_{\substack{\eta\nmod{p^\nu}\\ \eta\not \equiv \xi\mmod{p^\nu}}}\rho_\nu(\xi)^2
\rho_\nu(\eta)^2|\grf_\nu (\bfalp;\xi)^2\grf_\nu(\bfalp;\eta)^{2s-2}|\d\bfalp .
\end{equation}

\par On recalling the definitions (\ref{3.8}) and (\ref{3.19}), we deduce from (\ref{6.7}) 
and (\ref{6.8}) that
$$\oint_{p^B}T_1(\bfalp)^s\d\bfalp =U_{s,k}^{B,\nu}(\bfgra)\quad \text{and}\quad 
\oint_{p^B}T_2(\bfalp)^s\d\bfalp \le p^{\nu s}K_{\nu,\nu,c}^{1,\bfvarphi,\nu}
(\bfgra).$$
Since it follows from (\ref{6.6}) that
$$|f_\bfgra(\bfalp)|^{2s}\ll |T_1(\bfalp)|^s+|T_2(\bfalp)|^s,$$
we deduce from the definition (\ref{3.6}) that
\begin{equation}\label{6.9}
U_{s,k}^B(\bfgra)\ll U_{s,k}^{B,\nu}(\bfgra)+p^{\nu s}
K_{\nu,\nu,c}^{1,\bfvarphi,\nu}(\bfgra).
\end{equation}

\par Next we make use of the estimate (\ref{6.4}) to obtain the bound
$$U_{s,k}^{B,\nu}(\bfgra)\le (p^{H-\nu})^{\Lam+\eps}U_{s,k}^{B,H}(\bfgra).$$
Since the definition (\ref{6.5}) ensures that
$$(\Lam+\eps)(H-\nu)-(\Lam-\eps)H=2\eps H-(\Lam+\eps)\nu<-2\eps H,$$
we conclude via (\ref{6.3}) that
$$U_{s,k}^{B,\nu}(\bfgra)\le p^{-2\eps H}(p^H)^{\Lam-\eps}U_{s,k}^{B,H}(\bfgra)
\le p^{-2\eps H}U_{s,k}^B(\bfgra).$$
Thus we infer from (\ref{6.9}) that
$$U_{s,k}^B(\bfgra)\ll p^{-2\eps H}U_{s,k}^B(\bfgra)+p^{\nu s}
K_{\nu,\nu,c}^{1,\bfvarphi,\nu}(\bfgra),$$
whence
$$U_{s,k}^B(\bfgra)\ll p^{\nu s}K_{\nu,\nu,c}^{1,\bfvarphi,\nu}(\bfgra).$$
This completes the proof of the lemma.
\end{proof}

The congruence condition implicit in the mean value 
$K_{\nu,\nu,c}^{1,\bfvarphi,\nu}(\bfgra)$ is not yet strong enough to pursue our main 
iteration, and so we pause to strengthen it before proceeding further. It is at this point that 
the parameter $\mu$ enters the scene. We put
\begin{equation}\label{6.10}
\tet=\lceil \mu H\rceil .
\end{equation}
We briefly pause to record a useful, though essentially trivial, upper bound permitting us to 
relate exponential sums restricted to congruences associated with differing powers of $p$.

\begin{lemma}\label{lemma6.2} Suppose that $a$ and $b$ are non-negative integers 
with $a\le b$. Then for any positive number $w$ and any integer $\xi$, one has
$$\rho_a(\xi)^2|\grf_a(\bfalp;\xi)|^{2w}\le (p^{b-a})^w
\sum_{\substack{\zet\nmod{p^b}\\ \zet\equiv \xi\mmod{p^a}}}\rho_b(\zet)^2
|\grf_b(\bfalp;\zet)|^{2w}.$$
\end{lemma}

\begin{proof} By applying H\"older's inequality in combination with (\ref{3.4}), one obtains
\begin{align}
\rho_a(\xi)^{2w}|f_a(\bfalp;\xi)|^{2w}&=\biggl| \sum_{\substack{\zet\nmod{p^b}\\ 
\zet\equiv \xi\mmod{p^a}}}\rho_b(\zet)\grf_b(\bfalp;\zet)\biggr|^{2w}\notag \\
&\le U_1^wU_2^{w-1}\sum_{\substack{\zet\nmod{p^b}\\ \zet\equiv \xi\mmod{p^a}}}
\rho_b(\zet)^2|\grf_b(\bfalp;\zet)|^{2w},\label{6.11}
\end{align}
where
$$U_1=\sum_{\substack{\zet\nmod{p^b}\\ \zet\equiv \xi\mmod{p^a}}}1=p^{b-a}$$
and
$$U_2=\sum_{\substack{\zet\nmod{p^b}\\ \zet\equiv \xi\mmod{p^a}}}
\rho_b(\zet)^2=\sum_{n\equiv \xi\mmod{p^a}}|\gra_n|^2=\rho_a(\xi)^2.$$
The desired conclusion is immediate on substituting the latter relations into (\ref{6.11})
\end{proof}

\begin{lemma}\label{lemma6.3} One has $U_{s,k}^B(\bfgra)\ll p^{s\tet}
K_{\tet,\tet,c}^{1,\bfvarphi,\nu}(\bfgra)$.
\end{lemma}

\begin{proof} On recalling (\ref{3.20}), we deduce via two applications of Lemma 
\ref{lemma6.2} that the mean value $K_{\nu,\nu,c}^{1,\bfvarphi,\nu}(\bfgra;\xi,\eta)$ is 
bounded above by
$$\rho_\nu(\xi)^{-2}\rho_\nu(\eta)^{-2}(p^{\tet-\nu})^s
\sum_{\xi',\eta'\nmod{p^\tet}}\rho_\tet(\xi')^2\rho_\tet(\eta')^2
\oint_{p^B}|\grf_\tet(\bfalp;\xi')^2\grf_\tet(\bfalp;\eta')^{2s-2}|\d\bfalp ,$$
where the summation over $\xi'$ and $\eta'$ is subject to the conditions
$$\xi'\equiv\xi\mmod{p^\nu}\quad \text{and}\quad \eta'\equiv \eta \mmod{p^\nu}.$$
Thus $K_{\nu,\nu,c}^{1,\bfvarphi,\nu}(\bfgra;\xi,\eta)$ is at most
$$\rho_\nu(\xi)^{-2}\rho_\nu(\eta)^{-2}(p^{\tet-\nu})^s\sum_{\xi',\eta'\nmod{p^\tet}}
\rho_\tet(\xi')^2\rho_\tet(\eta')^2K_{\tet,\tet,c}^{1,\bfvarphi,\nu}(\bfgra;\xi',\eta').$$
In view of the definition (\ref{3.19}), we find that
\begin{align*}
K_{\nu,\nu,c}^{1,\bfvarphi,\nu}(\bfgra)&\le \rho_0^{-4}(p^{\tet-\nu})^s
\sum_{\xi'\nmod{p^\tet}}\sum_{\substack{\eta'\nmod{p^\tet}\\ \eta'\not \equiv \xi' 
\mmod{p^\nu}}}\rho_\tet(\xi')^2\rho_\tet(\eta')^2K_{\tet,\tet,c}^{1,\bfvarphi,\nu}
(\bfgra;\xi',\eta')\\
&=(p^{\tet-\nu})^sK_{\tet,\tet,c}^{1,\bfvarphi,\nu}(\bfgra).
\end{align*}
By substituting this bound into the estimate supplied by Lemma \ref{lemma6.1}, we arrive 
at the bound
$$U_{s,k}^B(\bfgra)\ll p^{s(\tet-\nu)+s\nu}K_{\tet,\tet,c}^{1,\bfvarphi,\nu}(\bfgra),$$
and the conclusion of the lemma follows at once.
\end{proof}

\section{Harnessing approximate translation-dilation invariance} This is the section of the 
paper that does the heavy lifting, for we now initiate our iterative process. Throughout, we 
suppose that $k\ge 2$, and we assume the validity of Theorem \ref{theorem3.1} for 
exponents smaller than $k$. Our strategy is to approximate the system of congruences 
underlying the mean value $K_{a,b,c}^{r,\bfvarphi, \nu}(\bfgra)$ by a corresponding 
system of Vinogradov type. As in our earlier efficient congruencing methods, we are able to 
extract a strong congruence condition on the underlying variables. Throughout, the 
conventions of the previous section remain in play, and in particular we write $R=r(r+1)/2$.

\begin{lemma}\label{lemma7.1} Suppose that $a$, $b$ and $r$ are integers with 
\begin{equation}\label{7.1}
1\le r\le k-1,\quad a\ge \del \tet,\quad b\ge \del \tet\quad \text{and}\quad ra\le 
(k-r+1)b\le B.
\end{equation}
Put
\begin{equation}\label{7.2}
b^\prime =\lceil (k-r+1)b/r\rceil.
\end{equation}
Then one has
$$K_{a,b,c}^{r,\bfvarphi,\nu}(\bfgra)\ll p^{k^2\nu}
K_{b^\prime ,b,c}^{r,\bfvarphi,\nu}(\bfgra).$$
\end{lemma}

\begin{proof} In view of the definition (\ref{3.19}), we focus initially on the mean value 
$K_{a,b,c}^{r,\bfvarphi,\nu}(\bfgra;\xi,\eta)$, in which we may suppose that 
$\eta\not\equiv \xi\mmod{p^\nu}$. The latter condition permits us the hypothesis 
$p^\gam \| (\xi-\eta)$ for some non-negative integer $\gam$ with $\gam<\nu$. We define 
the integer $\ome$ by putting $\ome=(\xi-\eta)p^{-\gam}$, so that one has 
$(\ome ,p)=1$. It is convenient for future use to introduce the parameter
\begin{equation}\label{7.3}
B^\prime=(k-r+1)b-ra-(k-r)\gam.
\end{equation}
It transpires that in our main argument it is useful to have a little extra room in which to 
work. We therefore begin by exploring the situation in which $B'\le \nu$. In such 
circumstances, it follows from (\ref{7.3}) that
$$(k-r+1)b-ra\le \nu +(k-r)\gam\le k\nu.$$
On recalling (\ref{7.2}), one then finds that
\begin{equation}\label{7.4}
b^\prime -a\le 1+k\nu/r,
\end{equation}
and an application of Lemma \ref{lemma6.2} leads to the bound
$$\rho_a(\xi)^2|\grf_a(\bfalp;\xi)|^{2R}\le (p^{b^\prime -a})^R
\sum_{\substack{\zet\nmod{p^{b^\prime}}\\ \zet\equiv \xi\mmod{p^a}}}
\rho_{b^\prime}(\zet)^2|\grf_{b^\prime}(\bfalp;\zet)|^{2R}.$$
Thus we discern from (\ref{3.19}) and (\ref{3.20}) that
$$K_{a,b,c}^{r,\bfvarphi,\nu}(\bfgra)\le (p^{b^\prime-a})^R
K_{b^\prime,b,c}^{r,\bfvarphi,\nu}(\bfgra).$$
Consequently, in view of (\ref{7.4}) and the relation $R=r(r+1)/2<kr/2$, we obtain the 
estimate
$$K_{a,b,c}^{r,\bfvarphi,\nu}(\bfgra)\le p^{k^2\nu}
K_{b^\prime,b,c}^{r,\bfvarphi,\nu}(\bfgra).$$
This bound delivers the conclusion of the lemma when $B^\prime \le \nu$.\par

We now focus on the alternate situation with
\begin{equation}\label{7.5}
B^\prime >\nu.
\end{equation}
We recall that, by orthogonality and the definition (\ref{3.20}), the mean value 
$K_{a,b,c}^{r,\bfvarphi,\nu}(\bfgra;\xi,\eta)$ counts the integral solutions of the 
simultaneous congruences (\ref{3.21}), with their attendant conditions (\ref{3.22}), and with 
each solution being counted with weight (\ref{3.23}). We temporarily focus exclusively on 
the nature of these congruences.\par

In the first phase of our argument, we seek to extract from the right hand side of 
(\ref{3.21}) a system resembling one of Vinogradov type. Observe that $\bfvarphi$ is a 
$p^c$-spaced system, so it follows just as in the argument of the proof of Lemma 
\ref{lemma4.1} leading to (\ref{4.6}) that, in the first instance, there is no loss of 
generality in supposing that
$$\varphi_j(t)=t^j+p^ct^{k+1}\psi_j(t)\quad (1\le j\le k),$$
for appropriate polynomials $\psi_j\in \dbZ[t]$. In view of the constraints (\ref{3.22}), we 
may substitute
$$v_l=p^bu_l+\eta\quad \text{and}\quad w_l=p^bz_l+\eta\quad (1\le l\le s-R),$$
and (\ref{3.21}) is transformed into the new system
\begin{align}
\sum_{i=1}^R\left(\varphi_j(x_i)-\varphi_j(y_i)\right)\equiv \sum_{l=1}^{s-R}
\left( \varphi_j(p^bu_l+\eta)-\varphi_j(p^bz_l+\eta)\right) &\mmod{p^B}\notag \\
&(1\le j\le k).\label{7.6}
\end{align}
Next, as in the argument of the proof of Lemma \ref{lemma4.1} leading from (\ref{4.6}) to 
(\ref{4.7}) and (\ref{4.8}), we may take appropriate integral linear combinations of the 
congruences comprising (\ref{7.6}) to confirm that there exist polynomials 
$\Psi_j\in \dbZ[t]$ for which
\begin{align*}
\sum_{i=1}^R\left(\Phi_j(x_i-\eta)-\Phi_j(y_i-\eta)\right)\equiv \sum_{l=1}^{s-R}\left( 
\Phi_j(p^bu_l)-\Phi_j(p^bz_l)\right) &\mmod{p^B}\\
&\ (1\le j\le k),
\end{align*}
in which
\begin{equation}\label{7.7}
\Phi_j(t)=t^j+p^ct^{k+1}\Psi_j(t).
\end{equation}
In particular, the system $\bfPhi$ is $p^c$-spaced, and one has
\begin{equation}\label{7.8}
\sum_{i=1}^R\Phi_j(x_i-\eta)\equiv \sum_{i=1}^R\Phi_j(y_i-\eta)\mmod{(p^{jb},p^B)}
\quad (1\le j\le k).
\end{equation}

\par In the second phase of our analysis, we seek to extract from (\ref{7.8}) a system 
resembling one of Vinogradov type. With a limited supply of variable names available, it will 
be expedient to recycle letters in circumstances where confusion is easily avoided. We recall 
from (\ref{3.22}) that in the solutions of (\ref{3.21}) of interest to us, we have 
$\bfx\equiv \bfy\equiv \xi\mmod{p^a}$. Thus we may substitute
\begin{equation}\label{7.9}
x_i=p^au_i+\xi\quad \text{and}\quad y_i=p^az_i+\xi\quad (1\le i\le R).
\end{equation}
We recall that $\xi-\eta=\ome p^\gam$, and that we may suppose that $B\ge (k-r+1)b$. 
Thus, by dropping mention in (\ref{7.8}) of those congruences with index $j$ satisfying 
$1\le j\le k-r$, and substituting (\ref{7.9}) into (\ref{7.8}), we deduce that
\begin{align}\label{7.10}
\sum_{i=1}^R\Phi_j(p^au_i+\ome p^\gam)\equiv \sum_{i=1}^R\Phi_j(p^az_i+\ome 
p^\gam)&\mmod{p^{(k-r+1)b}}\notag \\
&\ (k-r+1\le j\le k).
\end{align}

\par The situation in (\ref{7.10}) is somewhat akin to that encountered in the argument 
leading from (\ref{4.6}) to (\ref{4.7}) and (\ref{4.8}). We apply the binomial theorem within 
(\ref{7.7}). In this way, when $1\le l\le r$, we find that for suitable polynomials 
$\Tet_l\in \dbZ[t]$, one has
$$\Phi_{k-r+l}(p^ay+\ome p^\gam)-\Phi_{k-r+l}(\ome p^\gam)=\Ups_l(p^ay),$$
in which
$$\Ups_l(t)=\sum_{i=1}^r\Ome_{il}(\ome p^\gam)^{k-r+l-i}t^i+t^{r+1}\Tet_l(t),$$
and the integral coefficients $\Ome_{il}$ satisfy the congruence
$$\Ome_{il}\equiv \binom {k-r+l}{i}\mmod{p^c}\quad (1\le i,l\le r).$$
Much as in the argument of the proof of \cite[Lemma 3.2]{Woo2013}, our hypothesis $p>k$ 
ensures that the matrix $A=(\Ome_{il})_{1\le i,l\le r}$ satisfies 
$\det(A)\not\equiv 0\mmod{p}$, and hence that $A$ possesses a multiplicative inverse 
$A^{-1}$ modulo $p^{(k-r+1)b}$ having integral coefficients. Indeed, it is apparent that 
$(1!\cdot 2!\cdots r!)\det(A)$ is congruent modulo $p$ to
$$\det \left( (k-j+1)\cdots (k-j+1-(i-1))\right)_{1\le i,j\le r}=
\det \left( (k-j+1)^i\right)_{1\le i,j\le r},$$
and hence also to
$$\biggl( \prod_{l=1}^r(k-l+1)\biggr) \biggl( \prod_{1\le i<j\le r}\left( (k-j+1)-(k-i+1)
\right) \biggr) \not\equiv 0\mmod{p}.$$
Here, we have corrected an inconsequential oversight in the evaluation of a determinant in 
the proof of \cite[Lemma 3.2]{Woo2013}.\par

We now replace $\bfUps$ by $A^{-1}\bfUps$ and $\bfTet$ by $A^{-1}\bfTet$. This once 
again amounts to taking suitable integral linear combinations of the congruences comprising 
(\ref{7.10}). In this way, we see that there is no loss of generality in supposing that the 
coefficient matrix $A$ is equal to $I_r$. Since $(\ome,p)=1$, moreover, there is an integral 
multiplicative inverse $\ome^{-1}$ for $\ome$ modulo $p^{(k-r+1)b}$. Hence, there exist 
polynomials $\Xi_l\in \dbZ[t]$ having the property that whenever the system (\ref{7.10}) is 
satisfied, then
$$(\ome p^\gam)^{k-r}\sum_{i=1}^R(p^a)^l\left( \Psi_l(u_i)-\Psi_l(z_i)\right) \equiv 0
\mmod{p^{(k-r+1)b}}\quad (1\le l\le r),$$
in which
\begin{equation}\label{7.11}
\Psi_l(t)=t^l+p^{a-(k-r)\gam}\Xi_l(t).
\end{equation}
Notice here that, since we may suppose from (\ref{7.1}) that $a\ge \del \tet$, and 
$\gam\le \nu$, our hypothesis concerning the hierarchy (\ref{6.2}) combines with (\ref{6.5}) 
and (\ref{6.10}) to give
$$k\gam\le k\lceil 4\eps H\Lam^{-1}\rceil <\del^2\lceil \mu H\rceil \le \del a,$$
and hence
$$a-(k-r)\gam>(1-\del)a\ge \del(1-\del)\mu H>\tau B.$$
The exponent of $p$ on the right hand side of (\ref{7.11}) is therefore positive, and indeed 
the system $\bfPsi$ is $p^c$-spaced for some $c>\tau(k-r+1)b$.\par

The integral solutions $\bfx=p^a\bfu+\xi$, $\bfy=p^a\bfz+\xi$, $\bfv$, $\bfw$ of the 
original system of congruences (\ref{3.21}), counted with weight (\ref{3.23}) associated 
with the definition (\ref{3.20}) of $K_{a,b,c}^{r,\bfvarphi,\nu}(\bfgra;\xi,\eta)$, are 
therefore constrained by the additional system of congruences
\begin{equation}\label{7.12}
\sum_{i=1}^R\Psi_l(u_i)\equiv \sum_{i=1}^R\Psi_l(z_i)\mmod{p^{B^\prime}}\quad 
(1\le l\le r),
\end{equation}
in which $B^\prime$ is defined by (\ref{7.3}).\par

Our goal at this point is to apply the inductive hypothesis, meaning the conclusion of 
Theorem \ref{theorem3.1} with $k$ replaced by $r$, so as to extract congruence 
information from the relations (\ref{7.12}). We recall at this point that in view of 
(\ref{7.5}), we may suppose that $B^\prime >\nu$. Write 
$H^\prime =\lceil B^\prime /r\rceil $. Then in view of (\ref{7.2}) and (\ref{7.3}), one has
\begin{equation}\label{7.13}
b^\prime-H^\prime\le a+1+(k-r)\gam/r.
\end{equation}
We note from (\ref{6.5}) that $B'>4\eps H\Lam^{-1}$, so our hypotheses concerning $B$ 
and the hierarchy (\ref{6.2}) ensure that $B^\prime$ is sufficiently large in terms of $k$, 
$\Lam$, $\mu$, $\tau$, $\del$ and $\eps$. We reinterpret the definition (\ref{3.4}) of 
$\grf_h(\bfalp;\xi)$ in the shape (\ref{4.2}), so that
\begin{equation}\label{7.14}
\grf_a(\bfalp;\xi)=\rho_a(\xi)^{-1}\sum_{y\in \dbZ}\grc_y(\bfalp),
\end{equation}
where
$$\grc_y(\bfalp)=\gra_{p^ay+\xi}e(\psi(p^ay+\xi;\bfalp))\quad (y\in \dbZ),$$
and $\psi(n;\bfalp)$ is defined by (\ref{3.3}). Notice that
\begin{equation}\label{7.15}
\rho_0(1;\bfgrc)^2=\sum_{y\in \dbZ}|\grc_y(\bfalp)|^2=\sum_{y\in \dbZ}
|\gra_{p^ay+\xi}|^2=\rho_a(\xi)^2.
\end{equation}
Finally, define the exponential sum
\begin{equation}\label{7.16}
\grg_\bfgrc(\bfalp,\bfbet)=\rho_0(1;\bfgrc)^{-1}\sum_{y\in \dbZ}\grc_y(\bfalp) 
e(\bet_1\Psi_1(y)+\ldots +\bet_r\Psi_r(y)),
\end{equation}
and the mean value
$$J(\bfalp)=\oint_{p^{B^\prime}}|\grg_\bfgrc(\bfalp,\bfbet)|^{2R}\d\bfbet .$$

\par By orthogonality, the mean value $J(\bfalp)$ counts the integral solutions of the 
system of congruences (\ref{7.12}), with each solution $\bfu$, $\bfz$ being counted with 
weight
\begin{align*}
&\rho_0(1;\bfgrc)^{-2R}\prod_{i=1}^R\grc_{u_i}(\bfalp){\overline \grc_{z_i}(\bfalp)}\\
&=\rho_a(\xi)^{-2R}\biggl( \prod_{i=1}^R\gra_{p^au_i+\xi}{\overline \gra_{p^az_i+\xi}}
\biggr) e\biggl( \sum_{i=1}^R \left( \psi(p^au_i+\xi;\bfalp)-\psi(p^az_i+\xi;\bfalp )
\right)\biggr) .
\end{align*}
But the conditions (\ref{7.12}) on $\bfx=p^a\bfu+\xi$ and $\bfy=p^a\bfz+\xi$ are implied 
by (\ref{3.21}), as we have seen. Thus, on noting that from (\ref{7.14})-(\ref{7.16}), one 
has
$$\grg_\bfgrc(\bfalp,{\mathbf 0})=\grf_a(\bfalp;\xi),$$
we deduce that
\begin{equation}\label{7.17}
\oint_{p^B}|\grf_a(\bfalp;\xi)^{2R}\grf_b(\bfalp;\eta)^{2s-2R}|\d\bfalp=
\oint_{p^B}\oint_{p^{B^\prime}}|\grg_\bfgrc(\bfalp,\bfbet)^{2R}
\grf_b(\bfalp;\eta)^{2s-2R}|\d\bfbet \d\bfalp,
\end{equation}
whence
\begin{equation}\label{7.18}
K_{a,b,c}^{r,\bfvarphi,\nu}(\bfgra;\xi,\eta)=
\oint J(\bfalp)|\grf_b(\bfalp;\eta)|^{2s-2R}\d\bfalp .
\end{equation}
The point here is that the inner integral over $\bfbet$ on the right hand side of (\ref{7.17}) 
is essentially redundant, since the integral over $\bfalp$ already restricts the variables 
underlying the term $|\grg_\bfgrc(\bfalp,\bfbet)|^{2R}$ to satisfy the system of 
congruences (\ref{7.12}).\par

The mean value $J(\bfalp)$ is of the type estimated in Corollary \ref{corollary3.2} to 
Theorem \ref{theorem3.1}. The system $\bfPsi$ is $p^c$-spaced for some 
$c\ge \tau B^\prime$ and $B^\prime$ is sufficiently large in terms of $k$, $\tau$ and 
$\eps^2$. Hence
\begin{align}
J(\bfalp)&=U_{R,r}^{B^\prime}(\bfgrc)\ll p^{B^\prime \eps^2}
U_{R,r}^{B^\prime,H^\prime}(\bfgrc)\notag \\
&=p^{B^\prime \eps^2}\rho_0(1;\bfgrc)^{-2}\sum_{\zet\nmod{p^{H^\prime}}}
\rho_{H^\prime}(\zet;\bfgrc)^2\oint_{p^{B^\prime}}|\grg_{\bfgrc^\prime}(\bfalp,
\bfbet)|^{2R}\d\bfbet ,\label{7.19}
\end{align}
in which the sequence $\bfgrc^\prime =(\grc^\prime_n(\bfalp))_{n\in \dbZ}$ is defined by 
putting
$$\grc^\prime_n(\bfalp)=\begin{cases} \grc_n(\bfalp),&\text{when 
$n\equiv \zet\mmod{p^{H^\prime}}$},\\
0,&\text{otherwise}.\end{cases}$$
We have $\rho_0(1;\bfgrc)=\rho_a(\xi)$. Also, on writing $\kap=p^a\zet+\xi$, we see that
$$\rho_{H^\prime}(\zet;\bfgrc)^2=\sum_{y\equiv \zet\mmod{p^{H^\prime}}}
|\gra_{p^ay+\xi}|^2=\sum_{n\equiv \kap\mmod{p^{a+H^\prime}}}|\gra_n|^2
=\rho_{a+H^\prime}(\kap)^2.$$
Thus, on substituting (\ref{7.19}) into (\ref{7.18}), we conclude that
\begin{align}
&\rho_a(\xi)^2K_{a,b,c}^{r,\bfvarphi,\nu}(\bfgra;\xi,\eta)\notag \\
&\ll p^{B^\prime \eps^2}\sum_{\substack{\kap\nmod{p^{a+H^\prime}}\\ 
\kap\equiv \xi\mmod{p^a}}}\rho_{a+H^\prime}(\kap)^2
\oint_{p^B}\oint_{p^{B^\prime}}|\grg_{\bfgrc^\prime}(\bfalp,\bfbet)^{2R}
\grf_b(\bfalp;\eta)^{2s-2R}|\d\bfbet \d\bfalp .\label{7.20}
\end{align}

\par We next unravel the mean value occurring on the right hand side of (\ref{7.20}), 
reversing our previous course. The mean value
$$\oint_{p^{B^\prime}}|\grg_{\bfgrc^\prime}(\bfalp,\bfbet)|^{2R}\d\bfbet $$
counts the integral solutions of the system of congruences (\ref{7.12}), with each solution 
$\bfu$, $\bfz$ being counted with weight
\begin{align*}
&\rho_0(1;\bfgrc^\prime)^{-2R}\prod_{i=1}^R\grc^\prime_{u_i}(\bfalp)
{\overline \grc^\prime_{z_i}(\bfalp)}\\
&=\rho_{a+H^\prime}(\kap)^{-2R}\biggl( \prod_{i=1}^R\gra_{p^au_i+\xi}
{\overline \gra_{p^az_i+\xi}}\biggr) 
e\biggl( \sum_{i=1}^R \left( \psi(p^au_i+\xi;\bfalp)-\psi(p^az_i+\xi;\bfalp )
\right)\biggr) ,
\end{align*}
but now subject to the constraint $\bfu\equiv \bfz\equiv \zet\mmod{p^{H^\prime}}$. The 
conditions (\ref{7.12}) on $\bfx=p^a\bfu+\xi$ and $\bfy=p^a\bfz+\xi$ are again implied by 
(\ref{3.21}). Then since
$$\grg_{\bfgrc^\prime}(\bfalp;{\mathbf 0})=\grf_{a+H^\prime}(\bfalp;\kap),$$
we deduce that
\begin{align}
\oint_{p^B}\oint_{p^{B^\prime}}|\grg_{\bfgrc^\prime}(\bfalp,\bfbet)^{2R}&
\grf_b(\bfalp;\eta)^{2s-2R}|\d\bfbet \d\bfalp \notag \\
&=\oint_{p^B}|\grf_{a+H^\prime}(\bfalp;\kap)^{2R}\grf_b(\bfalp;\eta)^{2s-2R}|
\d\bfalp .\label{7.21}
\end{align}
Moreover, on making use of Lemma \ref{lemma6.2}, one sees that
\begin{equation}\label{7.22}
\rho_{a+H^\prime}(\kap)^2|\grf_{a+H^\prime}(\bfalp;\kap)|^{2R}\le 
\left( p^{b^\prime -(a+H^\prime)}\right)^R 
\sum_{\substack{\xi^\prime \nmod{p^{b^\prime}}\\ \xi^\prime \equiv 
\kap\mmod{p^{a+H^\prime}}}}\rho_{b^\prime}(\xi^\prime)^2
|\grf_{b^\prime}(\bfalp;\xi^\prime)|^{2R}.
\end{equation}
Thus, on substituting (\ref{7.22}) into (\ref{7.21}) and thence into (\ref{7.20}), we obtain 
the bound
\begin{equation}\label{7.23}
\rho_a(\xi)^2K_{a,b,c}^{r,\bfvarphi,\nu}(\bfgra;\xi,\eta)\ll 
p^{B^\prime \eps^2+R(b^\prime -a-H^\prime)}
\sum_{\substack{\xi^\prime \nmod{p^{b^\prime}}\\ \xi^\prime \equiv \xi\mmod{p^a}}}
\rho_{b^\prime}(\xi^\prime)^2K_{b^\prime,b,c}^{r,\bfvarphi,\nu}(\bfgra;\xi^\prime,\eta).
\end{equation}
Next, substituting (\ref{7.23}) into (\ref{3.19}), we conclude that
$$\rho_0^4K_{a,b,c}^{r,\bfvarphi,\nu}(\bfgra)\ll 
p^{B^\prime \eps^2+R(b^\prime -a-H^\prime )}
\sum_{\xi^\prime \nmod{p^{b^\prime}}}\sum_{\substack{\eta\nmod{p^b}\\ 
\eta \not \equiv \xi^\prime \mmod{p^\nu}}}\rho_{b^\prime}(\xi^\prime)^2\rho_b(\eta)^2
K_{b^\prime,b,c}^{r,\bfvarphi,\nu}(\bfgra;\xi^\prime,\eta),$$
whence
\begin{equation}\label{7.24}
K_{a,b,c}^{r,\bfvarphi,\nu}(\bfgra)\ll 
p^{B^\prime \eps^2+R(b^\prime -a-H^\prime )}K_{b^\prime,b,c}^{r,\bfvarphi,\nu}
(\bfgra).
\end{equation}

\par In order to complete the proof of the lemma, it remains only to note that, from 
(\ref{7.13}), one has
\begin{align*}
R(b^\prime -a-H^\prime)&= r(r+1)(b^\prime-a-H^\prime)/2\\
&\le (r+1)(r+(k-r)\gam)/2<k^2\nu/2.
\end{align*}
Thus, in view of the definition (\ref{6.5}) and our hierarchy (\ref{6.2}),
$$B^\prime \eps^2+R(b^\prime -a-H^\prime)<\nu +k^2\nu/2<k^2\nu .$$
The conclusion of the lemma therefore follows from (\ref{7.24}) in this second situation with 
$B^\prime >\nu$.
\end{proof}

\section{The iterative step} The work of the previous section shows how to generate 
powerful congruence constraints on the variables underlying the mean value 
$K_{a,b,c}^{r,\bfvarphi,\nu}(\bfgra)$. As in earlier efficient congruencing methods, we 
must now interchange the roles of the two sets of variables so that these congruence 
constraints may be employed anew to generate yet stronger constraints. In our earlier 
multigrade method, this was achieved by employing H\"older's inequality in a somewhat 
greedy manner. On this occasion, we take a slightly more measured approach, though the 
underlying ideas remain the same. At this point, the parameters $\bfvarphi$, $\nu$ and $c$ 
have ceased to possess any particular significance, and we omit mention of them in our 
various notations. Also, throughout this section, we suppose that $a$, $b$ and $r$ are 
integers satisfying (\ref{7.1}), and we define $b^\prime$ via (\ref{7.2}).

\begin{lemma}\label{lemma8.1} When $r\ge 2$, one has
\begin{equation}\label{8.1}
K_{a,b}^r(\bfgra)\ll p^{k^2\nu}K_{b,b^\prime}^{k-r}(\bfgra)^{\tfrac{1}{k-r+1}}
K_{b^\prime,b}^{r-1}(\bfgra)^{\tfrac{k-r}{k-r+1}}.
\end{equation}
Meanwhile, when $r=1$, one instead has
\begin{equation}\label{8.2}
K_{a,b}^1(\bfgra)\ll p^{k^2\nu}K_{b,kb}^{k-1}(\bfgra)^{1/k}
U_{s,k}^{B,b}(\bfgra)^{1-1/k}.
\end{equation}
\end{lemma}

\begin{proof} Observe that
\begin{align*}
k(k+1)-r(r+1)&=(k+r)(k-r)+(k-r)\\
&=\frac{(k-r)(k-r+1)}{k-r+1}+\left( k(k+1)-r(r-1)\right) \cdot \frac{k-r}{k-r+1}.
\end{align*}
Then for any integers $\zet$ and $\eta$, it is a consequence of H\"older's inequality that 
when $1\le r\le k-1$, one has
$$\oint_{p^B}|\grf_{b^\prime}(\bfalp;\zet)^{r(r+1)}\grf_b(\bfalp;\eta)^{k(k+1)-r(r+1)}|
\d\bfalp \le U_1^{\tfrac{1}{k-r+1}}U_2^{\tfrac{k-r}{k-r+1}},$$
where
$$U_1=\oint_{p^B}|\grf_b(\bfalp;\eta)^{(k-r)(k-r+1)}\grf_{b^\prime}(\bfalp;\zet)^{
k(k+1)-(k-r)(k-r+1)}|\d\bfalp $$
and
$$U_2=\oint_{p^B}|\grf_{b^\prime}(\bfalp;\zet)^{r(r-1)}\grf_b(\bfalp;\eta)^{
k(k+1)-r(r-1)}|\d\bfalp .$$
Thus, on recalling (\ref{3.19}) and (\ref{3.20}), we deduce first that
$$K_{b^\prime,b}^r(\bfgra;\zet,\eta)\le K_{b,b^\prime}^{k-r}(\bfgra;\eta,\zet)^{
\tfrac{1}{k-r+1}}K_{b^\prime,b}^{r-1}(\bfgra;\zet,\eta)^{\tfrac{k-r}{k-r+1}},$$
and hence, by another application of H\"older's inequality, that
\begin{equation}\label{8.3}
K_{b^\prime,b}^r(\bfgra)\le K_{b,b^\prime}^{k-r}(\bfgra)^{\tfrac{1}{k-r+1}}
K_{b^\prime,b}^{r-1}(\bfgra)^{\tfrac{k-r}{k-r+1}}.
\end{equation}
Since Lemma \ref{lemma7.1} shows that $K_{a,b}^r(\bfgra)\ll p^{k^2\nu}
K_{b^\prime,b}^r(\bfgra)$, the conclusion (\ref{8.1}) of the lemma is immediate from 
(\ref{8.3}) in the case $r\ge 2$.\par

In order to handle the case $r=1$, we begin by noting that from (\ref{3.19}) and 
(\ref{3.20}), one has
$$K_{b^\prime,b}^0(\bfgra)=\rho_0^{-4}\sum_{\xi\nmod{p^{b^\prime}}}
\sum_{\substack{\eta\nmod{p^b}\\ \xi\not \equiv \eta\mmod{p^\nu}}}\rho_{b^\prime}
(\xi)^2\rho_b(\eta)^2\oint_{p^B}|\grf_b(\bfalp;\eta)|^{2s}\d\bfalp .$$
Since it follows from (\ref{3.8}) that
$$\sum_{\eta\nmod{p^b}}\rho_b(\eta)^2\oint_{p^B}|\grf_b(\bfalp;\eta)|^{2s}
\d\bfalp =\rho_0^2U_{s,k}^{B,b}(\bfgra)$$
and
$$\sum_{\xi\nmod{p^{b^\prime}}}\rho_{b^\prime}(\xi)^2=\rho_0^2,$$
it follows that $K_{b^\prime,b}^0(\bfgra)=U_{s,k}^{B,b}(\bfgra)$, and so the desired 
conclusion (\ref{8.2}) again follows from (\ref{8.3}) when $r=1$.
\end{proof}

We next interpret the conclusion of Lemma \ref{lemma8.1} in terms of the anticipated 
order of magnitude normalisation defined in (\ref{3.24}).

\begin{lemma}\label{lemma8.2} When $r\ge 2$, one has
\begin{equation}\label{8.4}
\llbracket K_{a,b}^r(\bfgra)\rrbracketsub{\Lambda}\ll p^{k^2\nu}\llbracket 
K_{b,b^\prime }^{k-r}(\bfgra)\rrbracketsub{\Lam}^{1/(k-r+1)}
\llbracket K_{b^\prime,b}^{r-1}(\bfgra)\rrbracketsub{\Lam}^{1-1/r}.
\end{equation}
Meanwhile, when $r=1$ one instead has
\begin{equation}\label{8.5}
\llbracket K_{a,b}^1(\bfgra)\rrbracketsub{\Lam}\ll p^{2k^2\nu}\llbracket K_{b,kb}^{k-1}
(\bfgra)\rrbracketsub{\Lam}^{1/k}(p^{-b})^{\Lam (1-1/k)}.
\end{equation}
\end{lemma}

\begin{proof} By reference to (\ref{3.24}), we deduce from Lemma \ref{lemma8.1} that 
for $r\ge 2$, one has
$$\llbracket K_{a,b}^r(\bfgra)\rrbracketsub{\Lam}\ll 
(p^{k^2\nu})^{\textstyle{\frac{k-1}{r(k-r)}}}V_1^{\textstyle{\frac{1}{k-r+1}}}
V_2^{\textstyle{\frac{k-r}{k-r+1}}},$$
where
$$V_1=\left( \frac{K_{b,b^\prime}^{k-r}(\bfgra)}{p^{\Lam H}U_{s,k}^{B,H}(\bfgra)}
\right)^{\textstyle{\frac{k-1}{r(k-r)}}}=\llbracket K_{b,b^\prime}^{k-r}(\bfgra)
\rrbracketsub{\Lam},$$
and
$$V_2=\left( \frac{K_{b^\prime,b}^{r-1}(\bfgra)}{p^{\Lam H}U_{s,k}^{B,H}(\bfgra)}
\right)^{\textstyle{\frac{k-1}{r(k-r)}}}=\llbracket K_{b^\prime,b}^{r-1}(\bfgra)
\rrbracketsub{\Lam}^{\textstyle{\left(\frac{r-1}{r}\right) \left( \frac{k-r+1}{k-r}\right) }}.
$$
The first conclusion (\ref{8.4}) is now immediate for $r\ge 2$. Here, we have made use of 
the elementary fact that
$$\max_{1\le r\le k-1}\frac{k-1}{r(k-r)}=1.$$
\par

When $r=1$, on the other hand, one finds in like manner that
\begin{equation}\label{8.6}
\llbracket K_{a,b}^1(\bfgra)\rrbracketsub{\Lam} \ll p^{k^2\nu}V_3^{1/k}V_4^{1-1/k},
\end{equation}
where
\begin{equation}\label{8.7}
V_3=\left( \frac{K_{b,kb}^{k-1}(\bfgra)}{p^{\Lam H}U_{s,k}^{B,H}(\bfgra)}\right) 
=\llbracket K_{b,kb}^{k-1}(\bfgra)\rrbracketsub{\Lam}
\end{equation}
and
$$V_4=\frac{U_{s,k}^{B,b}(\bfgra)}{p^{\Lam H}U_{s,k}^{B,H}(\bfgra)}.$$
In view of the upper bound (\ref{6.4}), one has
$$U_{s,k}^{B,b}(\bfgra)\le (p^{H-b})^{\Lam+\eps}U_{s,k}^{B,H}(\bfgra),$$
and hence, on recalling (\ref{6.1}) and (\ref{6.5}), one finds that
$$V_4\le p^{(H-b)\eps -\Lam b}\le p^{\Lam \nu -\Lam b}\le p^{s\nu-\Lam b}.$$
On substituting this bound together with (\ref{8.7}) into (\ref{8.6}), we conclude that
$$\llbracket K_{a,b}^1(\bfgra)\rrbracketsub{\Lam}\ll p^{(k^2+s)\nu}\llbracket 
K_{b,kb}^{k-1}(\bfgra)\rrbracketsub{\Lam}^{1/k}(p^{-b})^{\Lam (1-1/k)},$$
and the conclusion (\ref{8.5}) of the lemma follows on recalling that
$$s=k(k+1)/2\le k^2.$$
\end{proof}

\section{Distilling multigrade combinations into monograde processes} We turn our 
attention next to the problem of analysing the impact of applying Lemma \ref{lemma8.2} 
iteratively so as to estimate $K_{a,b}^r(\bfgra)$ in terms of a tree of possible outcomes. 
Such processes seem, at first appearance, difficult to control. However, the multigrade 
efficient congruencing method of \cite{Woo2015, Woo2016, Woo2017} offers the tools to 
accommodate such an analysis. In broad terms, we weight the possible outcomes in such a 
manner that one can follow any single path through the tree, and compute the outcome 
without reference to the multitude of alternate paths available.\par

In this section and the next, in the interests of concision, we write
$$\Ktil_{a,b}^r=\llbracket K_{a,b}^r(\bfgra)\rrbracketsub{\Lam}.$$
Also, when $1\le j\le k-1$, we write
$$\rho_j=\frac{j}{k-j+1}\quad \text{and}\quad b_j=\left\lceil \frac{b}{\rho_j}\right\rceil .$$

\begin{lemma}\label{lemma9.1} Suppose that
\begin{equation}\label{9.1}
1\le r\le k-1,\quad a\ge \del \tet,\quad b\ge k\del \tet\quad \text{and}\quad ra\le (k-r+1)b.
\end{equation}
Then whenever $kb\le B$, one has
\begin{equation}\label{9.2}
\Ktil_{a,b}^r\ll p^{(r+1)k^2\nu}(p^{-b})^{(1-1/k)\Lam/r}
\prod_{j=1}^r\left( \Ktil_{b,b_j}^{k-j}\right)^{\rho_j/r}.
\end{equation}
\end{lemma}

\begin{proof} We establish (\ref{9.2}) by induction on $r$. Observe that in the case 
$r=1$, it follows from Lemma \ref{lemma8.2} that
$$\Ktil_{a,b}^1\ll p^{2k^2\nu}(\Ktil_{b,kb}^{k-1})^{1/k}(p^{-b})^{(1-1/k)\Lam},$$
and this delivers (\ref{9.2}) in this initial case. Note here that the hypotheses (\ref{9.1}) 
together with the assumption $kb\le B$ imply the corresponding hypotheses (\ref{7.1}) 
implicitly assumed in the statement of Lemma \ref{lemma8.2}.\par

Suppose next that when $kb\le B$, the estimate (\ref{9.2}) has been confirmed for all 
indices $r<\Rho$, for some integer $\Rho$ with $2\le \Rho\le k-1$. In these circumstances, 
it follows from Lemma \ref{lemma8.2} that
\begin{equation}\label{9.3}
\Ktil_{a,b}^\Rho\ll p^{k^2\nu}(\Ktil_{b,b_\Rho}^{k-\Rho})^{\rho_\Rho/\Rho}
(\Ktil_{b_\Rho,b}^{\Rho-1})^{1-1/\Rho},
\end{equation}
with
$$b_\Rho=\left\lceil \frac{k-\Rho+1}{\Rho}b\right\rceil \ge \frac{2b}{k}>\del \tet .$$
The latter lower bound, together with the upper bound
$$(\Rho-1)b_\Rho \le (\Rho-1)\biggl( \frac{k-\Rho+1}{\Rho}b+1\biggr) 
<(k-\Rho+2)b,$$
ensures that the conditions are satisfied permitting the inductive hypothesis (\ref{9.2}) to 
be deployed to estimate $\Ktil_{b_\Rho,b}^{\Rho-1}$. Thus, we have
$$\Ktil_{b_\Rho,b}^{\Rho-1}\ll p^{\Rho k^2\nu}
(p^{-b})^{(1-1/k)\Lam/(\Rho-1)}\prod_{j=1}^{\Rho-1} 
(\Ktil_{b,b_j}^{k-j})^{\rho_j/(\Rho-1)}.$$
On substituting this bound into (\ref{9.3}), we deduce that
$$\Ktil_{a,b}^\Rho \ll p^{\Rho k^2\nu}(p^{-b})^{(1-1/k)\Lam/\Rho}
(\Ktil_{b,b_\Rho}^{k-\Rho})^{\rho_\Rho/\Rho}\prod_{j=1}^{\Rho-1} 
(\Ktil_{b,b_j}^{k-j})^{\rho_j/\Rho},$$
and the inductive hypothesis (\ref{9.2}) follows when $r=\Rho$. The desired conclusion 
(\ref{9.2}) therefore follows by induction for $1\le r\le k-1$.
\end{proof}

An immediate consequence of Lemma \ref{lemma9.1} provides an intermediate 
monograde iterative relation.

\begin{lemma}\label{lemma9.2} Suppose that $1\le r\le k-1$ and $kb\le B$. Suppose further 
that the conditions (\ref{9.1}) hold. Then there 
exists an integer $r^\prime$ with $1\le r^\prime \le r$ having the property that
$$\Ktil_{a,b}^r\ll (\Ktil_{b,b_{r^\prime}}^{k-r^\prime})^{\rho_{r^\prime}}
(p^{-b})^{\Lam/(2k)}.$$
\end{lemma}

\begin{proof} By employing the elementary inequality
$$|z_1\cdots z_n|\le |z_1|^n+\ldots +|z_n|^n,$$
it follows from the relation (\ref{9.2}) of Lemma \ref{lemma9.1} that
$$\Ktil_{a,b}^r\ll p^{(r+1)k^2\nu }(p^{-b})^{(1-1/k)\Lam/r}\sum_{j=1}^r
(\Ktil_{b,b_j}^{k-j})^{\rho_j}.$$
Thus we deduce that there exists an integer $r^\prime$ with $1\le r^\prime\le r$ such that
\begin{equation}\label{9.4}
\Ktil_{a,b}^r\ll p^{(r+1)k^2\nu}(p^{-b})^{(1-1/k)\Lam /r}
(\Ktil_{b,b_{r^\prime}}^{k-r^\prime})^{\rho_{r^\prime}}.
\end{equation}
Note that when $1\le r\le k-1$, one has $(1-1/k)/r\ge 1/k$. Moreover, in view of 
(\ref{6.5}), (\ref{6.10}), (\ref{9.1}) and the hierarchy (\ref{6.2}), one has
$$b\Lam/k\ge \del \tet \Lam \ge \del \mu H\Lam>2k^3\nu\ge 2(r+1)k^2\nu.$$
Thus we infer that
$$p^{(r+1)k^2\nu}(p^{-b})^{(1-1/k)\Lam/r}\le (p^{-b})^{\Lam/(2k)}.$$
The conclusion of Lemma \ref{lemma9.2} follows by substituting this bound into 
(\ref{9.4}).
\end{proof}

The estimate supplied by Lemma \ref{lemma9.2} bounds $\Ktil_{a,b}^r$ in terms of 
$\Ktil_{b,b_{r^\prime}}^{k-r^\prime}$, for a suitable integer $r^\prime$ with 
$1\le r^\prime\le r$. It is possible that the parameter $b_{r^\prime}$ might be 
substantially smaller than $b$, and thus one might fear that continued iteration of such a 
relation might be limited by this shrinking parameter. At this point, we take the opportunity 
to dispell such fears once and for all.

\begin{lemma}\label{lemma9.3} Suppose that
\begin{equation}\label{9.5}
1\le r\le k-1,\quad a\ge \del \tet,\quad b\ge k^2\del \tet\quad \text{and}\quad 
ra\le (k-r+1)b.
\end{equation}
Then, whenever $k^2b\le B$, there exist integers $a^\prime$, $b^\prime$, $r^\prime$, 
and there exists a positive number $\rho$, having the property that
\begin{equation}\label{9.6}
\Ktil_{a,b}^r\ll (\Ktil_{a^\prime,b^\prime}^{r^\prime})^\rho p^{-b\Lam /(2k)},
\end{equation}
with
\begin{equation}\label{9.7}
a^\prime \ge \del \tet,\quad b^\prime \ge k^2\del \tet,\quad r^\prime a^\prime 
\le (k-r^\prime +1)b^\prime ,
\end{equation}
\begin{equation}\label{9.8}
1\le r^\prime \le k-1,\quad 0<\rho<(1-1/k)^2,
\end{equation}
\begin{equation}\label{9.9}
(1+2/k)b\le b^\prime \le k^2b,\quad b^\prime =\left\lceil \frac{r^\prime+1}{k-r^\prime}
a^\prime \right\rceil,\quad \rho b'\ge b.
\end{equation} 
\end{lemma}

\begin{proof} We recall that we are permitted the assumption that $a$, $b$, $r$ satisfy 
(\ref{9.5}). Thus, as a consequence of Lemma \ref{lemma9.2}, there exists an integer 
$r_1$ with $1\le r_1\le r$ having the property that
\begin{equation}\label{9.10}
\Ktil_{a,b}^r\ll (\Ktil_{b,b_{r_1}}^{k-r_1})^{\rho_{r_1}}(p^{-b})^{\Lam/(2k)},
\end{equation}
where
$$b_{r_1}=\left\lceil \frac{k-r_1+1}{r_1}b\right\rceil \quad \text{and}\quad \rho_{r_1}
=\frac{r_1}{k-r_1+1}.$$
Notice here that $b\ge k^2\del \tet\ge \del \tet$,
$$b_{r_1}\ge b/k\ge k\del \tet \quad \text{and}\quad b_{r_1}\le kb\le k^{-1}B.$$
Moreover, one has
$$(k-r_1)b\le (r_1+1)\left\lceil \frac{k-r_1+1}{r_1}b\right\rceil =(r_1+1)b_{r_1}.$$
We are therefore at liberty to apply Lemma \ref{lemma9.2} to estimate $\Ktil_{b,b_{r_1}}
^{k-r_1}$.
Thus, there exists an integer $r_2$ with $1\le r_2\le k-r_1$ having the property that
\begin{equation}\label{9.11}
\Ktil_{b,b_{r_1}}^{k-r_1}\ll (\Ktil_{b_{r_1},b_{r_2}}^{k-r_2})^{\rho_{r_2}}
(p^{-b_{r_1}})^{\Lam/(2k)},
\end{equation}
where
$$b_{r_2}=\left\lceil \frac{k-r_2+1}{r_2}b_{r_1}\right\rceil \quad \text{and}\quad 
\rho_{r_2}=\frac{r_2}{k-r_2+1}.$$

\par On substituting (\ref{9.11}) into (\ref{9.10}), we find that
\begin{equation}\label{9.12}
\Ktil_{a,b}^r\ll (\Ktil_{a^\prime,b^\prime}^{r^\prime})^\rho p^{-\sig \Lam/(2k)},
\end{equation}
where
$$a^\prime =b_{r_1},\quad b^\prime =b_{r_2},\quad r^\prime =k-r_2,\quad 
\rho=\rho_{r_1}\rho_{r_2}\quad \text{and}\quad \sig=b+\rho_{r_1}b_{r_1}.$$
Since $\sig\ge b$, the upper bound (\ref{9.12}) will deliver the conclusion of the lemma 
provided that we are able to verify the conditions (\ref{9.7})-(\ref{9.9}), a matter that 
we now address.\par

Observe first that since $1\le r_2\le k-r_1$ and $1\le r_1\le k-1$, one has
$$1\le r_1\le r^\prime =k-r_2\le k-1.$$
Moreover,
$$\rho=\rho_{r_1}\rho_{r_2}=\frac{r_1}{k-r_1+1}\cdot \frac{r_2}{k-r_2+1}\le 
\frac{r_1}{k-r_1+1}\cdot \frac{k-r_1}{r_1+1}.$$
Thus
$$\rho\le \frac{r_1}{r_1+1}\cdot \frac{k-r_1}{k-r_1+1}<(1-1/k)^2.$$
We have therefore verified that the conditions (\ref{9.8}) are satisfied.\par

Next, we have
$$b^\prime=b_{r_2}\le kb_{r_1}\le k^2b,$$
and since $1\le r_2\le k-r_1$, we also have
\begin{align*}
b^\prime &=b_{r_2}\ge \frac{k-r_2+1}{r_2}b_{r_1}\ge \frac{r_1+1}{k-r_1}b_{r_1}\\
&\ge \frac{r_1+1}{k-r_1}\cdot \frac{k-r_1+1}{r_1}b=\frac{r_1+1}{r_1}\cdot 
\frac{k-r_1+1}{k-r_1}b\\
&\ge (k/(k-1))^2b>(1+2/k)b.
\end{align*}
Thus $(1+2/k)b\le b^\prime \le k^2b$. Also, since $r^\prime =k-r_2$ and 
$a^\prime =b_{r_1}$, one finds that
$$\left\lceil \frac{r^\prime+1}{k-r^\prime}a^\prime \right\rceil =\left\lceil 
\frac{k-r_2+1}{r_2}b_{r_1}\right\rceil =b^\prime .$$
In addition, we have
$$\frac{b^\prime}{b}\ge \frac{b_{r_1}/\rho_{r_2}}{b}\ge 
\frac{b/\rho_{r_1}}{\rho_{r_2}b}=\frac{1}{\rho_{r_1}\rho_{r_2}}=\frac{1}{\rho}.$$
Then the conditions (\ref{9.9}) are satisfied.\par

Finally, we have
$$a^\prime =b_{r_1}=\left\lceil \frac{k-r_1+1}{r_1}b\right\rceil \ge \frac{b}{k}
\ge k\del \tet$$
and
$$b^\prime \ge (1+2/k)b\ge k^2\del \tet.$$
Meanwhile,
$$r^\prime a^\prime =(k-r_2)b_{r_1}<(k-r_2+1)b_{r_1},$$
whilst
$$(k-r^\prime +1)b^\prime =(r_2+1)b_{r_2}\ge \frac{k-r_2+1}{r_2}(r_2+1)b_{r_1},$$
so that $r^\prime a^\prime \le (k-r^\prime+1)b^\prime$. This confirms that the conditions 
(\ref{9.7}) are satisfied. All of the conditions (\ref{9.7})-(\ref{9.9}) having been 
confirmed, the proof of the lemma is complete.
\end{proof}

\section{The proof of Theorem \ref{theorem3.1}} Our apparatus for the main $p$-adic 
concentration argument has now been assembled, in the shape of Lemma \ref{lemma9.3}. 
The lower bound (\ref{6.3}) may be exploited to show that an initial mean value 
$\Ktil_{\tet,\tet}^1$ is large. Then, whenever $\Ktil_{a,b}^r$ is large, the estimate 
(\ref{9.6}) of Lemma \ref{lemma9.3} shows that a related mean value 
$\Ktil_{a^\prime,b^\prime}^{r^\prime}$ is larger still, and inflated by an additional factor 
$p^{b\Lam/(2\rho k)}$. After iteration of this idea, this overabundance of ``$p$-adic 
energy'' blows up to the point that it exceeds even the trivial estimate supplied by Lemma 
\ref{lemma4.2}, delivering a contradiction to the hypothesis that $\Lam>0$.

\begin{proof}[The proof of Theorem \ref{theorem3.1}] Throughout, we consider a 
natural number $k$ and we put $s=k(k+1)/2$. In view of Lemma \ref{lemma5.1}, we may 
suppose that $k\ge 2$, and we may also suppose that the conclusion of Theorem 
\ref{theorem3.1} has already been established for exponents smaller than $k$.\par

We aim to show that $\lam(s,k)\le 0$. We may 
therefore work throughout under the assumption that $\lam(s,k)=\Lam$ with $\Lam>0$, 
and seek a contradiction. Of course, should $\lam(s,k)\le 0$, then there is nothing to prove. 
We initiate the iteration with an appeal to Lemma \ref{lemma6.3}, which shows that
$$U_{s,k}^B(\bfgra)\ll p^{s\tet}K_{\tet,\tet}^1(\bfgra).$$
In view of the assumption (\ref{6.3}), we deduce from (\ref{3.24}) that
$$\llbracket K_{\tet,\tet}^1(\bfgra)\rrbracketsub{\Lam}\gg 
\frac{p^{-s\tet}U_{s,k}^B(\bfgra)}{p^{\Lam H}U_{s,k}^{B,H}(\bfgra)}\ge 
p^{-s\tet-H\eps}.$$
Our hierarchy (\ref{6.2}) combines with (\ref{6.10}), therefore, to ensure that
\begin{equation}\label{10.1}
\Ktil_{\tet,\tet}^1\gg p^{-2s\tet}.
\end{equation}

\par Next, we apply Lemma \ref{lemma9.3} repeatedly. Put
\begin{equation}\label{10.2}
N=\lceil 16sk/\Lam\rceil .
\end{equation}
Our assumption that $B$, and hence also $H$, is sufficiently large in terms of our 
hierarchy of parameters (\ref{6.2}) ensures that $2k^{2N+2}\tet<H$. We claim that 
sequences $(a_n)$, $(b_n)$, $(r_n)$, $(\rho_n)$ may be defined for $0\le n\le N$ in 
such a manner that
\begin{equation}\label{10.3}
1\le r_n\le k-1,\quad k^2\del\tet \le b_n\le k^{2n+2}\tet ,
\end{equation}
\begin{equation}\label{10.4}
\del \tet \le a_n\le (k-r_n+1)b_n/r_n,
\end{equation}
\begin{equation}\label{10.5}
0<\rho_n<(1-1/k)^2,\quad \rho_nb_n\ge b_{n-1}\quad (n\ge 1),
\end{equation}
and so that
\begin{equation}\label{10.6}
\Ktil_{\tet,\tet}^1\ll (\Ktil_{a_n,b_n}^{r_n})^{\rho_1\cdots \rho_n}
(p^{-\Lam/(2k)})^{nb_0}.
\end{equation}
We initiate these sequences by putting $a_0=b_0=\tet$ and $\rho_0=r_0=1$, so that 
(\ref{10.6}) is immediate in the case $n=0$ from the usual convention that an empty 
product is $1$, whence $\rho_1\cdots \rho_n=1$ for $n=0$.\par

We now attend to the task of confirming this claim, proceeding by induction on $n$. 
Suppose that such has already been confirmed for $0\le n<m$, with $1\le m\le N$. Then we 
have
\begin{equation}\label{10.7}
\Ktil_{\tet,\tet}^1\ll (\Ktil_{a_{m-1},b_{m-1}}^{r_{m-1}})^{\rho_1\cdots \rho_{m-1}}
(p^{-\Lam/(2k)})^{(m-1)b_0}.
\end{equation}
We estimate $\Ktil_{a_{m-1},b_{m-1}}^{r_{m-1}}$ by appealing to Lemma 
\ref{lemma9.3}. The conditions (\ref{10.3})-(\ref{10.5}) may be assumed to hold with 
$n=m-1$. Thus, since
$$k^2b_{m-1}\le k^{2m+2}\tet\le k^{2N+2}\tet<H/2,$$
we see that the hypotheses required to apply Lemma \ref{lemma9.3} are satisfied. 
Consequently, there exist integers $a_m$, $b_m$, $r_m$, $\rho_m$ having the property 
that
\begin{equation}\label{10.8}
\Ktil_{a_{m-1},b_{m-1}}^{r_{m-1}}\ll (\Ktil_{a_m,b_m}^{r_m})^{\rho_m}
p^{-b_{m-1}\Lam/(2k)},
\end{equation}
with
$$a_m\ge \del \tet,\quad b_m\ge k^2\del \tet,\quad r_ma_m\le (k-r_m+1)b_m,$$
$$1\le r_m\le k-1,\quad 0<\rho_m<(1-1/k)^2,$$
$$(1+2/k)b_{m-1}\le b_m\le k^2b_{m-1},\quad 
b_m=\left\lceil \frac{r_m+1}{k-r_m}a_m\right\rceil ,\quad \rho_mb_m\ge b_{m-1}.$$
Since we may suppose from (\ref{10.3}) in the case $n=m-1$ that 
$b_{m-1}\le k^{2m}\tet$, we see that $b_m\le k^2b_{m-1}\le k^{2m+2}\tet$, and so 
the conditions (\ref{10.3})-(\ref{10.5}) are all met with $n=m$. Moreover, on substituting 
(\ref{10.8}) into (\ref{10.7}), we obtain the bound
$$\Ktil_{\tet,\tet}^1\ll (\Ktil_{a_m,b_m}^{r_m})^{\rho_1\cdots \rho_m}
(p^{-\Lam/(2k)})^{(m-1)b_0+\rho_1\cdots \rho_{m-1}b_{m-1}}.$$
However, the condition (\ref{10.5}) for $1\le n\le m-1$ ensures that
$$\rho_1\cdots \rho_{m-1}b_{m-1}\ge \rho_1\cdots \rho_{m-2}b_{m-2}\ge \ldots 
\ge \rho_1b_1\ge b_0,$$
and so
$$(m-1)b_0+\rho_1\cdots \rho_{m-1}b_{m-1}\ge mb_0.$$
We therefore conclude that (\ref{10.6}) holds for $n=m$, and hence our claimed 
assertion follows by induction for $0\le n\le N$.\par

At this point in our argument, we may combine the lower bound (\ref{10.1}) with the upper 
bound (\ref{10.6}) in the case $n=N$. Thus we see that
\begin{equation}\label{10.9}
p^{-2s\tet}\ll \Ktil_{\tet,\tet}^1\ll (\Ktil_{a_N,b_N}^{r_N})^\rho 
(p^{-\Lam /(2k)})^{N\tet},
\end{equation}
where $\rho=\rho_1\cdots \rho_N<1$. On recalling the definition (\ref{3.24}) and the 
assumption that $\lam(s,k)=\Lam$, the conclusion of Lemma \ref{lemma4.2} shows that
$$\llbracket K_{a_N,b_N}^{r_N}\rrbracketsub{\Lam} \ll (p^H)^{\Lam+\eps}/p^{\Lam H}
=p^{H\eps}.$$
Then, again employing the properties of our hierarchy (\ref{6.2}), we may suppose that 
$\Ktil_{a_N,b_N}^{r_N}\ll p^\tet$, whence (\ref{10.9}) delivers the bound
$$p^{-2s\tet}\ll p^{\tet (1-N\Lam/(2k)}),$$
and hence
\begin{equation}\label{10.10}
(p^\tet)^{4s}\gg (p^\tet)^{N\Lam /(2k)}.
\end{equation}
Observe that the definition (\ref{6.10}) of $\tet$ shows that $p^\tet$ is sufficiently large 
in terms of $s$, $k$ and $\Lam$. Hence, the upper bound (\ref{10.10}) can hold only when
$$4s\ge N\Lam/(2k).$$
In view of (\ref{10.2}), we therefore obtain the bound
$$\Lam\le 8sk/N\le \Lam/2,$$
which yields a contradiction to the assumption that $\Lam>0$.\par

We are therefore forced to conclude that $\Lam$ cannot be positive, whence 
$\lam(s,k)=0$. This completes the proof of Theorem \ref{theorem3.1} for the exponent 
$k$, and the theorem follows in full by induction on $k$.
\end{proof}

Corollary \ref{corollary3.2} follows from Theorem \ref{theorem3.1} by simply interpreting 
the definitions (\ref{3.12}) and (\ref{3.13}) of $\lam^*(s,\tet;\tau)$ and $\lam(s,\tet)$. 
Thus, since $\lam(s,k)=0$, we may suppose that for every $\eps>0$, whenever $\tau>0$ 
is sufficiently small, one has $\lam^*(s,k;\tau)<\eps$. But for each $\eps>0$, one has
$$U_{s,k}^B(\bfgra)\ll p^{(\lam^*(s,k;\tau)+\eps)H}U_{s,k}^{B,H}(\bfgra)$$
for all sequences $(\gra_n)\in \dbD$, all $\bfvarphi\in \Phi_\tau(B)$ and all large enough 
values of $B$. The corresponding conclusion for the zero sequence 
$(\gra_n)\in \dbD_0\setminus \dbD$ is, of course, trivial.\par

The inquisitive reader might wonder at what point in the proof of Theorem \ref{theorem3.1} 
and Corollary \ref{corollary3.2} did we find ourselves limited to the situation in which 
$H=\lceil B/k\rceil$, rather than allowing the possibility that $H$ might exceed 
$\lceil B/k\rceil$. This is a little subtle, since at face value our argument does not involve 
estimates for $K_{a,b,c}^{r,\bfvarphi,\nu}(\bfgra)$ with $a$ and $b$ close in size to $B/k$. 
However, in the proof of Lemma \ref{lemma4.1}, in equation (\ref{4.9}), one encounters a 
situation wherein certain congruences would be trivially satisfied were $h$ to exceed $B/k$. 
This failure of independence amongst the congruences would compromise our estimates, and 
implicitly generate associated difficulties in \S8. Since we are imposing the condition 
$h\le (1-\del)H$ in Lemma \ref{lemma4.1}, one finds it necessary to restrict $H$ to be no 
larger than about $B/k$.

\section{Solutions of congruences in short intervals} Rather than embark at once on the 
proof of Theorem \ref{theorem1.1} and its corollaries, we spend some time in this section 
on a more immediate application of Theorem \ref{theorem3.1} to the topic of solutions of 
congruences in short intervals. We are interested in rational functions $\chi_j\in \dbQ(t)$ 
$(1\le j\le k)$, and the number of integral solutions of systems of congruences of the shape
\begin{equation}\label{11.1}
\sum_{i=1}^s\chi_j(x_i)\equiv \sum_{i=1}^s\chi_j(y_i)\mmod{p^B}\quad (1\le j\le k),
\end{equation}
with $X<x_i,y_i\le X+Y$. Writing $\chi_j(t)=\varphi_j(t)/\gam_j(t)$ for suitable polynomials 
$\varphi_j,\gam_j\in \dbZ[t]$ with $(\varphi_j(t),\gam_j(t))=1$, it is apparent that it is 
sensible to exclude choices for the variables $x$ with $\gam_j(x)\equiv 0\mmod{p}$. With 
such choices for $x$ excluded, there is a multiplicative inverse for $\gam_j(x)$, say 
$\gam_j(x)^{-1}$ modulo $p^B$, and this counting problem makes sense. We are also 
able to make sense of the Wronskian $W(t;\bfchi)$ defined by (\ref{1.1}) in these 
circumstances. Indeed, as a rational number $A/Q$ in lowest terms, the Wronskian 
$W(x;\bfchi)$ has denominator $Q$ not divisible by $p$ in the situation under discussion. If 
$Q^{-1}$ denotes an integer defining the multiplicative inverse of $Q$ modulo $p^B$, then 
we shall always replace $W(x;\bfchi)$ by the integer $AQ^{-1}$.\par

We obtain strong estimates for the number of solutions of the system (\ref{11.1}) when 
$s\le k(k+1)/2$ and $Y\le p^{B/k}$, and we restrict to solutions with both denominators 
$\gam_j$ and the Wronskian non-vanishing modulo $p$.\par

\begin{theorem}\label{theorem11.1} When $k\in \dbN$ and $1\le j\le k$, suppose that 
$\chi_j\in \dbQ(t)$ is a rational function with $\chi_j=\varphi_j/\gam_j$ for suitable 
polynomials $\varphi_j,\gam_j\in \dbZ[t]$ satisfying $(\varphi_j,\gam_j)=1$. Let $s$ be 
a positive number with $s\le k(k+1)/2$, and let $p$ be a prime number with $p>k$. Also, 
suppose that $(\gra_n)_{n\in \dbZ}$ is a sequence of complex numbers. Define
$$h(\bfalp;X,Y)=\sideset{}{^*}\sum_{X<n\le X+Y}\gra_ne(\alp_1\chi_1(n)+\ldots 
+\alp_k\chi_k(n)),$$
where the summation is restricted by the conditions
\begin{equation}\label{11.2}
(W(n;\bfchi),p)=1\quad \text{and}\quad (\gam_j(n),p)=1\quad (1\le j\le k).
\end{equation}
Then whenever $\eps>0$ and $B$ is sufficiently large in terms of $\eps$ and $k$, one has
\begin{equation}\label{11.3}
\oint_{p^B}|h(\bfalp;X,Y)|^{2s}\d\bfalp \ll p^{B\eps}\biggl(1+\frac{Y}{p^{B/k}}\biggr)^s
\biggl( \sum_{X<n\le X+Y}|\gra_n|^2\biggr)^s.
\end{equation}
\end{theorem}

\begin{corollary}\label{corollary11.2} With the hypotheses of Theorem \ref{theorem11.1}, 
denote by $N_B(X,Y)$ the number of integral solutions of the system of congruences 
(\ref{11.1}) subject for $1\le i\le s$ to the conditions $X<x_i,y_i\le X+Y$ and
$$(W(x_i;\bfchi)W(y_i;\bfchi),p)=1\quad \text{and}\quad (\gam_j(x_i)\gam_j(y_i),p)=1
\quad (1\le j\le k).$$
Then whenever $\eps>0$ and $B$ is sufficiently large in terms of $\eps$ and $k$, one has
$$N_B(X,Y)\ll p^{B\eps}(Y+1)^s\biggl(\frac{Y}{p^{B/k}}+1\biggr)^s.$$
In particular, when $p^B\ge Y^k$, one has $N_B(X,Y)\ll p^{B\eps}(Y+1)^s$.
\end{corollary}

It is apparent that, in general, the diagonal solutions in (\ref{11.1}) make a contribution of 
order $Y^s$ to $N_B(X,Y)$, so the conclusion of Corollary \ref{corollary11.2} is essentially 
best possible.

\begin{proof}[The proof of Theorem \ref{theorem11.1}] Our goal is to transform the mean 
value on the left hand side of (\ref{11.3}) into one amenable to Theorem \ref{theorem3.1}. 
We may plainly suppose that the sequence $(\gra_n)$ satisfies the condition that 
$\gra_n=0$ for $n\le X$ and for $n>X+Y$. We are also at liberty, moreover, to assume that 
$\gra_n=0$ unless the conditions (\ref{11.2}) all hold. Recalling the notation of \S3, we may 
then define
$$F_p(\bfalp)=\rho_0(\bfgra)^{-1}\sum_{n\in \dbZ}\gra_n e(\psi(n;\bfalp)),$$
where we now write
$$\psi(n;\bfalp)=\alp_1\chi_1(n)+\ldots +\alp_k\chi_k(n).$$
With this notation, the claimed bound (\ref{11.3}) translates into the assertion that
\begin{equation}\label{11.4}
\oint_{p^B}|F_p(\bfalp)|^{2s}\d\bfalp \ll p^{B\eps}\biggl( \frac{Y}{p^{B/k}}+1\biggr)^s,
\end{equation}
and it is this that we now seek to establish. An application of H\"older's inequality delivers 
the conclusion (\ref{11.4}) for $0<s\le k(k+1)/2$ from that in the special case 
$s=k(k+1)/2$. We therefore restrict attention henceforth to the case $s=k(k+1)/2$.\par

Some preparation is required prior to the proof of the estimate (\ref{11.4}) via Corollary 
\ref{corollary3.2}. Let $\tau>0$ be sufficiently small in terms of $s$ and $k$. Then we may 
suppose that $B$ is sufficiently large in terms of $\tau$, as well as $s$, $k$ and $\eps$. We 
put $c=\lceil \tau B\rceil$. Next we sort the implicit summation in $F_p(\bfalp)$ into 
arithmetic progressions modulo $p^c$. In view of our assumptions on the sequence 
$(\gra_n)$ implied by the conditions (\ref{11.2}), we may suppose that $\gra_n=0$ 
whenever $n\equiv \xi\mmod{p^c}$ and
$$W(\xi;\bfchi)\gam_1(\xi)\cdots \gam_k(\xi) \equiv 0\mmod{p}.$$
With this assumption, one finds that
$$F_p(\bfalp)=\rho_0(\bfgra)^{-1}\sum_{\xi\nmod{p^c}}\rho_c(\xi)\grf_c(\bfalp;\xi),$$
and hence Lemma \ref{lemma6.2} delivers the bound
$$\rho_0(\bfgra)^2|F_p(\bfalp)|^{2s}\le p^{sc}\sum_{\xi\nmod{p^c}}\rho_c(\xi)^2
|\grf_c(\bfalp;\xi)|^{2s}.$$
Thus we obtain the estimate
\begin{equation}\label{11.5}
\oint_{p^B}|F_p(\bfalp)|^{2s}\d\bfalp \le p^{sc}\rho_0(\bfgra)^{-2}
\sum_{\xi\nmod{p^c}}\rho_c(\xi)^2I_p(\xi),
\end{equation}
where
$$I_p(\xi)=\oint_{p^B}|\grf_c(\bfalp;\xi)|^{2s}\d\bfalp .$$

\par The mean value $I_p(\xi)$ counts the integral solutions $\bfy,\bfz$ of the system of 
congruences
\begin{equation}\label{11.6}
\sum_{i=1}^s\left( \chi_j(p^cy_i+\xi)-\chi_j(p^cz_i+\xi)\right) \equiv 0\mmod{p^B}\quad 
(1\le j\le k),
\end{equation}
with each solution being counted with weight
\begin{equation}\label{11.7}
\rho_c(\xi)^{-2s}\prod_{i=1}^s \gra_{p^cy_i+\xi}{\overline \gra_{p^cz_i+\xi}}.
\end{equation}
Here, we note that we may restrict attention to the situation in which
$$W(\xi;\bfchi)\gam_1(\xi)\cdots \gam_k(\xi)\not\equiv 0\mmod{p} \quad \text{and}\quad 
\rho_c(\xi)>0.$$
In particular, it follows from (\ref{3.8}) and (\ref{11.5}) that
$$\oint_{p^B}|F_p(\bfalp)|^{2s}\d\bfalp \le p^{sc}U_{s,k}^{B,c}(\bfgra).$$

\par We reinterpret the system (\ref{11.6}) by applying Taylor's theorem to expand the 
rational functions $\chi_j(p^ct+\xi)$. When $1\le j\le k$, we find that for suitable 
polynomials $\Phi_j\in \dbZ[t]$, one has
$$\chi_j(p^ct+\xi)-\chi_j(\xi)\equiv \sum_{l=1}^k 
\ome_{lj}(p^ct)^l+(p^ct)^{k+1}\Phi_j(p^ct)\mmod{p^B},$$
in which the integral coefficients $\ome_{lj}$ are defined by taking
$$\ome_{lj}=(l!)^{-1}\chi_j^{(l)}(\xi)\quad (1\le l,j\le k).$$
A few words of explanation are in order here. First, since $\varphi_j=\gam_j\chi_j$, it 
follows from the differentiation rule of Leibniz that
$$\varphi_j^{(l)}=\sum_{m=0}^l\binom{l}{m}\gam_j^{(m)}\chi_j^{(l-m)},$$
whence one obtains the iterative relation
$$\gam_j(\xi)\frac{\chi_j^{(l)}(\xi)}{l!}=\frac{\varphi_j^{(l)}(\xi)}{l!}-\sum_{m=1}^l
\binom{l}{m}\left( \frac{\gam_j^{(m)}(\xi)}{m!}\right) \left( \frac{\chi_j^{(l-m)}(\xi)}
{(l-m)!}\right) .$$
But since $\gam_j,\varphi_j\in \dbZ[t]$, one finds that $m!$ divides every coefficient of 
$\gam_j^{(m)}(t)$, and likewise $l!$ divides every coefficient of $\varphi_j^{(l)}(t)$. In 
addition, we may suppose that $\gam_j(\xi)\not \equiv 0\mmod{p}$. An inductive argument 
therefore conveys us from this iterative relation to the conclusion that 
$\text{ord}_p(l!)\le \text{ord}_p(\chi_j^{(l)}(\xi))$ for all non-negative integers $l$. By 
multiplying through by appropriate multiplicative inverses modulo $p^B$, we may thus 
suppose that $(l!)^{-1}\chi_j^{(l)}(\xi)$ is an integer for $l\ge 0$. Second, the expansion of 
$\chi_j(p^ct+\xi)$ might be expected to be a non-terminating infinite series. However, the 
terms $(l!)^{-1}\chi_j^{(l)}(\xi) (p^ct)^l$ are necessarily congruent to $0$ modulo $p^B$ 
whenever $l$ is sufficiently large in terms of $B$.\par

The determinant of the coefficient matrix $\Ome =(\ome_{lj})_{1\le l,j\le k}$ is given by 
the formula
$$\det(\Ome)=W(\xi;\bfchi)\biggl( \prod_{l=1}^kl!\biggr)^{-1}.$$
Since we may suppose that $p>k$, the hypothesis $(W(\xi;\bfchi),p)=1$ permits us to 
conclude that $(\det(\Ome),p)=1$, whence $\Ome$ possesses a multiplicative inverse 
$\Ome^{-1}$ modulo $p^B$ having integral coefficients. We now replace $\bfchi$ by 
$\Ome^{-1}\bfchi$ and $\bfPhi$ by $\Ome^{-1}\bfPhi$. This amounts to taking suitable 
integral linear combinations of the congruences comprising (\ref{11.6}). In this way, we see 
that there is no loss of generality in supposing that the coefficient matrix $\Ome$ is equal to 
$I_k$. Hence, there exist polynomials $\Xi_j\in \dbZ[t]$ having the property that whenever 
the system (\ref{11.6}) is satisfied, then
$$\sum_{i=1}^s(p^c)^j\left( \Psi_j(y_i)-\Psi_j(z_i)\right)\equiv 0\mmod{p^B}\quad 
(1\le j\le k),$$
in which $\Psi_j(t)=t^j+p^c\Xi_j(t)$. In particular, the system of polynomials $\bfPsi$ is 
$p^c$-spaced.\par

The discussion of the previous paragraph shows that in any solution $\bfy,\bfz$ of the 
system (\ref{11.6}), counted with weight (\ref{11.7}), one has the additional constraints
$$\sum_{i=1}^s(\Psi_j(y_i)-\Psi_j(z_i))\equiv 0\mmod{p^{B-kc}}\quad (1\le j\le k).$$
Define the coefficients
$$\grc_y(\bfalp)=\gra_{p^cy+\xi}e(\psi(p^cy+\xi;\bfalp)).$$
Also, write
$$\grg_\bfgrc(\bfalp,\bfbet)=\rho_0(\bfgrc)^{-1}\sum_{y\in \dbZ}\grc_y(\bfalp)
e(\bet_1\Psi_1(y)+\ldots +\bet_k\Psi_k(y)),$$
and define the mean value
$$J(\bfalp)=\oint_{p^{B-kc}}|\grg_\bfgrc(\bfalp,\bfbet)|^{2s}\d\bfbet .$$
Note that $\rho_0(\bfgrc)=\rho_c(\xi;\bfgra)$. Then, just as in the argument leading to 
(\ref{7.18}) above, one sees that
\begin{equation}\label{11.8}
I_p(\xi)=\oint_{p^B}|\grg_\bfgrc(\bfalp;{\mathbf 0})|^{2s}\d\bfalp =\oint_{p^B}
J(\bfalp)\d\bfalp .
\end{equation}

Observe that, when written using the notation defined in (\ref{3.6}), one has 
$J(\bfalp)=U_{s,k}^{B-kc, \bfPsi}(\bfgrc)$. Moreover, the system $\bfPsi$ is $p^c$-spaced 
and $B-kc$ is sufficiently large in terms of $k$, $\tau$ and $\eps$. Then Corollary 
\ref{corollary3.2} yields the bound
$$J(\bfalp)\ll p^{B\eps}U_{s,k}^{B-kc,H,\bfPsi}(\bfgrc),$$
where $H=\lceil B/k\rceil -c$. On substituting this bound into (\ref{11.8}), and thence into 
(\ref{11.5}), we infer that
\begin{equation}\label{11.9}
\oint_{p^B}|F_p(\bfalp)|^{2s}\d\bfalp \ll p^{sc+B\eps}
\rho_0(\bfgra)^{-2}\sum_{\xi\nmod{p^c}}
\rho_c(\xi)^2U_{s,k}^{B-kc,H,\bfPsi}(\bfgrc).
\end{equation}

\par Temporarily, we abbreviate $\grf_H(\bfbet;\eta;\bfgrc;\bfPsi)$ 
to $\grf_H(\bfbet;\eta)$. Thus, we have
$$\grf_H(\bfbet;\eta)=\rho_H(\eta;\bfgrc)^{-1}\sum_{y\equiv \eta\mmod{p^H}}
\grc_y(\bfalp)e(\bet_1\Psi_1(y)+\ldots +\bet_k\Psi_k(y)),$$
wherein the coefficients $\grc_y(\bfalp)$ are $0$ whenever $p^cy+\xi\le X$ or 
$p^cy+\xi>X+Y$. It follows via Cauchy's inequality, therefore, that
$$\rho_H(\eta;\bfgrc)^2|\grf_H(\bfbet;\eta)|^2\le \biggl( 
\sum_{\substack{y\equiv \eta \mmod{p^H}\\ X<p^cy+\xi\le X+Y}}1\biggr) 
\sum_{y\equiv \eta\mmod{p^H}}|\grc_y(\bfalp)|^2.$$
Moreover, one has
$$\rho_H(\eta;\bfgrc)^2=\sum_{y\equiv \eta\mmod{p^H}}|\grc_y(\bfalp)|^2,$$
so that
$$|\grf_H(\bfbet;\eta)|^2\le 1+Y/p^{c+H}.$$
We therefore infer from (\ref{3.8}) that
$$U_{s,k}^{B-kc,H,\bfPsi}(\bfgrc)\ll (1+Y/p^{c+H})^s.$$
Since $c+H=\lceil B/k\rceil$, we conclude from (\ref{11.9}) that
\begin{align*}
\oint_{p^B}|F_p(\bfalp)|^{2s}\d\bfalp &\ll p^{sc+B\eps}(1+Y/p^{B/k})^s
\rho_0(\bfgra)^{-2}\sum_{\xi\nmod{p^c}}\rho_c(\xi)^2\\
&\ll p^{(2s\tau+\eps)B}(1+Y/p^{B/k})^s.
\end{align*}
We recall that $\tau$ was chosen sufficiently small in terms of $s$ and $k$. Thus, for each 
positive number $\del$, we have
$$\oint_{p^B}|F_p(\bfalp)|^{2s}\d\bfalp \ll p^{B\del}(1+Y/p^{B/k})^s,$$
so that the estimate (\ref{11.4}) does indeed hold. This completes the proof of the 
theorem.
\end{proof}

Corollary \ref{corollary11.2} is immediate from Theorem \ref{theorem11.1} via 
orthogonality.

\section{Mean values of exponential sums} We now explain how Theorem \ref{theorem3.1}
 and Corollary \ref{corollary3.2} may be applied to bound mean values of exponential sums. 
In particular, we establish Theorem \ref{theorem1.1} and its corollaries. Much of the 
necessary work is already accomplished in the shape of Theorem \ref{theorem11.1}.

\begin{proof}[The proof of Theorem \ref{theorem1.1}] Let $\varphi_j\in \dbZ[t]$ 
$(1\le j\le k)$ be polynomials with non-vanishing Wronskian $W(t;\bfvarphi)$, and let 
$\eps>0$ be a small positive number. Let $\calZ$ denote the set of integral zeros of 
$W(t;\bfvarphi)$, and let $X$ be sufficiently large in terms of $\bfvarphi$, $s$, $k$ and 
$\eps$. We note that
$$\text{card}(\calZ)\le \text{deg}(W(t;\bfvarphi))\ll 1.$$
We suppose that $s=k(k+1)/2$. Finally, when $(\gra_n)_{n\in \dbZ}$ is a 
sequence of complex numbers, we write
$$F(\bfalp;X)=\rho_0^{-1}\sum_{|n|\le X}\gra_ne(\psi(n;\bfalp))$$
and
$$F_0(\bfalp;X)=\rho_0^{-1}\sum_{\substack{|n|\le X\\ n\not\in \calZ}}\gra_n
e(\psi(n;\bfalp)),$$
where $\psi(n;\bfalp)$ is defined as in (\ref{3.3}). In the argument to come, there will be 
no loss of generality in supposing that $\gra_n=0$ for $|n|>X$.\par

As a consequence of Cauchy's inequality, one has
$$\biggl| \sum_{\substack{|n|\le X\\ n\in \calZ}}\gra_n e(\psi(n;\bfalp))\biggr|^2\le 
\text{card}(\calZ)\sum_{|n|\le X}|\gra_n|^2\le \rho_0^2\,\text{card}(\calZ),$$
whence
$$|F(\bfalp;X)|\ll 1+|F_0(\bfalp;X)|.$$
Thus, one has
$$|F(\bfalp;X)|^{2s}\ll 1+|F_0(\bfalp;X)|^{2s},$$
so that
\begin{equation}\label{12.1}
\oint |F(\bfalp;X)|^{2s}\d\bfalp \ll 1+\oint |F_0(\bfalp;X)|^{2s}\d\bfalp .
\end{equation}
In this way, we discern that it suffices to restrict attention to situations in which the 
underlying variables possess non-vanishing Wronskians. By orthogonality, the mean value on 
the right hand side of (\ref{12.1}) counts the number of integral solutions of the system of 
equations (\ref{1.4}) with $|\bfx|,|\bfy|\le X$ and $x_i,y_i\not\in \calZ$, and with each 
solution $\bfx,\bfy$ being counted with weight
\begin{equation}\label{12.2}
\rho_0^{-2s}\prod_{i=1}^s\gra_{x_i}{\overline \gra_{y_i}}.
\end{equation}

\par We impose a non-vanishing condition modulo $p$ on the Wronskian in each of these 
solutions, for a suitable prime number $p$. Given a solution $\bfx,\bfy$ of (\ref{1.4}) of the 
type in question, the integer
$$\Xi(\bfx,\bfy)=\prod_{i=1}^sW(x_i;\bfvarphi)W(y_i;\bfvarphi)$$
is non-zero. Moreover, one has $|\Xi(\bfx,\bfy)|\le CX^D$, for some $C>0$ depending at 
most on $s$, $k$ and the coefficients of $\bfvarphi$, and $D$ a positive integer with
$$D\le 2s\sum_{j=1}^k\deg(\varphi_j).$$
Let $\calP$ denote the set of prime numbers $p$ with
$$(\log X)^2<p\le 3(\log X)^2.$$
Thus, when $X$ is large, it is a consequence of the prime number theorem that
$$\prod_{p\in \calP}p>(\log X)^{(\log X)^2/\log \log X}>CX^D.$$
We therefore deduce that for each solution $\bfx,\bfy$ of (\ref{1.4}) counted by the 
integral on the right hand side of (\ref{12.1}), there exists $p\in \calP$ with
$$\prod_{i=1}^sW(x_i;\bfvarphi)W(y_i;\bfvarphi)\not \equiv 0\mmod{p}.$$
In particular, one has
$$\oint |F_0(\bfalp;X)|^{2s}\d\bfalp \le \sum_{p\in \calP}\oint |F_p(\bfalp;X)|^{2s}
\d\bfalp ,$$
where
\begin{equation}\label{12.3}
F_p(\bfalp;X)=\rho_0^{-1}\sum_{W(n;\bfvarphi)\not\equiv 0\mmod{p}}|\gra_n|
e(\psi(n;\bfalp)).
\end{equation}
Thus, we conclude from (\ref{12.1}) that
\begin{equation}\label{12.4}
\oint |F(\bfalp;X)|^{2s}\d\bfalp \ll 1+(\log X)^2\max_{p\in \calP}\oint 
|F_p(\bfalp;X)|^{2s}\d\bfalp .
\end{equation}

\par Our goal is to establish that for each fixed prime $p\in \calP$, one has
\begin{equation}\label{12.5}
\oint |F_p(\bfalp;X)|^{2s}\d\bfalp \ll X^{\eps/2 }.
\end{equation}
It follows by substituting this estimate into (\ref{12.4}) that
$$\oint |F(\bfalp;X)|^{2s}\d\bfalp \ll 1+X^{\eps/2}(\log X)^2\ll X^\eps,$$
whence
\begin{equation}\label{12.6}
\oint \biggl| \sum_{|n|\le X}\gra_ne(\psi(n;\bfalp))\biggr|^{2s}\d\bfalp \ll X^\eps
\rho_0^{2s}=X^\eps \biggl( \sum_{|n|\le X}|\gra_n|^2\biggr)^s.
\end{equation}
This establishes the first conclusion (\ref{1.2}) of Theorem \ref{theorem1.1} in the special 
case $s=k(k+1)/2$. Of course, the desired conclusion for smaller values of $s$ follows by 
applying H\"older's inequality. Meanwhile, the final conclusion (\ref{1.3}) of Theorem 
\ref{theorem1.1} follows from (\ref{1.2}) by merely specialising the sequence $(\gra_n)$ 
to be $(1)$.\par

We focus now on the proof of the estimate (\ref{12.5}), and this involves preparation for 
the application of Theorem \ref{theorem11.1}. In our application of Theorem 
\ref{theorem11.1} we take $\gam_j=1$, so that $\chi_j=\varphi_j$ $(1\le j\le k)$. In view 
of the definition (\ref{12.3}), we may suppose that $\gra_n$ is real with $\gra_n\ge 0$ for 
each $n$. Also, we take
$$B=\left\lceil \frac{k\log X}{\log p}\right\rceil .$$
Thus $X\le p^{B/k}\le pX$.\par

The mean value on the left hand side of (\ref{12.5}) counts the integral solutions of the 
system of equations (\ref{1.4}) with $|\bfx|,|\bfy|\le X$ satisfying
$$W(x_i;\bfvarphi)W(y_i;\bfvarphi)\not \equiv 0\mmod{p}\quad (1\le i\le s),$$
and in which each solution is counted with weight (\ref{12.2}). Since these weights are now 
assumed to be non-negative, one sees that an upper bound for this weighted number of 
solutions is given by
$$\oint_{p^B}|F_p(\bfalp;X)|^{2s}\d\bfalp ,$$
for by orthogonality this mean value counts the integral solutions $\bfx,\bfy$ of the system 
of congruences
$$\sum_{i=1}^s(\varphi_j(x_i)-\varphi_j(y_i))\equiv 0\mmod{p^B}\quad (1\le j\le k),$$
with $\bfx,\bfy$ satisfying the same attendant conditions, and again counted with weight 
(\ref{12.2}).\par

By Theorem \ref{theorem11.1}, one has
$$\oint_{p^B} |F_p(\bfalp;X)|^{2s}\d\bfalp \ll \rho_0^{-2s}p^{B\eps/(4k)}\biggl( 
\frac{X}{p^{B/k}}+1\biggr)^s\biggl( \sum_{|n|\le X}|\gra_n|^2\biggr)^s\ll 
p^{B\eps/(4k)}.$$
Since $p^{B\eps}\ll X^{(k+1)\eps}$, we confirm the estimate (\ref{12.5}). Thus we deduce 
from (\ref{12.4}) that one has the bound (\ref{12.6}). The proof of Theorem 
\ref{theorem1.1}, as previously discussed, is now complete.
\end{proof}

The proof of Corollary \ref{corollary1.2} involves only a computation involving the 
Wronskian.

\begin{proof}[The proof of Corollary \ref{corollary1.2}] In order to establish Corollary 
\ref{corollary1.2}, we have merely to note that when $1\le d_1<d_2<\ldots <d_k$ and 
$\varphi_j(t)=t^{d_j}$ $(1\le j\le k)$, then the Wronskian $W(t;\bfvarphi)$ is non-zero as 
a polynomial. In order to see this, observe that
$$W(t;\bfvarphi)=\det\left(d_j(d_j-1)\cdots (d_j-i+1)t^{d_j-i}\right)_{1\le i,j\le k}.$$
Every monomial in this determinant is an integer multiple of $t^D$, where
$$D=\biggl( \sum_{j=1}^kd_j\biggr) -k(k+1)/2=\sum_{j=1}^k(d_j-j).$$
Thus $W(t;\bfvarphi)=t^D\det(\Ome)$, where
$$\Ome=\left( d_j(d_j-1)\cdots (d_j-i+1)\right)_{1\le i,j\le k}.$$
By taking appropriate linear combinations of the rows of the matrix $\Ome$, one sees that
$$\det(\Ome)=\det(d_j^i)_{1\le i,j\le t}=d_1\cdots d_k\prod_{1\le i<j\le k}(d_i-d_j).$$
Thus we see that $\det(\Ome)\ne 0$, and hence $W(t;\bfvarphi)\ne 0$. We therefore 
conclude from Theorem \ref{theorem1.1} that when $0<s\le k(k+1)/2$, one has
$$\oint \biggl| \sum_{1\le x\le X}e(\alp_1x^{d_1}+\ldots +\alp_kx^{d_k})\biggr|^{2s}
\d\bfalp \ll X^{s+\eps}.$$
This completes the proof of Corollary \ref{corollary1.2}.
\end{proof}

We address the proof of Corollary \ref{corollary1.4} before moving on to consider Corollary 
\ref{corollary1.3}, this being logically speaking a more direct course of action.

\begin{proof}[The proof of Corollary \ref{corollary1.4}] Write
$$\grh(\bfalp)=\sum_{|n|\le X}\gra_ne(n\alp_1+\ldots +n^k\alp_k).$$
When $\varphi_j(t)=t^j$ $(1\le j\le k)$, one finds that the Wronskian determinant is 
triangular, whence
$$W(t;\varphi)=\prod_{j=1}^kj!\ne 0.$$
Thus the conditions required to apply Theorem \ref{theorem1.1} hold, and one obtains the 
bound
\begin{equation}\label{12.7}
\oint |\grh(\bfalp)|^{2s}\d\bfalp \ll X^\eps \biggl( \sum_{|n|\le X}|\gra_n|^2\biggr)^s
\end{equation}
for $0<s\le k(k+1)/2$. Meanwhile, an application of Cauchy's inequality shows that
$$|\grh(\bfalp)|^2\le (2X+1)\sum_{|n|\le X}|\gra_n|^2.$$
Thus, when $X$ is large, we deduce from the special case $s=k(k+1)/2$ of (\ref{12.7}) 
that when $t>k(k+1)/2$, then
\begin{align*}
\oint |\grh(\bfalp)|^{2t}\d\bfalp &\ll X^{t-k(k+1)/2}\biggl( \sum_{|n|\le X}
|\gra_n|^2\biggr)^{t-k(k+1)/2}\oint |\grh(\bfalp)|^{k(k+1)}\d\bfalp \\
&\ll X^{t-k(k+1)/2+\eps}\biggl( \sum_{|n|\le X}|\gra_n|^2\biggr)^t,
\end{align*}
and the proof of the corollary is complete.
\end{proof}

\begin{proof}[The proof of Corollary \ref{corollary1.3}] For each $\eps>0$, the estimate
$$J_{s,k}(X)\ll X^\eps (X^s+X^{2s-k(k+1)/2})$$
is immediate from the special case of Corollary \ref{corollary1.4} in which $(\gra_n)=(1)$. 
Meanwhile, the asymptotic formula claimed in the corollary for $s>k(k+1)/2$ follows from 
this estimate by the standard literature in Vinogradov's mean value theorem, and permits 
the factor $X^\eps$ to be omitted from this upper bound when $s>k(k+1)/2$. The 
asymptotic formula for $J_{s,k}(X)$ asserted in Corollary \ref{corollary1.3} is a well-known 
consequence of the main conjecture in Vinogradov's mean value theorem. We briefly outline 
how to apply the circle method to prove this formula.\par

Write $L=X^{1/(4k)}$. Then, when $1\le q\le L$, $1\le a_j\le q$ $(1\le j\le k)$ and 
$(q,a_1,\ldots ,a_k)=1$, define the major arc $\grM(q,\bfa)$ by
$$\grM(q,\bfa)=\{ \bfalp \in [0,1)^k: \text{$|\alp_j-a_j/q|\le LX^{-j}$ $(1\le j\le k)$}\}.$$
The arcs $\grM(q,\bfa)$ are disjoint, as is easily verified. Let $\grM$ denote their union, and 
put $\grm=[0,1)^k\setminus \grM$.\par

Write
$$f(\bfalp;X)=\sum_{1\le x\le X}e(\alp_1x+\ldots +\alp_kx^k).$$
Also, when $\bfalp\in \grM(q,\bfa)\subseteq \grM$, put
$$V(\bfalp;q,\bfa)=q^{-1}S(q,\bfa)I(\bfalp-\bfa/q;X),$$
where
$$S(q,\bfa)=\sum_{r=1}^qe((a_1r+\ldots +a_kr^k)/q)$$
and
$$I(\bfbet;X)=\int_0^Xe(\bet_1\gam+\ldots +\bet_k\gam^k)\d\gam .$$
We then define the function $V(\bfalp)$ to be $V(\bfalp;q,a)$ when $\bfalp\in \grM(q,\bfa)
\subseteq \grM$, and to be $0$ otherwise.\par

The contribution of the minor arcs $\grm$ is easily estimated. Thus, as in 
\cite[equation (7.1)]{Woo2017a} (based on \cite[\S9]{Woo2012}), we find that for each 
$\eps>0$ one has
$$\sup_{\bfalp\in \grm}|f(\bfalp;X)|\ll X^{1-\tau+\eps},$$
where $\tau^{-1}=8k^2$. In this way, when $2s>k(k+1)$, one finds that
\begin{align*}
\int_\grm|f(\bfalp;X)|^{2s}\d\bfalp &\ll \Bigl( \sup_{\bfalp\in \grm}|f(\bfalp;X)|\Bigr)^{
2s-k(k+1)}\oint |f(\bfalp;X)|^{k(k+1)}\d\bfalp \\
&\ll X^{2s-k(k+1)/2+2\eps -(2s-k(k+1))\tau}.
\end{align*}
Thus
\begin{equation}\label{12.8}
\int_\grm |f(\bfalp;X)|^{2s}\d\bfalp =o(X^{2s-k(k+1)/2}),
\end{equation}
provided that we take $\eps$ sufficiently small in terms of $s$ and $k$.\par

On the other hand, as in the proof of \cite[Lemma 7.1]{Woo2017a}, one finds that 
when $\bfalp\in \grM(q,\bfa)\subseteq \grM$, one has
$$f(\bfalp;X)-V(\bfalp;q,\bfa)\ll L^2.$$
Using the decomposition
$$z{\overline z}-w{\overline w}=(z-w){\overline z}+w({\overline z}-{\overline w}),$$
it follows that
$$|f(\bfalp;X)|^2-|V(\bfalp;q,\bfa)|^2\ll L^2X.$$
Hence, as a consequence of the mean value theorem, one obtains the bound
\begin{align*}
|f(\bfalp;X)|^{2s}-|V(\bfalp;q,\bfa)|^{2s}&\ll 
(|f(\bfalp;X)|^2-|V(\bfalp;q,\bfa)|^2)X^{2s-2}\\
&\ll L^2X^{2s-1}.
\end{align*}
Since $\text{mes}(\grM)\ll L^{2k+1}X^{-k(k+1)/2}$, we deduce that
\begin{align}
\int_\grM |f(\bfalp;X)|^{2s}\d\bfalp -\int_\grM |V(\bfalp)|^{2s}\d\bfalp &\ll (L^{2k+3}/X)
X^{2s-k(k+1)/2}\notag \\
&=o(X^{2s-k(k+1)/2}).\label{12.9}
\end{align}

\par By applying \cite[Theorems 1.3 and 2.4]{ACK2004} as in the argument concluding the 
proof of \cite[Lemma 7.1]{Woo2017a}, one sees that
$$\int_\grM |V(\bfalp)|^{2s}\d\bfalp =\grS \grJ,$$
where, when $2s\ge \frac{1}{2}k(k+1)+1$, one has
\begin{equation}\label{12.10}
\grJ=X^{2s-k(k+1)/2}\int_{\dbR^k}|I(\bfbet;1)|^{2s}\d\bfbet +o(X^{2s-k(k+1)/2}),
\end{equation}
and, when $2s\ge \frac{1}{2}k(k+1)+2$, one has
\begin{equation}\label{12.11}
\grS=\sum_{q=1}^\infty \sum_{\substack{1\le \bfa\le q\\ (q,a_1,\ldots ,a_k)=1}}
\left| q^{-1}S(q,\bfa)\right|^{2s}+o(1).
\end{equation}
Here, the integral on the right hand side of (\ref{12.10}) is absolutely convergent, and the 
sum on the right hand side of (\ref{12.11}) is also absolutely convergent. Hence, there 
exists a real number $C_{s,k}\ge 0$ for which
$$\int_\grM |V(\bfalp)|^{2s}\d\bfalp =C_{s,k}X^{2s-k(k+1)/2}+o(X^{2s-k(k+1)/2}),$$
and we conclude from (\ref{12.8}) and (\ref{12.9}) that
\begin{align}
\oint |f(\bfalp;X)|^{2s}\d\bfalp &=\int_\grM |f(\bfalp;X)|^{2s}\d\bfalp +
\int_\grm |f(\bfalp;X)|^{2s}\d\bfalp \notag \\
&=C_{s,k}X^{2s-k(k+1)/2}+o(X^{2s-k(k+1)/2}).\label{12.12}
\end{align}
Making use of the familiar lower bound $J_{s,k}(X)\gg X^{2s-k(k+1)/2}$, we conclude that 
when $s>k(k+1)/2$, the asymptotic formula (\ref{12.12}) holds for some $C_{s,k}>0$. 
This completes the proof of Corollary \ref{corollary1.3}.
\end{proof}

\section{A remark on Tarry's problem} The resolution of the main conjecture in Vinogradov's 
mean value theorem provides a means of delivering a definitive result concerning Tarry's 
problem. When $h$, $k$ and $s$ are positive integers with $h\ge 2$, consider the 
Diophantine system
\begin{equation}\label{13.1}
\sum_{i=1}^sx_{i1}^j=\sum_{i=1}^sx_{i2}^j=\ldots =\sum_{i=1}^sx_{ih}^j\quad 
(1\le j\le k).
\end{equation}
Let $W(k,h)$ denote the least natural number $s$ having the property that the simultaneous 
equations (\ref{13.1}) possess an integral solution $\bfx$ with
$$\sum_{i=1}^sx_{iu}^{k+1}\ne \sum_{i=1}^sx_{iv}^{k+1}\quad (1\le u<v\le h).$$
The problem of estimating $W(k,h)$ has been the subject of extensive investigation by 
E.~M.~Wright and L.-K.~Hua (see \cite{Hua1938}, \cite{Hua1949} and \cite{Wri1948}). 
There are numerous applications that need not detain us here. However, we note in 
particular that Croot and Hart \cite{CH2010} have found application of these ideas in work 
on the sum-product conjecture. Classically, Hua \cite{Hua1949} was able to show that 
$W(k,h)\le k^2(\log k+O(1))$. In recent work \cite[Theorem 12.1]{Woo2017} based on 
efficient congruencing, the author was able to improve this conclusion, showing that 
$W(k,h)\le \tfrac{1}{2}k(k+1)+1$ for $k$ sufficiently large. We now show that the latter 
hypothesis on $k$ may be dropped.

\begin{theorem}\label{theorem13.1} When $h$ and $k$ are natural numbers with $h\ge 2$, 
one has $W(k,h)\le \tfrac{1}{2}k(k+1)+1$.
\end{theorem}

\begin{proof} The argument of the proof of \cite[Theorem 1.3]{Woo2012} shows that 
$W(k,h)\le s$ whenever one can establish the bound
$$J_{s,k+1}(X)=o(X^{2s-k(k+1)/2}).$$
However, as a consequence of Corollary \ref{corollary1.3}, for all natural numbers $k$ one 
has $J_{s,k+1}(X)\ll X^{s+\eps}$ whenever $1\le s\le (k+1)(k+2)/2$. Thus, with 
$s=\tfrac{1}{2}k(k+1)+1$, one finds that
$$J_{s,k+1}(X)\ll X^{\eps+1+k(k+1)/2}=X^{\eps-1}\cdot X^{2s-k(k+1)/2}.$$
We therefore conclude that $W(k,h)\le \tfrac{1}{2}k(k+1)+1$, and the proof of the theorem 
is complete.
\end{proof}

The bound obtained in Theorem \ref{theorem13.1} apparently achieves the limits of this 
kind of analytic argument. The best available lower bound for $W(k,h)$ is the trivial bound 
$W(k,h)\ge k+1$, one that for large values of $k$ seems unlikely to represent the true state 
of affairs. For small values of $k$, however, explicit numerical examples show that 
$W(k,2)=k+1$ for $2\le k\le 9$ and $k=11$ (see 
http://euler.free.fr/eslp/eslp.htm).

\section{Analogues of Hua's lemma, and Waring's problem} Estimates of the shape 
(\ref{1.3}) in Theorem \ref{theorem1.1}, and Corollary \ref{corollary1.2}, although 
exhibiting diagonal behaviour, in general fail to provide sufficiently strong mean value 
estimates that they may be applied directly in Diophantine applications such as Waring's 
problem. The idea of augmenting the underlying systems with additional low degree 
equations offers a means of bounding larger moments at a modest cost. As far as the 
author is aware, this idea seems to date from at least as far back as the work of Arkhipov 
and Karatsuba from the 1970's, although we have been unable to identify a suitable 
reference in the literature. In the context of recent advances made via efficient 
congruencing, this idea has also been noted in conference talks (such as the author's Tur\'an 
Conference talk in 2011). Most recently, this idea played a pivotal role in the discussion of 
\cite[Theorem 10]{Bou2017} associated with an analogue of Hua's lemma. Bourgain reports 
that an analogue of the proof of the main conjecture in Vinogradov's mean value theorem 
may be applied to confirm the special case of Corollary \ref{corollary1.2} above in which 
$(d_1,\ldots ,d_k)=(1,2,\ldots ,k-1,d)$ (see especially \cite[equation (6.5)]{Bou2017}). In 
this way, he obtains the bound
\begin{equation}\label{14.1}
\int_0^1\biggl| \sum_{1\le x\le X}e(\alp x^d)\biggr|^{r(r+1)}\d\alp \ll X^{r^2+\eps}
\quad (1\le r\le d),
\end{equation}
which may be regarded as an analogue of Hua's lemma (see \cite{Hua1938b}). In this 
section we apply related ideas to obtain estimates similar in shape to those of estimate 
(\ref{1.3}) of Theorem \ref{theorem1.1} and Corollary \ref{corollary1.2}, though for 
higher moments potentially of use in Diophantine applications. These results not only 
generalise those of Bourgain \cite[Theorem 10]{Bou2017} to systems of equations, but 
also generalise them in the case of a single equation in addition to describing the details 
suppressed in the former treatment (i.e.~the proof of Corollary \ref{corollary1.2} in the 
special case mentioned above).\par

Throughout this section, we suppose that $\varphi_j\in \dbZ[t]$ $(1\le j\le k)$ 
is a system of polynomials with $\deg(\varphi_j)=d_j$ satisfying the condition
\begin{equation}\label{14.2}
1\le d_k<d_{k-1}<\ldots <d_1.
\end{equation}
It is convenient to adopt the convention that $d_0=+\infty$ and $d_{k+1}=0$. We put
$$D=d_1+\ldots +d_k,$$
and when $r\in \dbN$, we define
\begin{equation}\label{14.3}
\Del_{r,\bfd}=\sum_{i=1}^k\max \{ 0,d_i-(r-i+1)\}.
\end{equation}
Finally, we define the exponential sum $F(\bfalp)=F(\bfalp;\bfvarphi)$ by
$$F(\bfalp;\bfvarphi)=\sum_{1\le n\le X}e(\psi(n;\bfalp)),$$
where $\psi(n;\bfalp)=\alp_1\varphi_1(n)+\ldots +\alp_k\varphi_k(n)$.

\begin{theorem}\label{theorem14.1} For each $r\in \dbN$ and $\eps>0$, one has
$$\int_{[0,1)^k}|F(\bfalp;\bfvarphi)|^{r(r+1)}\d\bfalp \ll X^\eps 
\left( X^{r(r+1)/2}+X^{r(r+1)-D+\Del_{r,\bfd}}\right).$$
\end{theorem}

We briefly extract a couple of corollaries.

\begin{corollary}\label{corollary14.2} Suppose that $\varphi\in \dbZ[t]$ is a polynomial of 
degree $d$. Then for each natural number $r$ with $1\le r\le d$, and for each $\eps>0$, 
one has
\begin{equation}\label{14.4}
\int_0^1 \biggl| \sum_{1\le n\le X}e(\alp \varphi(n))\biggr|^{r(r+1)}\d\bfalp 
\ll X^{r^2+\eps}.
\end{equation}
\end{corollary}

In the case $\varphi(t)=t^d$, this conclusion is just the bound (\ref{14.1}) described by 
Bourgain \cite[Theorem 10]{Bou2017}. The earlier work of Arkhipov and Karatsuba 
mentioned above would deliver a similar conclusion with the exponent $r(r+1)$ on the left 
hand side of (\ref{14.4}) replaced by an integer $s_0\sim 4r^2\log r$, and the exponent 
$r^2$ on the right hand side replaced by $s_0-r$. In general, this bound may be applied as 
a substitute for Hua's lemma \cite{Hua1938b}, which shows under the same hypotheses as 
in Corollary \ref{corollary14.2} that
$$\int_0^1\biggl| \sum_{1\le n\le X}e(\alp \varphi(n))\biggr|^{2^r}\d\alp \ll 
X^{2^r-r+\eps}.$$

\begin{corollary}\label{corollary14.3} For each $r\in \dbN$ and $\eps>0$, one has
$$\int_{[0,1)^k}|F(\bfalp;\bfvarphi)|^{r(r+1)}\d\bfalp \ll X^{r(r+1)/2+\eps},\quad 
\text{when $1\le r\le k$},$$
and when $1\le m\le k$ and $d_{m+1}+m\le r\le d_m+m-1$, one has
$$\int_{[0,1)^k}|F(\bfalp;\bfvarphi)|^{r(r+1)}\d\bfalp \ll 
X^{r^2-(m-1)(2r-m)/2-d_{m+1}-\ldots -d_k+\eps}.$$
\end{corollary}

\begin{proof} The first conclusion is immediate from Theorem \ref{theorem14.1}, or 
indeed Theorem \ref{theorem1.1}. As for the second, suppose that $1\le m\le k$ and
$$d_{m+1}+m\le r\le d_m+m-1.$$
Then one sees that the summands in (\ref{14.3}) contribute if and only if $1\le i\le m$, and 
thus
\begin{align*}
\Del_{r,\bfd}&=\sum_{i=1}^m(d_i-(r-i+1))\\
&=(D-d_{m+1}-\ldots -d_k)-mr+m(m-1)/2.
\end{align*}
Hence, we have
$$r(r+1)-D+\Del_{r,\bfd}=r^2-(m-1)r+m(m-1)/2-d_{m+1}-\ldots -d_k,$$
and the desired conclusion follows from Theorem \ref{theorem14.1}.
\end{proof}

\begin{proof}[The proof of Theorem \ref{theorem14.1}] When $1\le r\le k$, the conclusion 
of Theorem \ref{theorem14.1} is immediate from the case $s=r(r+1)/2$ of the bound 
(\ref{1.3}) of Theorem \ref{theorem1.1}. We may therefore suppose that $r>k$. The 
hypothesis (\ref{14.2}) implies that
$$(d_{i+1}-(r-i))-(d_i-(r-i+1))=d_{i+1}-d_i+1\le 0,$$
whence
$$d_{i+1}-(r-i)\le d_i-(r-i+1)\quad (0\le i\le k).$$
Since $d_0-(r+1)>0$ and $d_{k+1}-(r-k)<0$, it follows that there exists an integer $l$ with 
$0\le l\le k$ for which
$$d_{l+1}-(r-l)\le 0\quad \text{and}\quad d_l-(r-l+1)\ge 0.$$
We fix any integer $l$ with this property. One then has
\begin{equation}\label{14.5}
d_i\le r-i+1\quad (l+1\le i\le k).
\end{equation}

Let $e_1,\ldots ,e_{r-k}$ denote the distinct positive integers for which
$$\{e_1,\ldots ,e_{r-k}\}=\{1,2,\ldots ,r-l\}\setminus \{d_{l+1},\ldots ,d_k\}.$$
Notice that the condition (\ref{14.5}) ensures that there are indeed $r-k$ such integers. 
We then define
$$G(\bfalp)=\sum_{1\le n\le X}e(\psi(n;\bfalp)+\alp_{k+1}n^{e_1}+\ldots 
+\alp_rn^{e_{r-k}}).$$
Finally, for the sake of concision, we write $s=r(r+1)/2$.\par

By orthogonality, the mean value
\begin{equation}\label{14.6}
\oint |F(\bfalp;\bfvarphi)|^{2s}\d\bfalp 
\end{equation}
counts the integral solutions of the system of equations
\begin{equation}\label{14.7}
\sum_{i=1}^s(\varphi_j(x_i)-\varphi_j(y_i))=0\quad (1\le j\le k),
\end{equation}
with $1\le \bfx,\bfy\le X$. The mean value (\ref{14.6}) is therefore equal to the number of 
integral solutions of the augmented system of equations (\ref{14.7}) simultaneous with
$$\sum_{i=1}^s(x_i^{e_l}-y_i^{e_l})=h_l\quad (1\le l\le r-k),$$
with $1\le \bfx,\bfy\le X$ and $|h_l|\le sX^{e_l}$. The point here is that the range for the 
auxiliary variables $h_l$ is sufficiently large that this new system accommodates all possible 
choices for $\bfx$ and $\bfy$ satisfying (\ref{14.7}) alone. Thus, by orthogonality and an 
application of the triangle inequality, we find that the mean value (\ref{14.6}) is equal to
\begin{align*}
\sum_{|h_1|\le sX^{e_1}}\ldots \sum_{|h_{r-k}|\le sX^{e_{r-k}}}&
\oint |G(\bfbet)|^{2s}e(-\bet_{k+1}h_1-\ldots -\bet_rh_{r-k})\d\bfbet \\
&\ll X^{e_1+\ldots +e_{r-k}}\oint |G(\bfbet)|^{2s}\d\bfbet .
\end{align*}
Consequently, one has
\begin{equation}\label{14.8}
\oint |F(\bfalp;\bfvarphi)|^{2s}\d\bfalp \ll X^{(r-l)(r-l+1)/2-d_{l+1}-\ldots -d_k}\oint 
|G(\bfbet)|^{2s}\d\bfbet .
\end{equation}

\par Next we observe that the Wronskian of the system of polynomials
\begin{equation}\label{14.9}
\varphi_1(t),\ldots ,\varphi_k(t),t^{e_1},\ldots ,t^{e_{r-k}}
\end{equation}
may be rearranged so that the polynomials are of increasing degree. The leading monomials 
are then non-zero integral multiples of
$$t,t^2,\ldots ,t^{r-l},t^{d_l},t^{d_{l-1}},\ldots,t^{d_1}.$$
The Wronskian of the system (\ref{14.9}) is consequently non-zero, and so it follows from 
the estimate (\ref{1.3}) of Theorem \ref{theorem1.1} that for $s=r(r+1)/2$, one has
$$\oint |G(\bfbet)|^{2s}\d\bfbet \ll X^{s+\eps}.$$
On substituting this conclusion into (\ref{14.8}), we deduce that
$$\oint |F(\bfalp;\bfvarphi)|^{2s}\d\bfalp \ll X^{\Tet+\eps},$$
where
\begin{align*}
\Tet &=\tfrac{1}{2}(r-l)(r-l+1)-d_{l+1}-\ldots -d_k+\tfrac{1}{2}r(r+1)\\
&=r(r+1)-lr+\tfrac{1}{2}l(l-1)-D+d_1+\ldots +d_l\\
&=r(r+1)-D+\sum_{i=1}^l(d_i-r+i-1).
\end{align*}
Thus we have $\Tet=r(r+1)-D+\Del_{r,\bfd}$, where $\Del_{r,\bfd}$ is defined via 
(\ref{14.3}), and the conclusion of the theorem follows.
\end{proof}

One additional idea in the current repertoire of specialists may, on occasion, offer 
improvement in the bounds supplied by Theorem \ref{theorem14.1} and its corollaries. In 
order to describe this idea, we introduce a Hardy-Littlewood dissection. Let 
$\grm=\grm_\kap$ denote the set of real numbers $\alp\in [0,1)$ satisfying the property 
that, whenever $a\in \dbZ$ and $q\in \dbN$ satisfy $(a,q)=1$ and 
$|q\alp-a|\le (2\kap)^{-1}X^{1-\kap}$, then $q>(2\kap)^{-1}X$. Also, denote by 
$\grM_\kap$ the union of the intervals
$$\grM_\kap(q,a)=\{\alp\in [0,1):|q\alp-a|\le (2\kap)^{-1}X^{1-\kap}\},$$
with $0\le a\le q\le (2\kap)^{-1}X$ and $(q,a)=1$. Thus, the unit interval $[0,1)$ is the 
disjoint union of $\grm_\kap$ and $\grM_\kap$. A variant of the proof of 
\cite[Theorem 1.3]{Woo2012b} yields the following conclusion.

\begin{theorem}\label{theorem14.4} Suppose that $d_2\le d_1-2$. Then one has the 
following.
\begin{enumerate}
\item[(i)] When $s$ is a natural number with $2s\ge d_1(d_1+1)$, one has
$$\int_{\grm_{d_1}}\oint |F(\alp_1,\bfbet;\bfvarphi)|^{2s}\d\bfbet \d\alp_1 \ll 
X^{2s-D-1+\eps}.$$
\item[(ii)] When $s$ is a natural number with $2s>d_1(d_1-1)$, one has
$$\int_{\grM_{d_1}}\oint |F(\alp_1,\bfbet;\bfvarphi)|^{2s}\d\bfbet \d\alp_1 \ll 
X^{2s-D+\eps}.$$
\end{enumerate}
\end{theorem}

\begin{proof} Both conclusions follow by applying the argument of the proof of 
\cite[Theorem 2.1]{Woo2012b}, mutatis mutandis. We will be concise with the details in 
order to save space. Initially, we preserve the option of pursuing either case (i) or case (ii) 
of the theorem. Thus, we consider $\grB\subseteq [0,1)$, and we define the mean value
\begin{equation}\label{14.10}
I(\grB)=\int_\grB \oint |F(\alp_1,\bfbet;\bfvarphi)|^{2s}\d\bfbet \d\alp_1.
\end{equation}
For the sake of clarity and concision, we write $\kap$ for $d_1$. The reader should 
experience no difficulty in following the argument of the proof of 
\cite[Theorem 2.1]{Woo2012b} as far as \cite[equation (12)]{Woo2012b}, obtaining the 
bound
\begin{equation}\label{14.11}
I(\grB)\ll X^{(\kap-1)(\kap-2)/2-(D-\kap)}\int_\grB\oint |H(\alp_1,\bftet)|^{2s}
\d\bftet\d\alp_1,
\end{equation}
where
$$H(\alp_1,\bftet)=\sum_{1\le n\le X}e(\alp_1\varphi_1(n)+\tet_1n+\ldots +
\tet_{\kap-2}n^{\kap-2}).$$
Here, we have taken integral linear combinations of equations underlying the inner integral 
of (\ref{14.11}) so as to reduce to the monomials $n^j$ $(1\le j\le \kap-2)$. Such 
man\oe uvring also permits us to assume that $\varphi_1(n)$ takes the shape
$$\varphi_1(n)=An^\kap+Bn^{\kap-1},$$
for suitable integers $A$ and $B$ with $A>0$. Thus, as in \cite[equation (13)]{Woo2012b}, 
we discern that
$$\int_\grB \oint |H(\alp_1,\bftet)|^{2s}\d\bftet \d\alp_1=\sum_{|u|\le sX^{\kap-1}}
\int_\grB \oint |h(\alp_1,\bfbet;X)|^{2s}e(-\bet_{\kap-1}u)\d\bfbet \d\alp_1,$$
where
$$h(\alp_1,\bfbet;X)=\sum_{1\le n\le X}e(\psi(n;\alp_1,\bfbet)),$$
and
$$\psi(n;\alp_1,\bfbet)=\alp_1\varphi_1(n)+\bet_1n+\ldots +\bet_{\kap-1}n^{\kap-1}.$$

\par Write
$$K(\gam)=\sum_{1\le z\le X}e(-\gam z)\quad \text{and}\quad \Ktil(\bfgam)=
\prod_{i=1}^sK(\gam_i)K(-\gam_{s+i}).$$
In addition, put
$$\grh_y(\alp_1,\bfbet;\gam)=\sum_{1\le x\le 2X}e(\psi(x-y;\alp_1,\bfbet)+\gam(x-y))$$
and
\begin{equation}\label{14.12}
\grH_y(\alp_1,\bfbet;\bfgam)=\prod_{i=1}^s \grh_y(\alp_1,\bfbet;\gam_i)\grh_y(-\alp_1,
-\bfbet;-\gam_{s+i}).
\end{equation}
Then, just as in the argument leading to \cite[equation (18)]{Woo2012b}, one finds that
\begin{equation}\label{14.13}
\int_\grB \oint |H(\alp_1,\bftet)|^{2s}\d\bftet \d\alp_1 =\sum_{|u|\le sX^{\kap-1}}
\oint I_u(\bfgam,y)\Ktil(\bfgam)\d\bfgam ,
\end{equation}
where
$$I_u(\bfgam,y)=\int_\grB \oint \grH_y(\alp_1,\bfbet;\bfgam)e(-\bet_{\kap-1}u)\d\bfbet 
\d\alp_1.$$

\par On noting the correction made in \cite{Woo2015c}, the argument leading to 
\cite[equation (22)]{Woo2012b} yields the bound
$$\sum_{|u|\le sX^{\kap-1}}I_u(\bfgam,y)\ll \int_\grB\oint |\grH_0(\alp_1,\bfbet;\bfgam)|
\Psi_y(\alp_1,\bet_{\kap-1})\d\bfbet \d\alp_1,$$
where
$$\Psi_y(\alp_1,\bet_{\kap-1})=\biggl| \sum_{|u|\le sX^{\kap-1}}e(-(\kap Ay-B)u\alp_1-
u\bet_{\kap-1})\biggr|.$$
Thus we deduce that
\begin{equation}\label{14.14}
X^{-1}\sum_{1\le y\le X}\sum_{|u|\le sX^{\kap-1}}I_u(\bfgam,y)\ll \int_\grB \oint 
|\grH_0(\alp_1,\bfbet;\bfgam)|\Psi(\alp_1,\bet_{\kap-1})\d\bfbet \d\alp_1,
\end{equation}
where
\begin{align*}
\Psi(\alp_1,\bet_{\kap-1})&=X^{-1}\sum_{1\le y\le X}\min \{ X^{\kap-1},
\|(\kap Ay-B)\alp_1+\bet_{\kap-1}\|^{-1}\}\\
&\le X^{-1}\sum_{|z|\le \kap AX+|B|}\min \{ X^{\kap-1},\|z\alp_1 +\bet_{\kap-1}\|^{-1}\}.
\end{align*}
Suppose that $\alp_1\in \dbR$, and that $b\in \dbZ$ and $r\in \dbN$ satisfy $(b,r)=1$ 
and $|\alp_1-b/r|\le r^{-2}$. Then, just as in \cite[equation (23)]{Woo2012b}, one 
obtains the estimate
\begin{equation}\label{14.15}
\Psi(\alp_1,\bet_{\kap-1})\ll X^{\kap-1}(X^{-1}+r^{-1}+rX^{-\kap})\log (2r).
\end{equation}

\par It is at this point that our argument diverges according to whether we are in case (i) 
or case (ii). We first consider case (i), in which case we put $\grB=\grm_\kap$. Here, by 
Dirichlet's approximation theorem, given $\alp_1\in \grm_\kap$, one may find $b\in \dbZ$ 
and $r\in \dbN$ with $(b,r)=1$, $|r\alp_1-b|\le (2\kap)^{-1}X^{1-\kap}$ and 
$r\le 2\kap X^{\kap-1}$. The definition of $\grm_\kap$ ensures that $r>(2\kap)^{-1}X$, 
and hence it follows from (\ref{14.15}) that
$$\Psi(\alp_1,\bet_{\kap-1})\ll X^{\kap-2}\log X.$$
From here, the argument leading from \cite[equation (23)]{Woo2012b} to the conclusion of 
the proof of \cite[Theorem 2.1]{Woo2012b} conveys us via (\ref{14.13}) and (\ref{14.14}) 
to the bound
\begin{align*}
\int_{\grm_\kap}\oint &|H(\alp_1,\bftet)|^{2s}\d\bftet \d\alp_1\\
&\ll X^{\kap-2}(\log X)\sup_{\gam\in [0,1)}\oint |\grh_0(\alp_1,\bfbet;\gam)|^{2s}
\d\bfbet \d\alp_1 \oint |\Ktil(\bfgam)|\d\bfgam \\
&\ll X^{\kap-2}(\log X)^{2s+1}J_{s,\kap}(2X).
\end{align*}
When $2s\ge \kap(\kap+1)=d_1(d_1+1)$, we find from Corollary \ref{corollary1.3} that
$$J_{s,\kap}(2X)\ll X^{2s-\kap(\kap+1)/2+\eps},$$
and thus
$$\int_{\grm_\kap}\oint |H(\alp_1,\bftet)|^{2s}\d\bftet \d\alp_1\ll 
X^{2s-\kap(\kap-1)/2-2+\eps}.$$
We therefore conclude from (\ref{14.11}) that
$$I(\grm_\kap)\ll X^{2s-D-1+\eps}.$$
In view of the definition (\ref{14.10}), the first case of the theorem now follows.\par

Our starting point for the proof of case (ii) of the theorem is again the upper bound 
(\ref{14.15}). By appealing to a standard transference principle (see 
\cite[Lemma 14.1]{Woo2015b}), one deduces that whenever $b\in \dbZ$ and $r\in \dbN$ 
satisfy $(b,r)=1$ and $|\alp_1-b/r|\le r^{-2}$, then one has
$$\Psi(\alp_1,\bet_{\kap-1})\ll X^{\kap-1+\eps}(\lam^{-1}+X^{-1}+\lam X^{-\kap}),$$
where $\lam=r+X^\kap |r\alp_1-b|$. When $\alp\in \grM_\kap (r,b)\subseteq \grM_\kap$, 
moreover, one has $r\le X$ and $X^\kap |r\alp_1-b|\le X$, so that $\lam\le 2X$. We 
therefore see that, under such circumstances, one has
$$\Psi(\alp_1,\bet_{\kap-1})\ll X^{\kap-1+\eps}\Phi(\alp_1),$$
where $\Phi(\alp_1)$ is the function taking the value $(q+X^\kap |q\alp_1-a|)^{-1}$, 
when one has $\alp_1\in \grM_\kap(q,a)\subseteq \grM_\kap$, and otherwise 
$\Phi(\alp_1)=0$.\par

It follows from the above discussion that
\begin{align}
\int_{\grM_\kap}\oint |\grh_0(\alp_1,\bfbet;\gam)|^{2s}&\Psi(\alp_1,\bet_{\kap-1})
\d\bfbet \d\alp_1\notag \\
&\ll X^{\kap-1+\eps}\int_{\grM_\kap}\Phi(\alp_1)
\oint |\grh_0(\alp_1,\bfbet;\gam)|^{2s}\d\bfbet \d\alp_1\notag \\
&\ll X^{\kap-1+\eps}\int_{\grM_\kap}\Phi(\alp_1)
\oint |\grh_0(\alp_1,\bfbet;0)|^{2s}\d\bfbet \d\alp_1.\label{14.16}
\end{align}
Moreover, as a consequence of \cite[Lemma 2]{Bru1988}, we find that
\begin{equation}\label{14.17}
\int_{\grM_\kap} \Phi(\alp_1)\oint |\grh_0(\alp_1,\bfbet;0)|^{2s}\d\bfbet \d\alp_1 \ll 
X^{\eps-\kap}(XI_1+I_2),
\end{equation}
where
$$I_1=\int_0^1\oint |\grh_0(\alp_1,\bfbet;0)|^{2s}\d\bfbet \d\alp_1$$
and
$$I_2=\oint |\grh_0(0,\bfbet;0)|^{2s}\d\bfbet .$$
By appealing to Corollary \ref{corollary1.3}, one finds that whenever $2s\ge \kap(\kap-1)$, 
one has
$$I_1\ll J_{s,\kap}(2X)\ll X^{s+\eps}+X^{2s-\kap(\kap+1)/2+\eps}.$$
On the other hand, when $2s\ge \kap(\kap-1)$, it also follows from Corollary 
\ref{corollary1.3} that
$$I_2\ll J_{s,\kap-1}(2X)\ll X^{2s-\kap(\kap-1)/2+\eps}.$$
Provided that $2s\ge \kap(\kap-1)+2$, therefore, we deduce from (\ref{14.17}) that
$$\int_{\grM_\kap}\Phi(\alp_1)\oint |\grh_0(\alp_1,\bfbet;0)|^{2s}\d\bfbet \d\alp_1\ll 
X^{2s-\kap(\kap+1)/2+\eps},$$
whence (\ref{14.16}) yields the estimate
\begin{equation}\label{14.18}
\int_{\grM_\kap}\oint |\grh_0(\alp_1,\bfbet;\gam)|^{2s}\Psi(\alp_1,\bet_{\kap-1})
\d\bfbet \d\alp_1\ll X^{2s-\kap(\kap-1)/2-1+\eps}.
\end{equation}

On recalling (\ref{14.12}), we deduce from (\ref{14.18}) via H\"older's inequality that
$$\int_{\grM_\kap}\oint |\grH_0(\alp_1,\bfbet;\bfgam)|\Psi(\alp_1,\bet_{\kap-1})\d\bfbet 
\d\alp_1\ll X^{2s-\kap(\kap-1)/2-1+\eps},$$
and so (\ref{14.14}) yields the bound
$$X^{-1}\sum_{1\le y\le X}\sum_{|u|\le sX^{\kap-1}}I_u(\bfgam,y)\ll 
X^{2s-\kap(\kap-1)/2-1+\eps}.$$
Consequently, much as in the treatment of the previous case, we deduce from (\ref{14.13}) 
that
\begin{align*}
\int_{\grM_\kap}\oint |H(\alp_1,\bftet)|^{2s}\d\bftet \d\alp_1&\ll 
X^{2s-\kap(\kap-1)/2-1+\eps}\oint |\Ktil(\bfgam)|\d\bfgam \\
&\ll X^{2s-\kap(\kap-1)/2-1+2\eps}.
\end{align*}
On substituting this estimate into (\ref{14.11}), we find that
$$I(\grM_\kap)\ll X^{2s-D+2\eps}.$$
In view of the definition (\ref{14.10}), the conclusion of the theorem now follows in 
case (ii).
\end{proof}

By combining the two conclusions of Theorem \ref{theorem14.4}, we obtain a slight 
improvement on Theorem \ref{theorem14.1} for larger moments.

\begin{theorem}\label{theorem14.5} Suppose that $d_2\le d_1-2$. Put
\begin{equation}\label{14.19}
u=d_1(d_1+1)-\max_{k\le r\le d_1}\frac{d_1(d_1+1)-r(r+1)}{1+\Del_{r,\bfd}}.
\end{equation}
Then whenever $s$ is a real number with $2s\ge \max\{u,d_1(d_1-1)+2\}$, one has
$$\oint |F(\bfalp;\bfvarphi)|^{2s}\d\bfalp \ll X^{2s-D+\eps}.$$
\end{theorem}

\begin{proof} Suppose that the maximum in the definition (\ref{14.19}) occurs for the 
index $r$, and put $\Del=\Del_{r,\bfd}$. Then it follows from H\"older's inequality that
$$\int_{\grm_{d_1}}\oint |F(\alp_1,\bfbet;\bfvarphi)|^u\d\bfbet \d\alp_1 \le 
T_1^{\Del/(1+\Del)}T_2^{1/(1+\Del)},$$
where
$$T_1=\int_{\grm_{d_1}}\oint |F(\alp_1,\bfbet;\bfvarphi)|^{d_1(d_1+1)}\d\bfbet 
\d\alp_1$$
and
$$T_2=\int_0^1\oint |F(\alp_1,\bfbet;\bfvarphi)|^{r(r+1)}\d\bfbet \d\alp_1.$$
Thus, from Theorems \ref{theorem14.1} and \ref{theorem14.4}(i), one finds that
\begin{align*}
\int_{\grm_{d_1}}\oint |F(\alp_1,\bfbet;\bfvarphi)|^u\d\bfbet \d\alp_1&\ll X^{\eps-D}
\left( X^{d_1(d_1+1)-1} \right)^{\Del/(1+\Del)}\left( X^{r(r+1)+\Del}
\right)^{1/(1+\Del)}\\
&\ll X^{u-D+\eps}.
\end{align*}
Meanwhile, provided that $u\ge d_1(d_1-1)+2$, one finds from Theorem 
\ref{theorem14.4}(ii) that
$$\int_{\grM_{d_1}}\oint |F(\alp_1,\bfbet;\bfvarphi)|^u\d\bfbet \d\alp_1\ll X^{u-D+\eps}.
$$
Combining these two estimates, we see that
$$\oint |F(\bfalp;\bfvarphi)|^u\d\bfalp \ll X^{u-D+\eps},$$
and the conclusion of the theorem follows.
\end{proof}

Three corollaries of Theorem \ref{theorem14.5} may be of interest. First we consider the 
mean value
$$I_{s,d}(X)=\int_{[0,1)^{d-1}}\biggl| \sum_{1\le x\le X}e(\alp_dx^d+\alp_{d-2}x^{d-2}
+\ldots +\alp_1x)\biggr|^{2s}\d\bfalp ,$$
in which the argument of the exponential sum is a polynomial of degree $d$ in which there 
is no monomial of degree $d-1$. Hua investigated the problem of determining the smallest 
positive integer $S_d$ having the property that whenever $2s\ge S_d$, then
\begin{equation}\label{14.20}
I_{s,d}(X)\ll X^{2s-(d^2-d+2)/2+\eps}.
\end{equation}
Here, since the sum of the degrees in the associated Diophantine system of equations is 
$$1+2+\ldots +(d-2)+d=(d^2-d+2)/2,$$
the bound Hua sought is essentially best possible for $s\ge (d^2-d+2)/2$. This mean value 
played a critical role in his approach to Vinogradov's mean value theorem for small degrees 
(see \cite[Chapter 5]{Hua1965}). Thus, Hua obtained the bounds
$$S_3\le 10,\quad S_4\le 32,\quad S_5\le 86,\ldots .$$
More recently, as a consequence of progress on Vinogradov's mean value theorem 
stemming from the efficient differencing method, the author obtained the bounds 
$S_k\le 2k^2-2k$ (see \cite[Theorem 11.6]{Woo2013}) and $S_3\le 9$ (see 
\cite[Theorem 1.1]{Woo2015d}).

\par As a consequence of Theorem \ref{theorem14.5}, we obtain new bounds for $S_d$ 
for $d\ge 4$.

\begin{corollary}\label{corollary14.6} When $d\ge 3$, one has $S_d\le d^2$.
\end{corollary}

\begin{proof} Take $d_1=d$ and $d_i=d-i$ $(2\le i\le d-1)$, and put $k=d-1$. We apply 
Theorem \ref{theorem14.5} with $r=d-1$. In such circumstances, we find from (\ref{14.3}) 
that
$$\Del_{r,\bfd}=d-r+\sum_{i=2}^{d-1}\left( (d-i)-(r+1-i)\right)=1.$$
Thus, on putting
$$u=d(d+1)-\frac{d(d+1)-d(d-1)}{2}=d^2,$$
it follows from Theorem \ref{theorem14.5} that the upper bound (\ref{14.20}) holds 
whenever $2s\ge u$. Thus $S_d\le u=d^2$, and the proof of the corollary is complete.
\end{proof}

Next we consider Waring's problem. When $s$ and $d$ are natural numbers, let 
$R_{s,d}(n)$ denote the number of representations of the natural number $n$ as the sum 
of $s$ $d$th powers of positive integers. A formal application of the circle method suggests 
that for $d\ge 3$ and $s\ge d+1$, one should have
\begin{equation}\label{14.21}
R_{s,d}(n)=\frac{\Gam(1+1/d)^s}{\Gam(s/d)}\grS_{s,d}(n)n^{s/d-1}+o(n^{s/d-1}),
\end{equation}
where
$$\grS_{s,d}(n)=\sum_{q=1}^\infty\sum^q_{\substack{a=1\\ (a,q)=1}}\biggl( 
q^{-1}\sum_{r=1}^qe(ar^d/q)\biggr)^se(-na/q).$$
Granted appropriate congruence conditions on $n$, one has $1\ll \grS_{s,d}(n)\ll n^\eps$, 
so that the conjectured relation (\ref{14.21}) is a legitimate asymptotic formula. Let 
$\Gtil(d)$ denote the least integer $t$ with the property that, for all $s\ge t$, and all 
sufficiently large natural numbers $n$, one has the asymptotic formula (\ref{14.21}). 
We numerically sharpen the conclusion $\Gtil(d)\le d^2-d+O(\sqrt{d})$ recorded by 
Bourgain \cite[Theorem 11]{Bou2017}, achieving the limit of the method.\par

We define the integer $\tet=\tet(d)$ by
\begin{equation}\label{14.22}
\tet(d)=\begin{cases} 1,&\text{when $2d+2\ge \lfloor \sqrt{2d+2}\rfloor^2+
\lfloor \sqrt{2d+2}\rfloor$,}\\
2,&\text{when $2d+2< \lfloor \sqrt{2d+2}\rfloor^2+
\lfloor \sqrt{2d+2}\rfloor$.}
\end{cases}
\end{equation}

\begin{corollary}\label{corollary14.7} Let
$$s_0=d(d-1)+\min_{0\le m<d}\frac{2d+m(m-1)}{m+1}.$$
Then, whenever $s\ge s_0$, one has
\begin{equation}\label{14.23}
\int_0^1\biggl| \sum_{1\le n\le X}e(\alp n^d)\biggr|^s\d\bfalp \ll X^{s-d+\eps}.
\end{equation}
Thus, one has $\Gtil(d)\le \lfloor s_0\rfloor +1$, and in particular
$$\Gtil(d)\le d^2-d+2\lfloor \sqrt{2d+2}\rfloor -\tet(d).$$
\end{corollary}

\begin{proof} It is apparent from Corollary \ref{corollary14.2} that $\Del_{r,d}=d-r$ 
$(1\le r\le d)$. Then it follows from Theorem \ref{theorem14.5} that the estimate 
(\ref{14.23}) holds whenever $s\ge \max\{u,d(d-1)+2\}$, where
\begin{align*}
u&=d(d+1)-\max_{0\le m<d}\frac{d(d+1)-(d-m)(d-m+1)}{m+1}\\
&=d(d+1)-\max_{0\le m<d}\frac{2dm-m(m-1)}{m+1}=s_0.
\end{align*}
It is apparent that $s_0\ge d(d-1)+2$, and hence the conclusion (\ref{14.23}) holds 
whenever $s\ge s_0$. Moreover, granted the estimate (\ref{14.23}) in the case $s=s_0$, 
the methods of \cite[Chapter 4]{Vau1997} show that $\Gtil(k)\le \lfloor s_0\rfloor +1$.\par

All that remains to complete the proof of the corollary is the confirmation of the final bound 
on $\Gtil(d)$, and this we obtain by deriving an explicit bound on $s_0$. Take 
$m=\lfloor \sqrt{2d+2}\rfloor$, and define $\ome$ via the relation 
$\sqrt{2d+2}=m+\ome$. Then we have $0\le \ome<1$. With this choice of $m$, one finds 
that
\begin{align*}
\frac{2d+m(m-1)}{m+1}&=
\frac{(m+\ome)^2-2+m(m-1)}{m+1}\\
&=\frac{2m(m+1)-(3-2\ome )(m+1)+(1-\ome)^2}{m+1},\\
&=2m-3+\del,\end{align*}
where
\begin{equation}\label{14.24}
\del=2\ome+\frac{(1-\ome)^2}{m+1}.
\end{equation}
In all circumstances, one has
$$\del<2\ome+(1-\ome)^2/2=(1+\ome)^2/2<2,$$
whence
$$\frac{2d+m(m-1)}{m+1}<2m-1.$$
Then we may take $s_0=w$ with $w<d^2-d+2\lfloor \sqrt{2d+2}\rfloor-1$, yielding the 
bound
$$\Gtil(d)\le \lfloor w\rfloor +1\le d^2-d+2\lfloor \sqrt{2d+2}\rfloor -1.$$

\par On the other hand, provided that $\del<1$, we instead obtain
$$\frac{2d+m(m-1)}{m+1}<2m-2.$$
In such circumstances, we may take $s_0=w$ with 
$w<d^2-d+2\lfloor \sqrt{2d+2}\rfloor-2$, delivering the bound
$$\Gtil(d)\le \lfloor w\rfloor +1\le d^2-d+2\lfloor \sqrt{2d+2}\rfloor -2.$$
It follows from (\ref{14.24}) that $\del<1$ if and only if
$$(2m+2)\ome +(1-\ome)^2<m+1,$$
or equivalently
$$2d+2=(m+\ome)^2<m^2+m.$$
This completes the proof of the corollary.
\end{proof}

We remark that it seems that no improvement in the conclusion of Corollary 
\ref{corollary14.7} is gained by taking $m=\lceil \sqrt{2d+2}\rceil$, so that the stated 
bounds on $\Gtil(d)$ are the sharpest obtainable using this circle of ideas. The formula for 
$s_0$ in the statement of the corollary is equivalent to that given by Bourgain 
\cite[Theorem 11]{Bou2017}. We note that, as shown in \cite[Theorem 4.1]{Woo2012b}, 
the truth of the main conjecture in Vinogradov's mean value theorem (Corollary 
\ref{corollary1.3} or \cite{BDG2016}) delivers the bounds $\Gtil(4)\le 15$, $\Gtil(5)\le 23$, 
$\Gtil(6)\le 34$, $\Gtil(7)\le 47$, $\Gtil(8)\le 61$, $\Gtil(9)\le 78$, $\Gtil(10)\le 97$, and so 
on. The conclusion of Corollary \ref{corollary14.7} matches or improves on these bounds for 
$k\ge 10$.\par

We finish by briefly outlining how Theorem \ref{theorem14.5} may be applied to treat 
systems with one large and a number of smaller degree terms with an efficiency matching 
Corollary \ref{corollary14.7}. We again make use of the definition (\ref{14.22}) of the 
integer $\tet(d)$.

\begin{corollary}\label{corollary14.8} Suppose that
\begin{equation}\label{14.25}
d_2\le d_1-\lfloor \sqrt{2d_1+2}\rfloor-1.
\end{equation}
Then there is a positive number $\tau$ having the property that, with
$$s_0=d_1^2-d_1+2\lfloor \sqrt{2d_1+2}\rfloor -\tet(d_1)-\tau,$$
one has
$$\oint |F(\bfalp;\bfvarphi)|^{s_0}\d\bfalp \ll X^{s_0-D}.$$
\end{corollary}

\begin{proof} Write $m=\lfloor \sqrt{2d_1+2}\rfloor$. Under the hypothesis (\ref{14.25}), 
it is apparent that $d_1\ge d_2+2$. Also, on taking $r=d_1-m$, we find that 
$d_2-r+1\le 0$. We therefore deduce from (\ref{14.3}) that $\Del_{r,\bfd}=d_1-r=m$. 
Thus, just as in the proof of Corollary \ref{corollary14.7}, if we put
$$u=d_1(d_1+1)-\frac{d_1(d_1+1)-(d_1-m)(d_1-m+1)}{m+1},$$
then we find that $u<d_1(d_1-1)+2m-\tet$. Since $u\ge d_1(d_1-1)+2$, then again as in 
the proof of Corollary \ref{corollary14.7}, we see that when $s$ is a real number with 
$2s\ge u$, one has
$$\oint |F(\bfalp;\bfvarphi)|^{2s}\d\bfalp \ll X^{2s-D+\eps}.$$

\par In order to complete the proof of the corollary, we have now only to apply the 
Hardy-Littlewood method. The details are standard, and so we offer only the briefest outline 
of the necessary argument. Take $u_0$ to be a real number satisfying
$$u<u_0<d_1(d_1-1)+2\lfloor \sqrt{2d_1+2}\rfloor-\tet,$$
and put
$$\tau=d_1(d_1-1)+2\lfloor \sqrt{2d_1+2}\rfloor -\tet-u_0.$$
We then take $\del=\tau/(100d_1)$. We define the set of major arcs $\grN$ to be the 
union of the arcs
$$\grN(q,\bfa)=\{ \bfalp \in [0,1)^k:\text{$|\alp_i-a_i/q|\le X^{\del-d_i}$ $(1\le i\le k)$}\},
$$
with
$$0\le \bfa\le q,\quad q\le X^\del\quad \text{and}\quad (q,a_1,\ldots ,a_k)=1.$$
Also, we put 
$\grn=[0,1)^k\setminus \grN$. Then it follows from \cite[Theorem 1.6]{Woo2012} that 
whenever $|F(\bfalp;\bfvarphi)|>X^{1-\del/d_1^3}$, then $\bfalp \in \grN$. Thus
\begin{align*}
\int_\grn |F(\bfalp;\bfvarphi)|^{u_0+\tau}\d\bfalp &\ll (X^{1-\del /d_1^3})^\tau 
\oint |F(\bfalp;\bfvarphi)|^{u_0}\d\bfalp \\
&\ll X^{u_0+\tau-D}.
\end{align*}
Meanwhile, by applying the methods based on \cite[Theorems 1.3 and 2.4]{ACK2004}, just 
as in the proof of Corollary \ref{corollary1.3}, one obtains the bound
$$\int_\grN |F(\bfalp;\bfvarphi)|^{u_0+\tau}\ll X^{u_0+\tau-D}.$$
By combining these estimates, the conclusion of the corollary follows.
\end{proof}

When $k=2$, $d_1=d$ and $d_2=1$, the conclusion of Corollary \ref{corollary14.8} shows 
that the estimate
\begin{equation}\label{14.26}
\int_{[0,1)^2} \biggl| \sum_{1\le n\le X}e(\alp_1n^d+\alp_2n)\biggr|^s\d\bfalp 
\ll X^{s-d-1}
\end{equation}
holds whenever $s>u_0$, for some real number $u_0$ with
$$u_0<d(d-1)+2\lfloor \sqrt{2d+2}\rfloor -\tet(d),$$
provided at least that one has $d-\lfloor \sqrt{2d+2}\rfloor \ge 2$. This condition is satisfied 
for $d\ge 5$, as is readily confirmed. For small values of $d$, the methods of Hua 
\cite{Hua1965} play a role (see also \cite[Lemma 5]{BR2015}), for one has the bound
\begin{equation}\label{14.27}
\int_{[0,1)^2}\biggl|\sum_{1\le n\le X}e(\alp_1n^d+\alp_2n)\biggr|^{2^j+2}\d\bfalp 
\ll X^{2^j-j+1+\eps}\quad (2\le j\le d).
\end{equation}
By applying this estimate as a substitute for Theorem \ref{theorem14.1} in the proof of 
Theorem \ref{theorem14.5}, we find by applying H\"older's inequality that the estimate 
(\ref{14.26}) holds for $s>s_0(d)$, where
$$s_0(4)=15,\quad s_0(5)=23\tfrac{1}{3},\quad s_0(6)=34,\quad s_0(7)=46\tfrac{1}{2},
$$
$$s_0(8)=61\tfrac{1}{5},\quad s_0(9)=78,\quad s_0(10)=96\tfrac{6}{7}.$$
Here, one makes use of the case $j=3$ of the bound (\ref{14.27}) for $d\le 6$, and $j=4$ 
for $7\le d\le 10$. Meanwhile, the work of \cite[Theorem 1.1]{Woo2015d} shows that when 
$d=3$, then (\ref{14.26}) holds whenever $s>9$. We remark that, in this special case 
$k=2$, $d_1=d$ and $d_2=1$, very slightly weaker bounds could be extracted from Table 1 of the paper \cite{ACHK2017} that was submitted to the arXiv just prior to the 
submission of this memoir. The underlying minor arc bounds can be seen to be morally 
equivalent, though the major arc treatment differs.

\section{Vinogradov's mean value theorem in number fields} The application of the 
Hardy-Littlewood (circle) method in number fields is frequently complicated by the 
dependence of exponential sum estimates on the degree of the ambient field extension. In 
Diophantine problems of all but the lowest degrees $d$, existing methods for circumventing 
such difficulties demand the availability of a number of variables exponentially large in terms 
of $d$. Thus, for example, Birch \cite{Bir1961} has shown that in any algebraic number 
field, the rational solutions of a diagonal form of degree $d$ in $s$ variables have the 
expected asymptotic density whenever $s\ge 2^d+1$. The approach of Birch owes its 
success to the efficiency of Hua's lemma with $2^d$ variables. Indeed, so efficient is the 
latter that, equipped with even a weak version of Weyl's inequality having poor dependence 
on the degree of the field extension at hand, a satisfactory outcome can be derived with 
just one additional variable. Hitherto, such efficiency has been absent from versions of 
Vinogradov's mean value theorem that might otherwise be expected to deliver superior 
bounds for the number of variables (see, for example, the work of K\"orner \cite{Kor1962} 
and Eda \cite{Eda1967}). Our primary goal in this section is to establish such an efficient 
version of Vinogradov's mean value theorem in number fields, thereby opening access to 
sharp Diophantine applications in number fields. Indeed, we establish the main conjecture in 
number fields, and this delivers an analogue of Birch's theorem whenever $s\ge d^2+d+1$.

\par In order to be more concrete concerning our conclusions, we require some notation, 
beginning with the infrastructure for algebraic number fields. We refer the reader to 
\cite{Wan1991} for an introduction to the circle method in number fields. We consider an 
algebraic extension $K$ of degree $n$ over $\dbQ$. Let $K^{(l)}$ $(1\le l\le n_1)$ be 
the real conjugate fields associated with $K$, and let $K^{(m)}$ and $K^{(m+n_2)}$ 
$(n_1+1\le m\le n_1+n_2)$ be the pairs of complex conjugate fields associated with $K$. 
Here, one has $n_1+2n_2=n$. We write $\grO_K$ for the ring of integers of $K$, and we 
fix a basis $\Ome =\{\ome_1,\ldots ,\ome_n\}$ for $\grO_K$ over $\dbZ$. We then denote 
by $\calB(X)\subset \grO_K$ the unit cube
$$\calB(X)=\left\{ r_1\ome_1+\ldots +r_n\ome_n: r_i\in 
[-{\textstyle{\frac{1}{2}}}X^{1/n},{\textstyle{\frac{1}{2}}}X^{1/n})\cap \dbZ\ 
(1\le i\le n)\right\}.$$
It is apparent that $\text{card}(\calB(X))\asymp X$.\par

When $\gam\in K$, we denote by $\gam^{(i)}$ the conjugates of $\gam$, where 
$\gam^{(i)}\in K^{(i)}$ $(1\le i\le n)$. Then, as usual, we define the trace map 
$\text{Tr}=\text{Tr}_{K/\dbQ}$ and norm map $\text{N}=\text{N}_{K/\dbQ}$ by taking
$$\text{Tr}(\gam)=\gam^{(1)}+\ldots +\gam^{(n)}\quad \text{and}\quad 
\text{N}(\gam)=\gam^{(1)}\cdots \gam^{(n)}.$$
When $\gam_j\in K$ and $\tet_j\in \dbR$ for $1\le j\le n$, and
$$\lam(\bftet)=\tet_1\gam_1+\ldots +\tet_n\gam_n,$$
we define
$$\lam^{(i)}(\bftet)=\tet_1\gam_1^{(i)}+\ldots +\tet_n\gam_n^{(i)}\quad (1\le i\le n).$$
The trace map on $K$ can then be extended by defining
$$\text{Tr}(\lam(\bftet))=\lam^{(1)}(\bftet)+\ldots +\lam^{(n)}(\bftet).$$
The analogue of the function $e(\alp)=e^{2\pi i\alp}$ in this number field setting is then 
defined by taking $E(\lam(\bftet))=e(\text{Tr}(\lam(\bftet)))$.\par

Next, let $\grd^{-1}$ denote the inverse different, so that
$$\grd^{-1}=\{\eta\in K:\text{$\text{Tr}(\eta \xi)\in \dbZ$ for all $\xi\in \grO_K$}\}.$$ 
There is a basis $\Rho=\{ \rho_1,\ldots ,\rho_n\}$ of $\grd^{-1}$ dual to $\Ome$ having 
the property that
$$\text{Tr}(\rho_i\ome_j)=\begin{cases}1,&\text{when $i=j$,}\\
0,&\text{when $i\ne j$.}\end{cases}.$$
As a special case of the linear form $\lam(\bftet)$ defined above, we define 
$\alp=\alp(\bftet)$ by
$$\alp(\bftet)=\tet_1\rho_1+\ldots +\tet_n\rho_n.$$
When such a form occurs in an $n$-fold integral, we use the symbol $\d\alp$ to denote the 
$n$-fold differential $\d\tet_1\ldots \d\tet_n$. It is convenient then to write $\dbT$ for 
$[0,1)^n$. Now that we are equipped with this notation, we may record the fundamental 
orthogonality relation that underpins the circle method in number fields. Thus, when 
$\gam\in \grO_K$, one has
\begin{equation}\label{15.1}
\int_\dbT E(\alp \gam)\d\alp =\begin{cases} 1,&\text{when $\gam=0$,}\\
0,&\text{when $\gam\in \grO_K\setminus \{0\}$.}\end{cases}
\end{equation}

We may now announce the analogue of Theorem 1.1 in number fields.

\begin{theorem}\label{theorem15.1} Suppose that $\varphi_j\in \grO_K[t]$ $(1\le j\le k)$ is 
a system of polynomials with $W(t;\bfvarphi)\ne 0$. Let $s$ be a positive real number with 
$s\le k(k+1)/2$. Also, suppose that $(\gra_\nu)_{\nu\in \grO_K}$ is a sequence of 
complex numbers. Then for each $\eps>0$, one has
$$\int_{\dbT^k}\biggl| \sum_{\nu\in \calB(X)}\gra_\nu 
E(\alp_1\varphi_1(\nu)+\ldots +\alp_k\varphi_k(\nu))\biggr|^{2s}\d\bfalp \ll X^\eps 
\biggl( \sum_{\nu\in \calB(X)}|\gra_\nu|^2\biggr)^s.$$
In particular, one has
$$\int_{\dbT^k}\biggl| \sum_{\nu\in \calB(X)}E(\alp_1\varphi_1(\nu)+\ldots 
+\alp_k\varphi_k(\nu))\biggr|^{2s}\d\bfalp \ll X^{s+\eps}.$$
\end{theorem}

Throughout this section, we adopt the convention that implicit constants in Vinogradov's 
notation $\ll$ and $\gg$ may depend on $s$, $k$, $\bfvarphi$, $K$, $\Ome$, and also the 
small positive number $\eps$. As an immediate consequence of the orthogonality relation 
(\ref{15.1}), one obtains the following corollary.

\begin{corollary}\label{corollary15.2} When $s\in \dbN$, denote by 
$\calN_{s,\bfvarphi}(X;K)$ 
the number of solutions of the system of equations
$$\sum_{i=1}^s\left( \varphi_j(x_i)-\varphi_j(y_i)\right)=0\quad (1\le j\le k),$$
with $x_i,y_i\in \calB(X)$. Then whenever $W(t;\bfvarphi)\ne 0$ and $s\le k(k+1)/2$, one 
has
$$\calN_{s,\bfvarphi}(X;K)\ll X^{s+\eps}.$$
\end{corollary}

Finally, when $k\in \dbN$ and $s>0$, we define
$$J_{s,k}(X;K)=\int_{\dbT^k}\biggl| \sum_{\nu\in \calB(X)}E(\alp_1\nu+\ldots 
+\alp_k\nu^k)\biggr|^{2s}\d\bfalp .$$
Theorem \ref{theorem15.1} delivers an analogue of the main conjecture in Vinogradov's 
mean value theorem for algebraic number fields.

\begin{corollary}\label{corollary15.3} Suppose that $k\in \dbN$ and $s>0$. Then for each 
$\eps>0$, one has
$$J_{s,k}(X;K)\ll X^\eps (X^s+X^{2s-k(k+1)/2}).$$
\end{corollary}

The literature concerning Vinogradov's mean value theorem in number fields begins with the 
work of K\"orner \cite{Kor1962} more than half a century ago. When $[K:\dbQ]=n$ and 
$s\ge \frac{1}{4}nk(k+1)+rk$ $(r\in \dbN)$, K\"orner \cite[Satz 1]{Kor1962} delivers an 
estimate tantamount to
$$J_{s,k}(X;K)\ll X^{2s-\frac{1}{2}k(k+1)+\eta_{s,k}}(\log X)^r,$$
where $\eta_{s,k}=\frac{1}{2}k(k+1)(1-1/k)^r$. This was improved by Eda \cite{Eda1967} 
(see also \cite[Lemma 4]{Eda1975}) to the extent that the power of $\log X$ may be 
deleted, and the condition on $s$ relaxed to
$$s\ge \frac{n}{n-1}k(k+1)+rk-1.$$
Recent work of Kozlov \cite{Koz2001} and Sorokin \cite{Sor2007} provides some slight 
improvement in these results for the special case $K=\dbQ(\sqrt{-1})$, although their 
estimates are constrained to possess the same salient features. Thus, in all of this previous 
work, the exponent $\eta_{s,k}$ behaves roughly like $k^2e^{-s/k^2}$. In consequence, 
one must take $s$ to be at least as large as $k^2(2\log k+\log \log k+c)$, for a suitable 
positive constant $c$, in order that the quality of available estimates of Weyl type permit 
sufficient control of the mean value implicit in $J_{s,k}(X;K)$ necessary for applications. 
Indeed, the dependence of available Weyl estimates on the degree $n$ of the field 
extension may necessitate that this constant $c$ grow with $n$ at least as fast as $\log n$. 
By contrast, Corollary \ref{corollary15.3} permits full control to be exercised as soon as 
$s\ge k(k+1)/2$. Not only is the dependence on the degree of the ambient field extension 
entirely removed, but the dependence on $k$ is also substantially improved.\par

The conclusion of Theorem \ref{theorem15.1} follows from an analogue of Theorem 
\ref{theorem3.1}, as we now describe. Since the details are strikingly similar to those in the 
situation over $\dbZ$ described in \S\S3-12, we will economise on space by indicating only 
the places in the argument where special care must be taken. We again take $k$ to be an 
integer with $k\ge 1$, and consider polynomials $\varphi_1,\ldots ,\varphi_k\in \grO_K[t]$. 
Throughout the argument, the ring of integers $\grO_K$ replaces $\dbZ$. Let $\grp=(\pi)$ 
be a prime ideal of $\grO_K$, and put $p={\rm N}(\pi)$. We will be implicitly working in the 
$\grp$-adic field completing $K$ at the place $\grp$. We assume throughout that 
$p>(k!)^n$. The definitions of \S3 must now be made, mutatis mutandis, where throughout 
we emphasise that $\dbZ$ is replaced by $\grO_K$, congruences$\nmod{p^h}$ are 
replaced by congruences$\nmod{\grp^h}$, and the function $e(z)$ is replaced by $E(z)$. 
The definitions (\ref{1.9}) and (\ref{3.5}) must be adjusted to this new setting. First, when 
$F:\dbT^k\rightarrow \dbC$ is integrable, we write
$$\oint F(\bfalp)\d\bfalp =\int_{\dbT^k}F(\bfalp)\d\bfalp .$$
Next, when $B$ is a positive integer, we define
\begin{equation}\label{15.2}
\oint_{\grp^B}F(\bfalp)\d\bfalp =p^{-kB}\sum_{u_1\nmod{\grp^B}}\ldots 
\sum_{u_k\nmod{\grp^B}}F(\bfu \pi^{-B}).
\end{equation}
Some words of explanation are in order here. First, the summations on the right hand side 
of (\ref{15.2}) are taken over complete sets of residues modulo $\grp^B$. Next, each 
coordinate $u_i\pi^{-B}$ of the argument on the right hand side of (\ref{15.2}) may be 
written in the shape
$$u_i=\bet_{i1}\rho_1+\ldots +\bet_{in}\rho_n,$$
with $\bet_{ij}\in \dbR$, and we then reduce each coefficient $\bet_{ij}$ modulo $1$. With 
this convention, we may regard $\bfu \pi^{-B}$ as belonging to $\dbT^k$. Integrals with 
subscripts $p^B$ in \S\S3-12 are now replaced by this newly defined integral with subscript 
$\grp^B$, in the obvious fashion. With this definition, one may verify that for $\nu \in 
\grO_K$, one has the orthogonality relation
$$\oint_{\grp^B}E(\alp \nu)\d\alp =\begin{cases} 1,&\text{when 
$\nu\equiv 0\mmod{\grp^B}$,}\\
0,&\text{when $\nu\not\equiv 0\mmod{\grp^B}$.}\end{cases}$$

\par Equipped with these modified definitions, a version of Theorem \ref{theorem3.1} may 
now be stated in the number field setting. We recall in advance the definitions (\ref{3.6}) 
and (\ref{3.8}) in their modified manifestations.

\begin{theorem}\label{theorem15.4} Let $K$ be an algebraic extension of $\dbQ$ with 
$[K:\dbQ]<\infty$. Suppose that $k\in \dbN$, and that $\grp=(\pi)$ is a prime ideal of 
$\grO_K$ with $p={\rm N}_{K/\dbQ}(\pi)>(k!)^n$. Then one has $\lam(k(k+1)/2,k)=0$.
\end{theorem}

We remark that the condition $p>(k!)^n$ is imposed in order to ensure that any factor 
$k!$, occurring in a binomial expansion en route, is necessarily non-zero modulo $\grp$. 
If one were to have $k!\equiv 0\mmod{\grp}$, then one would have that ${\rm N}(\pi)$ 
divides ${\rm N}(k!)$ over $\dbZ$. But by the aforementioned hypothesis, one has 
${\rm N}(\pi)>(k!)^n\ge {\rm N}(k!)$, leading to a contradiction, and so the desired 
conclusion follows.

\begin{corollary}\label{corollary15.5} Let $K$ be an algebraic extension of $\dbQ$ with 
$[K:\dbQ]<\infty$. Suppose that $k\in \dbN$, and that $\grp=(\pi)$ is a prime ideal of 
$\grO_K$ with $p={\rm N}_{K/\dbQ}(\pi)>(k!)^n$. Suppose that $\tau>0$ and $\eps>0$. 
Let $B$ be sufficiently large in terms of $n$, $k$, $\tau$ and $\eps$. Put $s=k(k+1)/2$ and 
$H=\lceil B/k\rceil$. Then for every $\bfvarphi\in \Phi_\tau(B)$, and every sequence 
$(\gra_\nu)\in \dbD_0$, one has
$$U_{s,k}^B(\bfgra)\ll p^{B\eps}U_{s,k}^{B,H}(\bfgra).$$
\end{corollary}

The proof of Theorem \ref{theorem15.4} and its corollary follow just as in the 
corresponding argument of \S\S3-10 above in the rational case. Given our adjustments in 
notation, the argument follows verbatim, provided that the reader exercises care in ensuring 
that these notational perturbations are correctly administered. We therefore move on 
directly to consider the proof of Theorem \ref{theorem15.1}.

\begin{proof}[The proof of Theorem \ref{theorem15.1}] In all essentials, the proof of 
Theorem \ref{theorem15.1} follows the proof of Theorem \ref{theorem1.1} given in \S11 
and \S12, mutatis mutandis, and so we shall be very brief concerning the details. Here, with 
the notational modifications in hand, the only parts of these sections that require further 
discussion are located in \S12. Let $\calZ$ denote the set of zeros of $W(t;\bfvarphi)$ lying 
in $\grO_K$, and let $X$ be sufficiently large in terms of $\bfvarphi$, $k$, $\eps$, $K$ and 
$\Ome$. We suppose that $s=k(k+1)/2$. We note that one again has
$$\text{card}(\calZ)\le \text{deg}(W(t;\bfvarphi))\ll 1.$$
Define
$$F(\bfalp;X)=\rho_0^{-1}\sum_{\nu \in \calB(X)}\gra_\nu E(\psi(\nu;\bfalp))$$
and
$$F_0(\bfalp;X)=\rho_0^{-1}\sum_{\substack{\nu \in \calB(X)\\ \nu \not\in \calZ}}\gra_\nu 
E(\psi(\nu;\bfalp)).$$
We may suppose that the sequence $(\gra_\nu)$ satisfies the property that $\gra_\nu=0$ 
whenever $\nu\not\in \calB(X)$. Our task is to cover the exponential sum $F_0(\bfalp;X)$ 
by analogous exponential sums with variables constrained by appropriate non-singularity 
conditions modulo $\grp$, for suitable prime ideals $\grp=(\pi)$ in $\grO_K$.\par

Given a solution $\bfx,\bfy\in (\calB(X)\setminus \calZ)^s$ of the system (\ref{1.4}), the 
algebraic integer
$$\prod_{i=1}^sW(x_i;\bfvarphi)W(y_i;\bfvarphi)$$
is non-zero, and has norm bounded above by $CX^D$, for some $C>0$ depending at most 
on $K$, $s$ and the coefficients of $\bfvarphi$, and $D$ a positive integer with
$$D\le 2s\sum_{j=1}^k\text{deg}(\varphi_j).$$
Let $\calP$ denote the set of elements $\pi\in \grO_K$ having the property that $(\pi)$ is a 
prime ideal satisfying the condition
$$(\log X)^2<{\rm N}(\pi)\le 3(\log X)^2.$$
By the prime ideal theorem in number fields (see \cite{Lan1903}), when $X$ is sufficiently 
large, the number of such elements is at least $\tfrac{1}{2}(\log X)^2/\log \log X$. One 
therefore has
$${\rm N}\biggl( \prod_{\pi\in \calP}\pi\biggr) >(\log X)^{(\log X)^2/\log\log X}>
(CX^D)^n.$$
Thus we deduce that for each solution $\bfx,\bfy\in (\calB(X)\setminus \calZ)^s$ of 
(\ref{1.4}) counted by the mean value
$$\oint |F_0(\bfalp;X)|^{2s}\d\bfalp,$$
there exists $\pi\in \calP$ with
$$\prod_{i=1}^sW(x_i;\bfvarphi)W(y_i;\bfvarphi)\not\equiv 0\mmod{\grp},$$
in which we write $\grp=(\pi)$. In particular, one has
\begin{equation}\label{15.3}
\oint |F_0(\bfalp;X)|^{2s}\d\bfalp \le \sum_{\pi \in \calP}\oint |F_\pi(\bfalp;X)|^{2s}
\d\bfalp ,
\end{equation}
where
$$F_\pi (\bfalp;X)=\rho_0^{-1}
\sum_{\substack{\nu\in \calB(X)\\ W(\nu;\bfvarphi)\not\equiv 0\mmod{\grp}}}|
\gra_\nu|\,E(\psi(\nu;\bfalp)).$$

\par The argument of \S12 now resumes. Let $\tau>0$ be sufficiently small in terms of $s$ 
and $k$. We take
$$B=\left\lceil \frac{k\log (pX)}{\log p}\right\rceil,\quad c=\lceil \tau B\rceil \quad 
\text{and}\quad H=\lceil B/k\rceil -c,$$
which ensures that $pX\le {\rm N}(\pi^{H+c})\le p^2X$. Then we may suppose that $B$ is 
sufficiently large in terms of $\tau$, as well as $s$, $k$ and $\eps$. By orthogonality, one 
has
$$\oint |F_\pi(\bfalp;X)|^{2s}\d\bfalp \le \oint_{\grp^B}|F_\pi(\bfalp;X)|^{2s}\d\bfalp .$$
Thus, as a consequence of Corollary \ref{corollary15.5}, just as in the argument of the 
proof of Theorem \ref{theorem11.1} leading to (\ref{11.9}), one finds that there is a 
$\grp^c$-spaced system $\bfPsi$ and complex sequence $\bfgrc$ with 
$|\grc_y|=|\gra_{\pi^cy+\xi}|$ for which
$$\oint_{\grp^B}|F_\pi(\bfalp;X)|^{2s}\d\bfalp \ll p^{sc+B\eps}
\rho_0(\bfgra)^{-2}\sum_{\xi\nmod{\grp^c}}\rho_c(\xi)^2U_{s,k}^{B-kc,H,\bfPsi}
(\bfgrc).$$

\par Notice that whenever $\nu,\nu'\in \calB(X)$, one has $\nu-\nu'\in \calB(2X)$, and 
hence ${\rm N}(\nu-\nu')\ll X$. Then since $X$ is sufficiently large and $p>(\log X)^2$, we 
obtain the implication
\begin{equation}\label{15.4}
\nu\equiv \nu'\mmod{\grp^{H+c}}\quad \implies \quad \nu=\nu'.
\end{equation}
Consequently, just as in the concluding paragraph of the proof of Theorem 
\ref{theorem11.1}, we find that
$$U_{s,k}^{B-kc,H,\bfPsi}(\bfgrc)\ll (1+X/p^{c+H})^s\ll 1.$$
Thus we deduce that
$$\oint_{\grp^B}|F_\pi(\bfalp;X)|^{2s}\d\bfalp \ll p^{sc+B\eps}\rho_0(\bfgra)^{-2}
\sum_{\xi\nmod{\grp^c}}\rho_c(\xi)^2\ll p^{(2s\tau+\eps)B}.$$
Since $\tau$ was chosen sufficiently small in terms of $s$ and $k$, it follows that for each 
positive number $\del$, one has
$$\oint |F_\pi(\bfalp;X)|^{2s}\d\bfalp \ll X^\del,$$
and thus, on recalling (\ref{15.3}),
\begin{align*}
\oint |F(\bfalp;X)|^{2s}\d\bfalp &\ll 1+\oint |F_0(\bfalp;X)|^{2s}\d\bfalp \\
&\ll 1+X^\del\sum_{\pi \in \calP}1\ll X^{2\del}.
\end{align*}
The conclusion of the theorem follows on recalling the definitions of $F(\bfalp;X)$ and the 
weight $\rho_0$.
\end{proof}

It is difficult to resist announcing an easy consequence of Theorem \ref{theorem15.1} that 
follows by a straightforward application of the circle method in number fields. This supplies 
a Hasse principle for diagonal forms.

\begin{theorem}\label{theorem15.6} Let $K$ be an algebraic extension of $\dbQ$ of finite 
degree. Let $k,s\in \dbN$, and suppose that $s\ge k^2+k+1$. Suppose also that 
$a_1,\ldots ,a_s\in K$, and that the equation
$$a_1x_1^k+\ldots +a_sx_s^k=0$$
has non-zero solutions in every completion $K_v$ of $K$. Then this equation has a solution 
$\bfx\in K^s\setminus \{{\mathbf 0}\}$.
\end{theorem}

\begin{proof} Write
$$G(\bfalp;X)=\sum_{\nu\in \calB(X)}E(\alp_1\nu+\ldots +\alp_k\nu^k).$$
Then it follows from the triangle inequality that
\begin{align*}
\int_\dbT \biggl| \sum_{\nu\in \calB(X)}E(\alp \nu^k)\biggr|^{2s}\d\alp &=
\sum_\bfh \int_{\dbT^k}|G(\bfalp;X)|^{2s}E(-\alp_1h_1-\ldots -\alp_{k-1}h_{k-1})\d\bfalp 
\\
&\ll X^{k(k-1)/2}J_{s,k}(X;K),
\end{align*}
in which the $(k-1)$-tuples $\bfh$ are summed over the boxes
$$h_j\in 2sC\calB(X^j)\quad (1\le j\le k-1),$$
for a positive number $C$ sufficiently large in terms of $\Ome$ and $k$. Thus, when 
$s\ge k(k+1)$, one finds from Corollary \ref{corollary15.3} that
$$\int_\dbT \biggl| \sum_{\nu\in \calB(X)}E(\alp \nu^k)\biggr|^{2s}\d\alp \ll 
X^{2s-k+\eps}.$$
Using this estimate as a substitute for \cite[Lemma 2]{Bir1961} in the proof of 
\cite[Theorem 3]{Bir1961}, the proof of the theorem follows via a standard application of 
the circle method in algebraic number fields.
\end{proof}

\begin{corollary}\label{corollary15.7} Let $L$ be an algebraic extension of $\dbQ$, possibly 
of infinite degree. Let $k,s\in \dbN$ with $k$ odd, and suppose that 
$s\ge \exp(8(\log k)^2)$. Suppose also that $a_1,\ldots ,a_s\in L$. Then the equation
\begin{equation}\label{15.5}
a_1x_1^k+\ldots +a_sx_s^k=0
\end{equation}
has a solution $\bfx\in L^s\setminus \{{\mathbf 0}\}$. When $L$ is a totally 
imaginary extension of $\dbQ$, the same conclusion holds also for even $k$
\end{corollary}

\begin{proof} Since $L$ is an algebraic extension of $\dbQ$, the coefficients 
$a_1,\ldots ,a_s$ are algebraic. Put $K=\dbQ(a_1,\ldots ,a_s)$. Then $K:\dbQ$ is a finite 
algebraic extension of $\dbQ$, and it follows from \cite[Theorem 1]{BGR2008}, just as in 
the proof of \cite[Theorem 4.4]{Woo2016b}, that in every completion $K_v$ of $K$ the 
equation (\ref{15.5}) has a non-zero solution. Note that when the place $v$ is infinite, this 
is trivially inferred from the hypothesis that, either $k$ is odd, or else $L$ is a totally 
imaginary extension of $\dbQ$ and $k$ is even. It therefore follows from Theorem 
\ref{theorem15.6} that the equation (\ref{15.5}) possesses a non-zero $K$-rational 
solution, and hence also a non-zero $L$-rational solution.
\end{proof}

\begin{corollary}\label{corollary15.8} Let $L$ be a totally imaginary algebraic extension of 
$\dbQ$, possibly of infinite degree. Let $k,s,r\in \dbN$ with $k\ge 3$, and suppose that
$$s>r^{2^{k-1}}\exp(2^{k+2}(\log k)^2).$$
Then, whenever $F_i\in L[x_1,\ldots ,x_s]$ $(1\le i\le r)$ are homogeneous of degree $k$, 
the system of equations
$$F_i(\bfx)=0\quad (1\le i\le r)$$
possess a simultaneous solution $\bfx\in L^s\setminus \{{\mathbf 0}\}$.
\end{corollary}

\begin{proof} One may follow the argument of the proof of \cite[Corollary 1.3]{Woo1998}, 
mutatis mutandis, noting the proof of \cite[Theorem 4.4]{Woo2016b}, and substituting 
Corollary \ref{corollary15.7} into \cite[Theorem 1]{Woo1998}.
\end{proof}

The conclusion of this corollary provides an explicit version of a theorem of Peck 
\cite{Pec1949}. We note that \cite[Corollary 1.3]{Woo1998} establishes a similar conclusion 
subject to the stronger constraint $s>r^{2^{k-1}}\exp(2^kk)$. We intend to explore 
further applications of Theorem \ref{theorem15.1} in a later paper.

\section{Multidimensional analogues via restriction of scalars} The 
conclusions of \S15 may be employed to establish a class of mean value estimates 
associated with multidimensional systems. In simplest terms, what we have in mind is that a 
mean value estimate of the shape given in Corollary \ref{corollary15.2}, may be 
reinterpreted as a multidimensional mean value estimate over a lower degree field 
extension of $\dbQ$. Perhaps this is best illustrated with a concrete example. Thus, consider 
the Vinogradov system of degree $3$ over $K=\dbQ(\sqrt{-2})$. Theorem 
\ref{theorem15.1} shows that whenever $s\ge 1$, the number $J_{s,3}(X;K)$ of solutions 
of the system
\begin{equation}\label{16.1}
\sum_{i=1}^s(x_i^j-y_i^j)=0\quad (1\le j\le 3),
\end{equation}
with $\bfx,\bfy\in \calB(X)^s$, satisfies
$$J_{s,3}(X;K)\ll X^\eps(X^s+X^{2s-6}).$$
However, for each such solution $\bfx,\bfy$, one may write
$$x_i=u_i+v_i\sqrt{-2}\quad \text{and}\quad y_i=z_i+w_i\sqrt{-2},$$
with $\bfu,\bfv,\bfz,\bfw\in [-\tfrac{1}{2}X^{1/2},\tfrac{1}{2}X^{1/2})^s\cap \dbZ^s$. 
Define
$$\phi_3(x,y)=x^3-6xy^2,\quad \psi_3(x,y)=3x^2y-2y^3,$$
$$\phi_2(x,y)=x^2-2y^2,\quad \psi_2(x,y)=xy,$$
$$\phi_1(x,y)=x,\quad \psi_1(x,y)=y.$$
Then by expanding the expressions
$$(u_i+v_i\sqrt{-2})^j\quad \text{and}\quad (z_i+w_i\sqrt{-2})^j,$$
for $1\le j\le 3$, and writing the result in terms of the integral coordinate basis 
$\{1,\sqrt{-2}\}$, one sees that (\ref{16.1}) holds if and only if the system
\begin{align*}
\sum_{i=1}^s(\phi_j(u_i,v_i)-\phi_j(z_i,w_i))&=0,\\
\sum_{i=1}^s(\psi_j(u_i,v_i)-\psi_j(z_i,w_i))&=0,
\end{align*}
is satisfied simultaneously for $1\le j\le 3$.\par

Denote by $J_s(Y;\bfvarphi,\bfpsi)$ the number of integral solutions of the latter system 
with $1\le \bfu,\bfv,\bfz,\bfw\le Y$. Then from the estimate
$$J_{s,3}(X;\dbQ(\sqrt{-2}))\ll X^\eps(X^s+X^{2s-6}),$$
available via Corollary \ref{corollary15.3}, we deduce that
$$J_s(Y;\bfvarphi,\bfpsi)\ll Y^\eps (Y^{2s}+Y^{4s-12}).$$
This estimate delivers the main conjecture for this two dimensional system.\par

In order to describe this phenomenon in wider generality, we introduce some notation. Let 
$K$ be an algebraic extension of $\dbQ$ with $[K:\dbQ]=d$. Let $L$ be an algebraic 
extension of $K$ with $[L:K]=n$, and let $\grO_{L/K}$ denote the ring of integers 
associated with the field extension $L:K$. Write $\{\ome_1,\ldots ,\ome_n\}$ for an 
$\grO_K$-integral coordinate basis of $L:K$. Consider polynomials 
$\varphi_1,\ldots ,\varphi_k\in \grO_{L/K}[t]$ with $W(t;\bfvarphi)\ne 0$. We say that the 
system of polynomials $\psi_{lj}(t_1,\ldots ,t_n)\in \grO_{K/\dbQ}[\bft]$ $(1\le l\le n)$ is 
{\it generated from $\bfvarphi$ by restriction of $L$ down to $K$} when, for some 
$\lam_{i1},\ldots ,\lam_{in}\in K$, with $\text{det}(\lam_{il})_{1\le i,l\le n}\ne 0$, one has
$$\varphi_j(t_1\ome_1+\ldots +t_n\ome_n)=
\sum_{i=1}^n\ome_i\sum_{l=1}^n\lam_{il}\psi_{lj}(t_1,\ldots ,t_n)\quad (1\le j\le k).$$

\par We are now equipped to announce an analogue of the second conclusion of 
Theorem \ref{theorem1.1} for certain multidimensional systems. Here, we interpret 
$\calB(X)=\calB_K(X)$, as before, as a subset of $\grO_{K/\dbQ}$ relative to a fixed 
coordinate basis for $\grO_{K/\dbQ}$ over $\dbZ$.

\begin{theorem}\label{theorem16.1} Let $L:K:\dbQ$ be a tower of algebraic field 
extensions with $[L:\dbQ]<\infty$ and $[L:K]=n$. Given a system of polynomials 
$\varphi_1,\ldots ,\varphi_k\in \grO_{L/K}[t]$ with $W(t;\bfvarphi)\ne 0$, suppose that 
$$\psi_{lj}(t_1,\ldots ,t_n)\in \grO_{K/\dbQ}[\bft]\quad (1\le l\le n,\, 1\le j\le k)$$
is generated from $\bfvarphi$ by restriction of $L$ down to $K$. Finally, suppose that 
$1\le s\le k(k+1)/2$ and $\eps>0$. Then the number $J_s(X;\bfpsi;K)$ of solutions of the 
system
\begin{equation}\label{16.2}
\sum_{i=1}^s\left( \psi_{lj}(x_{i1},\ldots ,x_{in})-\psi_{lj}(y_{i1},\ldots ,y_{in})\right)=0
\quad (1\le l\le n, 1\le j\le k),
\end{equation}
with $\bfx,\bfy\in \calB(X)^{ns}$, satisfies 
$J_s(X;\bfpsi;K)\ll ({\rm card}(\calB(X)))^{ns+\eps}$.
\end{theorem}

\begin{proof} The system of equations (\ref{16.2}) over $\grO_{K/\dbQ}$ is satisfied if 
and only if the system
\begin{equation}\label{16.3}
\sum_{i=1}^s\left( \varphi_j(u_i)-\varphi_j(v_i)\right)=0\quad (1\le j\le k),
\end{equation}
is satisfied over $\grO_{L/K}$ by
$$u_i=x_{i1}\ome_1+\ldots +x_{in}\ome_n\quad \text{and}\quad 
v_i=y_{i1}\ome_1+\ldots +y_{in}\ome_n\quad (1\le i\le s).$$
By Corollary \ref{corollary15.2}, when $1\le s\le k(k+1)/2$ and $\eps>0$, the total number 
$T$ of solutions of (\ref{16.3}) with 
$$\bfu,\bfv\in\{ r_1\ome_1+\ldots +r_n\ome_n:\bfr\in \calB(X)^n\}^s$$
satisfies
$$T\ll X^{ns+\eps}\ll (\text{card}(\calB(X)))^{ns+\eps}.$$
The conclusion of the theorem follows.
\end{proof}

This theorem supplies infinitely many examples of multidimensional systems for which the 
main conjecture holds. We have merely to examine systems of polynomials generated from 
$(t,t^2,\ldots ,t^k)$ by restriction of one number field to a subfield (perhaps $\dbQ$). 
Other multidimensional conclusions can be found in \cite{PPW2013}, where the authors 
aimed for completely general (though non-optimal) results for arbitrary translation-dilation 
invariant systems in many variables. The most recent work based on decoupling in certain 
two dimensional problems may be found in \cite{BDG2017, Guo2017}.

\section{Vinogradov's mean value theorem in function fields} One of the key messages to 
be extracted from this memoir is that the nested efficient congruencing method is 
sufficiently robust that it may be employed in a myriad environments with minimal 
adjustment. In particular, in contrast with the $l^2$-decoupling method of \cite{BDG2016}, 
we require no multilinear Kakeya estimates that might be unavailable, or even inherently 
mysterious in nature, when contemplated in different settings. Nested efficient congruencing 
is a method likely to achieve success for analogues of Vinogradov's mean value theorem 
involving discrete sets of points in any Henselian field. All that one requires are appropriate 
characters with which to engineer certain orthogonality relations. One need not be 
constrained to non-archimedean environments, moreover, since with additional effort 
involved in handling non-ultrametric inequalities, the same arguments apply equally well in 
archimedean environments such as the real or complex numbers. Indeed, the use of real 
short intervals underpinned the original approach of Vinogradov \cite{Vin1935} (see also 
\cite{Hua1949b}, \cite[Chapter VI]{Tit1986}, and \cite{Ste2017} for recent developments 
associated with efficient congruencing). In this section we illustrate this robustness with one 
final example, namely that of function fields.\par

We explain the consequences of the work of \S\S3-12 in the most basic situation of a 
function field $\dbF_q(t)$ of characteristic $p>k$, where $k$ is the number of equations at 
hand. We emphasise that this is very far from the strongest type of result available in the 
function field setting, since the situation with small characteristic $p\le k$ is considerably 
more complicated. Comprehensive conclusions achieving the main conjecture in Vinogradov's 
mean value theorem are the subject of work in progress by the author joint with Y.-R.~Liu. 
In this section we avoid intruding on the latter work, and instead extract only the results 
that may be obtained with essentially no effort from our analysis  in \S\S3-12.\par

We begin by recalling the infrastructure required for harmonic analysis in the function field 
setting. Our coefficients come from the finite field $\dbF_q$ of characteristic $p$ having 
$q=p^l$ elements. Associated with the polynomial ring $\dbO=\dbF_q[t]$ defined over the 
field $\dbF_q$ is its field of fractions $\dbK=\dbF_q(t)$. In this section, we take $d$ to be 
the main parameter, a sufficiently large natural number, and we put
$$\dbO_d=\{ \nu \in \dbF_q[t]: \text{deg}(\nu)\le d\}.$$
We write $\dbK_\infty =\dbF_q((1/t))$ for the completion of $\dbF_q(t)$ at $\infty$. One 
may write each element $\alp \in \dbK_\infty $ in the shape $\alp =\sum_{i\le n}a_it^i$ for 
some $n\in \dbZ$ and coefficients $a_i=a_i(\alp )$ in $\dbF_q$ $(i\le n)$. We define 
$\text{ord }\alp $ to be the largest integer $i$ for which $a_i(\alp )\ne 0$. We then write 
$\langle \alp \rangle$ for $q^{\text{ord }\alp}$. In this context, we adopt the convention 
that $\text{ord }0=-\infty$ and $\langle 0\rangle =0$. Consider next the compact additive 
subgroup $\dbT $ of $\dbK_\infty $ defined by 
$\dbT =\{ \alp \in \dbK_\infty \,:\, \langle \alp \rangle <1\}$. Every element $\alp $ of 
$\dbK_\infty $ can be written uniquely in the shape $\alp =[\alp ]+\| \alp \|$, where 
$[\alp ]\in \dbF_q[t]$ and $\| \alp \| \in \dbT $, and we may normalise any Haar measure 
$\d\alp $ on $\dbK_\infty $ in such a manner that $\int_\dbT 1\d\alp =1$.\par

We are now equipped to define an analogue of the exponential function. There is a 
non-trivial additive character $e_q:\dbF_q\rightarrow \dbC^\times$ defined for each 
$a\in \dbF_q$ by taking $e_q(a)=e(\text{tr}(a)/p)$, where we write $e(z)$ for 
$e^{2\pi iz}$, and where $\text{tr}\!:\!\dbF_q\rightarrow \dbF_p$ denotes the familiar 
trace map. This character induces a map $e:\dbK_\infty \rightarrow \dbC^\times$ by 
defining, for each element $\alp \in \dbK_\infty$, the value of $e(\alp)$ to be 
$e_q(a_{-1}(\alp ))$. The orthogonality relation underlying the Fourier analysis of 
$\dbF_q[t]$, established for example in \cite[Lemma 1]{Kub1974}, takes the shape
\begin{equation}\label{17.1}
\int_\dbT e(h\alp)\,d\alp =\begin{cases} 0,&\text{when $h\in \dbF_q[t]\setminus 
\{ 0\}$,}\\1,&\text{when $h=0$}.\end{cases}
\end{equation}

\begin{theorem}\label{theorem17.1} Suppose that $\varphi_j\in \dbO[x]$ $(1\le j\le k)$ is a 
system of polynomials with $W(x;\bfvarphi)\ne 0$. Let $s$ be a positive real number with 
$s\le k(k+1)/2$. Also, suppose that $(\gra_\nu)_{\nu\in \dbO}$ is a sequence of complex 
numbers. Then provided that ${\rm ch}(\dbF_q)>k$ and $\eps>0$, one has
$$\int_{\dbT^k}\biggl| \sum_{\nu\in \dbO_d}\gra_\nu e(\alp_1\varphi_1(\nu)+\ldots 
+\alp_k\varphi_k(\nu))\biggr|^{2s}\d\bfalp \ll (q^d)^\eps \biggl( \sum_{\nu \in 
\dbO_d}|\gra_\nu|^2\biggr)^s.$$
In particular, one has
$$\int_{\dbT^k}\biggl| \sum_{\nu\in \dbO_d}e(\alp_1\varphi_1(\nu)+\ldots 
+\alp_k\varphi_k(\nu))\biggr|^{2s}\d\bfalp \ll (q^d)^{s+\eps }.$$
\end{theorem}

In this section, implicit constants in Vinogradov's notation may depend on $s$, $k$, 
$\bfvarphi$, $q$, and also the small positive number $\eps$. As an immediate consequence 
of the orthogonality relation (\ref{17.1}), one obtains the following corollary.

\begin{corollary}\label{corollary} When $s\in \dbN$, denote by $N_{s,\bfvarphi}(d,q)$ 
the number of solutions of the system of equations
$$\sum_{i=1}^s (\varphi_j(x_i)-\varphi_j(y_i))=0\quad (1\le j\le k),$$
with $x_i,y_i\in \dbO_d$ $(1\le i\le s)$. Suppose that ${\rm ch}(\dbF_q)>k$. Then 
whenever $W(x;\bfvarphi)\ne 0$ and $s\le k(k+1)/2$, one has 
$N_{s,\bfvarphi}(d,q)\ll (q^d)^{s+\eps}$.
\end{corollary}

Finally, when $k\in \dbN$ and $s>0$, we define
$$J_{s,k}(d;\dbF_q)=\int_{\dbT^k}\biggl| \sum_{\nu\in \dbO_d}e(\alp_1\nu+\ldots 
+\alp_k\nu^k)\biggr|^{2s}\d\bfalp .$$
Theorem \ref{theorem17.1} delivers an analogue of the main conjecture in Vinogradov's 
mean value theorem in function fields of large characteristic.

\begin{corollary}\label{corollary17.3}
Suppose that $k\in \dbN$ and $s>0$. Then whenever ${\rm ch}(\dbF_q)>k$ and $\eps>0$, 
one has
$$J_{s,k}(d;\dbF_q)\ll (q^d)^\eps \left( (q^d)^s+(q^d)^{2s-k(k+1)/2}\right) .$$
\end{corollary}

There is an unpublished manuscript, more than a decade old, of Y.-R. Liu and the author in 
the function field setting that has been quoted from time to time, and was updated to 
reflect developments arising from the early efficient congruencing methods. This was 
subsumed by the multidimensional work \cite{KLZ2014}, which would achieve an analogue 
of Corollary \ref{corollary17.3} for $s>k(k+1)$ when $\text{ch}(\dbF_q)>k$, with sharper 
conclusions available for $\text{ch}(\dbF_q)\le k$. Forthcoming work of Y.-R.~Liu and the 
author removes the hypothesis $\text{ch}(\dbF_q)>k$ from an analogue of Corollary 
\ref{corollary17.3} in which the main conjecture is proved in general for Vinogradov's mean 
value theorem in function fields. This forthcoming joint work represents the definitive 
statement on the subject.\par

The conclusion of Theorem \ref{theorem17.1} follows from an analogue of Theorem 
\ref{theorem3.1}, as we now sketch. The details are again strikingly similar to those in the 
situation over $\dbZ$ described in \S\S3--12, so we are skimpy on details. Once more, the 
parameter $k$ is an integer with $k\ge 1$, and we consider polynomials 
$\varphi_j\in \dbO[x]$ $(1\le j\le k)$. Throughout the argument, the ring of polynomials 
$\dbO$ replaces $\dbZ$. Let $\pi\in \dbO$ be a monic irreducible polynomial, and hence of 
positive degree. The definitions of \S3 must be made once again, mutatis mutandis, where 
we replace $\dbZ$ by $\dbO$, congruences modulo $p^h$ by congruences modulo $\pi^h$, 
and the function $e(z)$ from \S3 by its doppelg\"anger defined in this section, throughout. 
The definitions (\ref{1.9}) and (\ref{3.5}) must again be adjusted to the present setting. 
First, when $F:\dbT^k\rightarrow \dbC$ is integrable, we write
$$\oint F(\bfalp)\d\bfalp =\int_{\dbT^k}F(\bfalp)\d\bfalp .$$
Next, when $B$ is a positive integer, we define
$$\oint_{\pi^B}F(\bfalp)\d\bfalp =\langle \pi\rangle^{-kB}\sum_{u_1\mmod{\pi^B}}
\ldots \sum_{u_k\mmod{\pi^B}}F(\bfu \pi^{-B}),$$
where the summations are taken over complete sets of residues modulo $\pi^B$. With this 
definition, one may verify that for $\nu \in \dbO$, one has the orthogonality relation
$$\oint_{\pi^B}e(\alp \nu)\d\alp =\begin{cases} 1,&
\text{when $\nu\equiv 0\mmod{\pi^B}$,}\\
0,&\text{when $\nu\not\equiv 0\mmod{\pi^B}$.}\end{cases}$$

\par These definitions permit an analogue of Theorem \ref{theorem3.1} to be stated in the 
function field setting. We recall the definitions (\ref{3.6}) and (\ref{3.8}), now adjusted to 
their function field manifestations.

\begin{theorem}\label{theorem17.4} Let $\dbO=\dbF_q[t]$. Suppose that $k\in \dbN$, and 
that $\pi\in \dbO$ is a monic irreducible polynomial. Then, under the assumption that 
${\rm ch}(\dbF_q)>k$, one has $\lam(k(k+1)/2,k)=0$.
\end{theorem}

\begin{corollary}\label{corollary17.5} Let $\dbO=\dbF_q[t]$. Suppose that $k\in \dbN$, and 
that $\pi\in \dbO$ is a monic irreducible polynomial. Suppose also that $\tau>0$ and 
$\eps>0$. Let $B$ be sufficiently large in terms of $k$, $\tau$ and $\eps$. Put 
$s=k(k+1)/2$ and $H=\lceil B/k\rceil$. Finally, suppose that ${\rm ch}(\dbF_q)>k$. Then 
for every $\bfvarphi\in \Phi_\tau(B)$ and every sequence $(\gra_n)\in \dbD_0$, one has
$$U_{s,k}^B(\bfgra)\ll \langle \pi\rangle^{B\eps}U_{s,k}^{B,H}(\bfgra).$$
\end{corollary}

The proof of Theorem \ref{theorem17.4} and its corollary follow just as in the 
corresponding argument of \S\S3--10 above in the rational integer case. Given our 
adjustments in notation, the argument follows verbatim, provided that the reader exercises 
due diligence in ensuring that these notational perturbations are correctly construed. 
Perhaps it is worth noting that the restriction to situations subject to the condition 
$\text{ch}(\dbF_q)>k$ arises from the implied assumption that $W(x;\bfvarphi)\ne 0$. If 
the system $\bfvarphi$ is $\pi^c$-spaced for some positive integer $c$ and monic 
irreducible polynomial $\pi$, then the determinant $W(x;\bfvarphi)$ is congruent modulo 
$\pi^c$ to a triangular determinant with entries $j!$ $(1\le j\le k)$ along the diagonal. 
When $\text{ch}(\dbF_q)\le k$, it follows that one of these diagonal entries is zero modulo 
$\pi$, and hence one fails to be able to engineer a non-vanishing Wronskian modulo $\pi$. 
The solution to this problem is to make use of a basis for Taylor expansions in small positive 
characteristic smaller than the na\"ive basis $\{1,x,x^2,\ldots \}$ of consecutive powers. 
This then entails adjusting also the definition of the Wronskian for function fields accordingly, 
and in this setting the environment has changed sufficiently that other adjustments are 
required in the discussion of \S\S3-10. This is a matter to which we return in our 
forthcoming joint work with Y.-R. Liu.\par

Having left the reader with the mechanical reproduction of this verbatim proof, we move on 
immediately to the proof of Theorem \ref{theorem17.1}

\begin{proof}[The proof of Theorem \ref{theorem17.1}] Our argument follows the proof of 
Theorem \ref{theorem1.1} given in \S\S11 and 12 with very few adjustments. We again 
denote by $\calZ$ the set of zeros of $W(x;\bfvarphi)$ lying in $\dbO$. Since 
$W(x;\bfvarphi)$ is non-zero, it follows that
$$\text{card}(\calZ)\le \text{deg}(W(x;\bfvarphi))\ll 1.$$
We define
$$F(\bfalp;d)=\rho_0^{-1}\sum_{\nu \in \dbO_d}\gra_\nu e(\psi(\nu;\bfalp))$$
and
$$F_0(\bfalp;d)=\rho_0^{-1}\sum_{\nu \in \dbO_d\setminus \calZ}\gra_\nu 
e(\psi(\nu;\bfalp)).$$
We may suppose that the sequence $(\gra_\nu)$ satisfies the property that $\gra_\nu=0$ 
whenever $\nu\not\in \dbO_d$. Our task is to cover the exponential sum $F_0(\bfalp;d)$ by 
exponential sums with variables constrained by non-singularity conditions modulo $\pi$, for 
suitable monic irreducible polynomials $\pi\in \dbO$.\par

Given a solution $\bfx,\bfy\in (\dbO_d\setminus \calZ)^s$ of the system (\ref{1.4}), the 
polynomial
$$\prod_{i=1}^sW(x_i;\bfvarphi)W(y_i;\bfvarphi)$$
is non-zero, and has degree bounded above by $C+Dd$, for some $C>0$ depending at 
most on $s$ and the coefficients of $\bfvarphi$, and $D$ a positive integer with
$$D\le 2s\sum_{j=1}^k\text{deg}(\varphi_j).$$
Let $\calP$ denote the set of elements $\pi \in \dbO$ with $\pi$ monic and irreducible, and 
satisfying
$$(\log d)^2<\text{deg}(\pi)\le 3(\log d)^2.$$
By the analogue of the prime number theorem over $\dbF_q[t]$, when $d$ is sufficiently 
large, the number of such elements is at least $q^{2(\log d)^2}/(\log d)^2$. 
One therefore has
$$\text{deg}\biggl( \prod_{\pi\in \calP}\pi\biggr) =\sum_{\pi \in \calP}\text{deg}(\pi)
>q^{2(\log d)^2}>d^{\log d}.$$
Thus $\text{deg}\left( \prod_{\pi\in \calP}\pi\right) >C+Dd$. Then for each solution 
$\bfx, \bfy\in \dbO^s$ of (\ref{1.4}) counted by the mean value
$$\oint |F_0(\bfalp;d)|^{2s}\d\bfalp ,$$
there exists $\pi \in \calP$ with
$$\prod_{i=1}^s W(x_i;\bfvarphi)W(y_i;\bfvarphi)\not \equiv 0\mmod{\pi}.$$
In particular, one has
\begin{equation}\label{17.2}
\oint |F_0(\bfalp;d)|^{2s}\d\bfalp \le \sum_{\pi \in \calP}\oint 
|F_\pi (\bfalp;d)|^{2s}\d\bfalp ,
\end{equation}
where
$$F_\pi (\bfalp;d)=\rho_0^{-1}\sum_{\nu \in \dbO_d\setminus \calZ}|\gra_\nu|
e(\psi(\nu;\bfalp )).$$

\par Resuming the argument of \S12, we take $\tau>0$ to be sufficiently small in terms of 
$s$ and $k$, and put
$$B=\left\lceil k(d+1)/\text{deg}(\pi)\right\rceil ,\quad c=\lceil \tau B\rceil \quad 
\text{and}\quad H=\lceil B/k\rceil -c,$$
which ensures that
$$d+1\le \text{deg}\left( \pi^{H+c}\right) \le d+1+\text{deg}(\pi).$$
Then we may suppose that $B$ is sufficiently large in terms of $\tau$, as well as $s$, $k$ 
and $\eps$. By orthogonality, one has
$$\oint |F_\pi(\bfalp;d)|^{2s}\d\bfalp \le \oint_{\pi^B}|F_\pi(\bfalp;d)|^{2s}\d\bfalp .$$
Then, by Corollary \ref{corollary17.5}, just as in the argument of the proof of Theorem 
\ref{theorem11.1} leading to (\ref{11.9}), one sees that there is a $\pi^c$-spaced system 
$\bfPsi$ and complex sequence $\bfgrc$ with $|\grc_y|=|\gra_{\pi^cy+\xi}|$ for which one 
has
$$\oint_{\pi^B}|F_\pi(\bfalp;d)|^{2s}\d\bfalp \ll \langle \pi\rangle^{sc+B\eps}
\rho_0(\bfgra)^{-2}\sum_{\xi\nmod{\pi^c}}\rho_c(\xi)^2U_{s,k}^{B-kc,H,\bfPsi}(\bfgrc).
$$

\par Observe that when $\nu,\nu'\in \dbO_d$, then
$$\nu\equiv \nu'\mmod{\pi^{H+c}}\quad \implies \quad \nu=\nu'.$$
Thus, as in the concluding paragraph of the proof of Theorem \ref{theorem11.1}, one infers 
that
$$U_{s,k}^{B-kc,H,\bfPsi}(\bfgrc)\ll (1+q^d/\langle \pi\rangle^{H+c})^s\ll 1.$$
Then we deduce that
$$\oint_{\pi^B}|F_\pi(\bfalp;d)|^{2s}\d\bfalp \ll \langle \pi\rangle^{sc+B\eps}
\rho_0(\bfgra)^{-2}\sum_{\xi\mmod{\pi^c}}\rho_c(\xi)^2\ll \langle 
\pi\rangle^{(2s\tau+\eps)B}.$$
Recall that $\tau$ was chosen sufficiently small in terms of $s$ and $k$. Then it follows that 
for each positive number $\del$, one has
$$\oint |F_\pi(\bfalp;d)|^{2s}\d\bfalp \ll (q^d)^\del,$$
whence, on recalling (\ref{17.2}),
\begin{align*}
\oint |F(\bfalp;d)|^{2s}\d\bfalp &\ll 1+\oint |F_0(\bfalp;d)|^{2s}\d\bfalp \\
&\ll 1+(q^d)^\del\sum_{\pi \in \calP}1\ll (q^d)^{2\del}.
\end{align*}
The conclusion of the theorem follows on recalling the definitions of $F(\bfalp;d)$ and the 
weight $\rho_0$.
\end{proof}

\bibliographystyle{amsbracket}

\begin{thebibliography}{18}

\bibitem{ACHK2017}
T. C. Anderson, B. Cook, K. Hughes and A. Kumchev, \emph{Improved 
$\ell^p$-boundedness for integral $k$-spherical maximal functions}, preprint available as 
arXiv:1707.08667.

\bibitem{ACK2004}
G. I. Arkhipov, V. N. Chubarikov and A. A. Karatsuba, \emph{Trigonometric sums in number 
theory and analysis}, de Gruyter Expositions in Mathematics, \textbf{39}, Walter de Gruyter, 
Berlin, 2004.

\bibitem{Bak1986}
R. C. Baker, \emph{Diophantine inequalities}, London Mathematical Society Monographs, 
New Series, \textbf{1}, Oxford University Press, New York, 1986.

\bibitem{Bir1961}
B. J. Birch, \emph{Waring's problem in algebraic number fields}, Proc. Cambridge Philos. 
Soc. \textbf{57} (1961), 449--459.

\bibitem{BB2010}
V. Blomer and J. Br\"udern, \emph{The number of integer points on Vinogradov's quadric},  
Monatsh. Math. \textbf{160} (2010), no. 3, 243--256.

\bibitem{Bou2017} J. Bourgain, \emph{On the Vinogradov mean value}, Tr. Mat. Inst. 
Steklova \textbf{296} (2017), Analiticheskaya i Kombinatornaya Teoriya Chisel, 36--46.

\bibitem{BDG2017}
J. Bourgain, C. Demeter and S. Guo, \emph{Sharp bounds for the cubic Parsell-Vinogradov 
system in two dimensions}, preprint available as arXiv:1608.06346.

\bibitem{BDG2016}
J. Bourgain, C. Demeter and L. Guth, \emph{Proof of the main conjecture in Vinogradov's 
mean value theorem for degrees higher than three}, Ann. of Math. (2) \textbf{184} 
(2016), no. 2, 633--682.

\bibitem{BGR2008}
D. Brink, H. Godinho and P. H. A. Rodrigues, \emph{Simultaneous diagonal equations over
$\grp$-adic fields}, Acta Arith. \textbf{132} (2008), no. 4, 393--399.

\bibitem{Bru1988}
J. Br\"udern, \emph{A problem in additive number theory}, Math. Proc. Cambridge Philos. 
Soc. \textbf{103} (1988), no. 1, 27--33.

\bibitem{BR2015}
J. Br\"udern and O. Robert, \emph{Rational points on linear slices of diagonal 
hypersurfaces}, Nagoya Math. J. \textbf{218} (2015), 51--100.

\bibitem{CH2010} E. Croot and D. Hart, \emph{$h$-fold sums from a set with few 
products}, SIAM J. Discrete Math. \textbf{24} (2010), no. 2, 505--519.

\bibitem{Eda1967}
Y. Eda, \emph{On the mean-value theorem in an algebraic number field}, Japan. J. Math. 
\textbf{36} (1967), 5--21.

\bibitem{Eda1975}
Y. Eda, \emph{On Waring's problem in algebraic number fields}, Rev. Colombiana Mat. 
\textbf{9} (1975), no. 2, 29--73.

\bibitem{For2002}
K. B. Ford, \emph{Vinogradov's integral and bounds for the Riemann zeta function}, Proc. 
London Math. Soc. (3) \textbf{85} (2002), no. 3, 565--633.

\bibitem{FW2014}
K. B. Ford and T. D. Wooley, \emph{On Vinogradov's mean value theorem: strongly 
diagonal behaviour via efficient congruencing}, Acta Math. \textbf{213} (2014), no. 2, 
199--236.

\bibitem{Guo2017}
S. Guo, \emph{On a binary system of Prendiville: The cubic case}, preprint available as 
arXiv:1701.06732.

\bibitem{HB2015}
D. R. Heath-Brown, \emph{The cubic case of Vinogradov's mean value theorem --- a 
simplified approach to Wooley's ``efficient congruencing''}, available as arXiv:1512.03272.

\bibitem{Hua1938b}
L.-K. Hua, \emph{On Waring's problem}, Quart. J. Math. Oxford \textbf{9} (1938), 
199--202.

\bibitem{Hua1938}
L.-K. Hua, \emph{On Tarry's problem}, Quart. J. Math. Oxford \textbf{9} (1938), 315--320.

\bibitem{Hua1949}
L.-K. Hua, \emph{Improvement of a result of Wright}, J. London Math. Soc. \textbf{24} 
(1949), 157--159.

\bibitem{Hua1949b}
L.-K. Hua, \emph{An improvement of Vinogradov’s mean-value theorem and several 
applications}, Quart. J. Math. Oxford \textbf{20} (1949), 48--61.

\bibitem{Hua1965}
L.-K. Hua, \emph{Additive theory of prime numbers}, American Math. Soc., Providence, RI, 
1965.

\bibitem{Kor1962}
O. K\"orner, \emph{\"Uber Mittelwerte trigonometrischer Summen und ihre Anwendung in 
algebraischen Zahlk\"orpern}, Math. Ann. \textbf{147} (1962), 205--239.

\bibitem{Koz2001}
I. M. Kozlov, \emph{The mean value theorem of I. M. Vinogradov for Gaussian numbers}, 
Proceedings of the IV International Conference ``Modern Problems of Number Theory and 
its Applications'' (Tula, 2001), Chebyshevski$\breve{\text{\i}}$ Sb. \textbf{1} (2001), 
25--39.

\bibitem{Kub1974}
R. M. Kubota, \emph{Waring's problem for $\dbF_q[x]$}, Dissertationes Math. (Rozprawy 
Mat.) \textbf{117} (1974), 60pp.

\bibitem{KLZ2014}
W. Kuo, Y.-R. Liu and X. Zhao, \emph{Multidimensional Vinogradov-type estimates in function 
fields}, Canad. J. Math. \textbf{66} (2014), no. 4, 844--873.

\bibitem{Lan1903}
E. Landau, \emph{Neuer Beweis des Primzahlsatzes und Beweis des Primidealsatzes}, 
Math. Ann. \textbf{56} (1903), no. 4, 645--670.

\bibitem{PPW2013}
S. T. Parsell, S. M. Prendiville and T. D. Wooley, \emph{Near-optimal mean value estimates 
for multidimensional Weyl sums}, Geom. Funct. Anal. \textbf{23} (2013), no. 6, 1962--2024.

\bibitem{PW2002}
S. T. Parsell and T. D. Wooley, \emph{A quasi-paucity problem}, Michigan Math. J.
\textbf{50} (2002), no. 3, 461--469.

\bibitem{Pec1949}
L. G. Peck, \emph{Diophantine equations in algebraic number fields}, Amer. J. Math. 
\textbf{71} (1949), 387--402.

\bibitem{Pie2017}
L. Pierce, \emph{The Vinogradov mean value theorem (after Wooley, and Bourgain, 
Demeter and Guth)}, preprint available as arXiv:1707.00119.

\bibitem{Rog1986}
Rogovskaya, N. N., \emph{An asymptotic formula for the number of solutions of a system 
of equations}, Diophantine Approximations, Part II, Moskov. Gos. Univ., Moscow, 1986, pp. 
78--84.

\bibitem{SW2010}
P. Salberger and T. D. Wooley, \emph{Rational points on complete intersections of higher 
degree, and mean values of Weyl sums}, J. London Math. Soc. (2) \textbf{82} (2010), 
no. 2, 317--342.

\bibitem{Sor2007}
P. N. Sorokin, \emph{The mean value theorem of I. M. Vinogradov for a trigonometric sum 
in Gaussian numbers}, Vestnik Moscov. Univ. Ser. I Mat. Mekh. (2007), no. 6, 63--65.

\bibitem{Ste2017}
R. S. Steiner, \emph{Effective Vinogradov's mean value theorem via efficient boxing},  
preprint available as arXiv:1603.02536v2.

\bibitem{Tit1986}
E. C. Titchmarsh, \emph{The theory of the Riemann zeta-function}, Second edition, Edited 
and with a preface by D. R. Heath-Brown, The Clarendon Press, Oxford University Press, 
New York, 1986.

\bibitem{Vau1997}
R. C. Vaughan, \emph{The Hardy-Littlewood method}, Cambridge University Press, 
Cambridge, 1997.

\bibitem{VW1997}
R. C. Vaughan and T. D. Wooley, \emph{A special case of Vinogradov's mean value 
theorem}, Acta Arith. \textbf{79} (1997), no. 3, 193--204.

\bibitem{Vin1935}
I. M. Vinogradov, \emph{New estimates for Weyl sums}, Dokl. Akad. Nauk SSSR \textbf{8} 
(1935), 195--198.

\bibitem{Vin1947}
I. M. Vinogradov, \emph{The method of trigonometrical sums in the theory of numbers}, 
Trav. Inst. Math. Stekloff \textbf{23} (1947), 109pp.

\bibitem{Wan1991}
Y. Wang, \emph{Diophantine equations and inequalities in algebraic number fields}, 
Springer-Verlag, Berlin, 1991.

\bibitem{Woo1992}
T. D. Wooley, \emph{On Vinogradov's mean value theorem}, Mathematika \textbf{39} 
(1992), no. 2, 379--399.

\bibitem{Woo1993}
T. D. Wooley, \emph{A note on symmetric diagonal equations}, Number theory with an
emphasis on the Markoff spectrum (Provo, UT, 1991), Editors: A. D. Pollington and W.
Moran, Dekker, New York, 1993, pp. 317--321.

\bibitem{Woo1996}
T. D. Wooley, \emph{A note on simultaneous congruences}, J. Number Theory \textbf{58} 
(1996), no. 2, 288--297.

\bibitem{Woo1998}
T. D. Wooley, \emph{On the local solubility of Diophantine systems}, Compositio Math. 
\textbf{111} (1998), no. 2, 149--165.

\bibitem{Woo2012}
T. D. Wooley, \emph{Vinogradov's mean value theorem via efficient congruencing}, Ann. 
of Math. (2) \textbf{175} (2012), no. 3, 1575--1627.

\bibitem{Woo2012b}
T. D. Wooley, \emph{The asymptotic formula in Waring's problem}, Internat. Math. Res. 
Notices \textbf{2012} (2012), no. 7, 1485--1504.

\bibitem{Woo2013}
T. D. Wooley, \emph{Vinogradov's mean value theorem via efficient congruencing, II}, 
Duke Math. J. \textbf{162} (2013), no. 4, 673--730.

\bibitem{Woo2014a}
T. D. Wooley, \emph{Translation invariance, exponential sums, and Waring's problem}, 
Proceedings of the International Congress of Mathematicians, August 13--21, 2014, Seoul, 
Korea, Volume II, Kyung Moon Sa Co. Ltd., Seoul, Korea, 2014, pp. 505--529.

\bibitem{Woo2015d}
T. D. Wooley, \emph{Mean value estimates for odd cubic Weyl sums}, Bull. Lond. Math. Soc. 
\textbf{47} (2015), no. 6, 946--957.

\bibitem{Woo2015b}
T. D. Wooley, \emph{Rational solutions of pairs of diagonal equations, one cubic and one 
quadratic}, Proc. London Math. Soc. (3) \textbf{110} (2015), no. 2, 325--356.

\bibitem{Woo2015c}
T. D. Wooley, \emph{Corrigendum: ``The asymptotic formula in Waring's problem''}, 
Internat. Math. Res. Notices \textbf{2015} (2015), no. 20, 10702.

\bibitem{Woo2015}
T. D. Wooley, \emph{Multigrade efficient congruencing and Vinogradov's mean value 
theorem}, Proc. London Math. Soc. (3) \textbf{111} (2015), no. 3, 519--560.

\bibitem{Woo2016}
T. D. Wooley, \emph{The cubic case of the main conjecture in Vinogradov's mean value 
theorem}, Adv. Math. \textbf{294} (2016), 532--561.

\bibitem{Woo2016a}
T. D. Wooley, \emph{Perturbations of Weyl sums}, Internat. Math. Res. Notices 
\textbf{2016} (2016), no. 9, 2632--2646.

\bibitem{Woo2016b}
T. D. Wooley, \emph{Solvable points on smooth projective varieties}, Monatsh. Math. 
\textbf{180} (2016), no. 2, 391--403.

\bibitem{Woo2017}
T. D. Wooley, \emph{Approximating the main conjecture in Vinogradov's mean value 
theorem}, Mathematika \textbf{63} (2017), no. 1, 292--350.

\bibitem{Woo2017a}
T. D. Wooley, \emph{Discrete Fourier restriction via efficient congruencing}, Internat. Math. 
Res. Notices \textbf{2017} (2017), no. 5, 1342--1389.

\bibitem{Wri1948}
E. M. Wright, \emph{The Prouhet-Lehmer problem}, J. London Math. Soc. \textbf{23} 
(1948), 279--285.

\end{thebibliography}
\providecommand{\bysame}{\leavevmode\hbox to3em{\hrulefill}\thinspace}

\end{document}